\documentclass[11pt]{article}
\usepackage{amsmath}
\usepackage{amsfonts}
\usepackage{amsthm}
\usepackage{amssymb}
\usepackage{latexsym}
\usepackage[pdf]{diagrams}

\def\emp{\emptyset}
\def\Xg{X_{\kern-0.07em{g}}}
\def\Win{W^{\kern+0.03em{0}}}			
\def\Wni{W^{\kern+0.03em{00}}}			
\def\T{\mathbf{T}}
\def\TH{\T_{H}}
\def\A{\mathbf{A}}
\def\RG{R_{H}}
\def\tw{\widetilde}
\def\Hid{\mathbf{H}}
\def\p{\mathfrak{p}}
\def\Sym{\mathrm{Sym}}
\def\Laa{\mathcal{S}}
\def\xbar{x'}
\def\ybar{y'}
\def\H{\mathcal{H}}
\def\n{\mathfrak{N}}
\def\GSp{\mathrm{GSp}}
\def\Spec{\mathrm{Spec}}
\def\Res{\mathrm{Res}}
\def\Ind{\mathrm{Ind}}
\def\Ker{\mathrm{Ker}}
\def\Gal{\mathrm{Gal}}
\def\Frob{\mathrm{Frob}}
\def\GL{\mathrm{GL}}
\def\PGL{\mathrm{PGL}}
\def\Hom{\mathrm{Hom}}
\def\End{\mathrm{End}}
\def\Tang{\mathrm{Tan}}
\def\Aut{\mathrm{Aut}}
\def\Ind{\mathrm{Ind}}
\def\Inn{\mathrm{Inn}}
\def\Out{\mathrm{Out}}
\def\Proj{\mathrm{Proj}}
\def\Ad{\mathrm{ad}}
\def\loc{\mathrm{Local}}
\def\glob{\mathrm{Global}}
\def\Sf{\mathfrak{S}}
\def\Z{\mathbf{Z}}
\def\D{\mathcal D}
\def\Ddet{\D^{\mathrm{det}}}
\def\Q{\mathbf{Q}}
\def\Qbar{\overline{\Q}}
\def\Kbar{\overline{K}}
\def\Gal{\mathrm{Gal}}
\def\eps{\epsilon}
\def\G{\mathcal{G}}
\def\R{\bf{R}}
\def\Rrhobar{R}

\def\OL{\mathcal{O}}

\def\La{\Lambda}
\def\Lo{\mathcal{L}}
\def\m{\mathfrak{m}}
\def\Cat{\mathcal{C}}
\def\Prho{P\kern-0.1em{\rho}}
\def\Wg{W_g}
\def\Hc{\langle H,c \rangle}
\def\Hcb{H^1_{\mathrm{cusp}}}
\def\Hcc{H^1}
\def\even{even}			     
\def\HS{H^1_{\Sigma, (p)}(K,W)}     
\def\HN{H^1_{\Sigma,f}(K,W)}         
\def\Yin{Y^0}			     
\def\Yni{Y^{00}}		     
\def\Yiota{Y^{\iota = - 1}}
\def\Vdag{V_{\lambda}}
\def\Fg{\widetilde{F}}
\def\epsg{\epsilon_{K,g}}
\def\twphi{\widetilde{\phi}}
\def\twpsi{\widetilde{\psi}}

\def\twPsi{\widetilde{\Psi}}
\def\im{\mathrm{Im}}
\def\Remark{\noindent \bf{Remark}\rm:\ }
\def\Remarks{\noindent \bf{Remarks}\rm:\ }
\newarrow{Equals}{=}{=}{=}{=}{=}
\newarrow{EqualsC}{.}{.}{.}{.}{>>}
\newarrow{ToC}{.}{.}{.}{.}{>}
\newcounter{Lcount}
\def\rhou{\rho^{univ}_{\Sigma}}
\def\r{\wp}
\def\Hom{\mathrm{Hom}}
\def\K{\mathbf{K}}
\def\L{\mathbf{L}}
\def\SL{\mathrm{SL}}
\def\q{\mathfrak{q}}
\def\ad{\mathrm{ad}^0(\rhobar)}
\def\F{\mathbf{F}}
\def\C{\mathbf{C}}
\def\N{\mathcal{N}}

\def\rhobar{\overline{\rho}}
\def\eps{\epsilon}
\def\Tang{\mathrm{Tan}}
\def\PTang{P\mathrm{Tan}}
\def\IR{\mathrm{Irr}}

\DeclareFontEncoding{OT2}{}{} 
\newcommand{\textcyr}[1]{
{\fontencoding{OT2}\fontfamily{wncyr}\fontseries{m}\fontshape{n}
\selectfont #1}}
\newcommand{\Sha}{{\mbox{\textcyr{Sh}}}}

\newtheorem{theorem}{Theorem}[section]
\newtheorem{corollary}[theorem]{Corollary}
\newtheorem{definition}[theorem]{Definition}

\newtheorem{conjecture}[theorem]{Conjecture}
\newtheorem{example}{Example}
\newtheorem{hypothesis}{Hypothesis}
\newtheorem{lemma}[theorem]{Lemma}
\newtheorem{proposition}[theorem]{Proposition}
\newtheorem{sublemma}{Sublemma}

\begin{document}

\title{Nearly Ordinary Galois Deformations over Arbitrary Number Fields}
\author{Frank Calegari\footnote{Supported in part by
the American Institute of Mathematics, and the National Science Foundation} \and Barry Mazur\footnote{Supported in part by
the National Science Foundation}}
\maketitle

{\small
\abstract{Let $K$ be an arbitrary number field, and let $\rho: \Gal(\Kbar/K) 
\rightarrow GL_2(E)$
 be a nearly ordinary irreducible geometric Galois representation. In this paper, we study the nearly ordinary deformations of $\rho$. When $K$ is totally real and $\rho$ is modular, results of Hida imply that the nearly ordinary deformation space associated to $\rho$ contains a Zariski dense set of points corresponding to``automorphic'' Galois representations. We conjecture that if $K$ is
 \emph{not} totally real, then this is never the case, except in three exceptional cases, corresponding to: 
 $(1)$ ``base change'', $(2)$ ``CM'' forms, and $(3)$ ``even'' representations. The latter case conjecturally can only occur if the image of $\rho$ is finite. Our results come in two flavours. First, we prove a general result for Artin representations, conditional on a strengthening of Leopoldt's conjecture. Second, when 
$K$ is an imaginary quadratic field, we prove an unconditional result that implies the existence of ``many'' positive dimensional components (of certain deformation spaces) that do not contain infinitely many classical points. Also included are some speculative remarks about ``$p$-adic functoriality'', as well as some remarks on how our methods should apply to $n$-dimensional representations of 
$\Gal(\Qbar/\Q)$ when $n > 2$.}}

\section{Introduction}
\label{section:introduction}
Let $K$ be a number field,     $\Kbar/K$ be an algebraic closure of $K$,
and $p$ a prime that splits completely in $K$.
Fix $$\rho: \Gal(\Kbar/K) \rightarrow \GL_2(E),$$  a continuous 
``nearly ordinary" (see Definition~\ref{no}  below) absolutely irreducible $p$-adic
Galois representation unramified outside a finite set of places of $K$, where $E$ is a field extension of $\Q_p$.

Choose a lattice for the representation $\rho$, and let
$\rhobar: \Gal(\Kbar/K) \rightarrow \GL_2(\OL_E/\m_E)$ denote the associated
residual representation.  Fixing $S$ a finite set of places of $K$ that include the primes dividing $p$, the primes of ramification for $\rho$, and all archimedean places, we may view $\rhobar$ as homomorphism
$$\rhobar: G_{K,S} \rightarrow \GL_2(\OL_E/\m_E)$$ where $G_{K,S}$ is the quotient group of  $\Gal(\Kbar/K)$ obtained by division by the closed normal subgroup generated by the inertia groups at all primes of $K$ not in $S$.

 Let us suppose, for now, that $\rhobar$ is
absolutely irreducible. Then we may form
 $\Rrhobar$, the universal deformation ring  
of  continuous nearly ordinary deformations of the Galois representation 
$\rhobar$ (unramified outside $S$); this is a complete noetherian local ring that comes along with a {\it universal} nearly ordinary  
deformation  of our representation $\rhobar$,
$$\rho^{\rm univ}: G_{K,S} \rightarrow \GL_2(\Rrhobar),$$
from which we may recover $\rho$ as a specialization at some
$E$-valued point of $\Spec(\Rrhobar)$.  For reference, see section 30 of  \cite{Mazur1}.

Suppose  that
 $K$ is totally real, 
 $\rho$ is attached to a Hilbert modular form $\pi$ of regular weight, 
$p$ splits completely in $K$ and  $\rho|_{D_{v}}$ is nearly ordinary
for all $v|p$. In this situation, a
construction of Hida~\cite{Hidareal,Hidareal2}  gives us a (``Hida-Hecke"-algebra) 
quotient $\Rrhobar \to \T$ that is a finite flat algebra over the ($n$-parameter) 
affine coordinate ring of weights. Consequently, we have a 
corresponding ``$n$-parameter" family of nearly ordinary Galois 
representations which project onto a component of ($p$-adic) weight space, 
and which also contain
a Zariski dense set of classical (automorphic) representations.

  The situation is quite different in the case where $K$ is not totally real. 
Although our interest in this paper is mainly limited to the study
of Galois representations rather than automorphic forms, to
 cite one example before giving the main the theorem of this article, we have:
  
\begin{theorem} Let $K = \Q(\sqrt{-2})$, and let
$\n = 3 - 2 \sqrt{-2}$. Let  $\mathbf{T}$ denote the nearly
ordinary $3$-adic Hida algebra of tame level $\n$. Then the
affine scheme $\mathrm{Spec}(\mathbf{T})$  only contains finitely 
many classical points and moreover $($is nonempty and$)$ has 
pure relative dimension one over $\mathrm{Spec}(\Z_p)$.
\label{theorem:hidafamilies}
\end{theorem}
  This is proved in subsection~\ref{section:auto} below.   
Although our general aim  is to understand deformation spaces of 
nearly ordinary Galois representations, our primary focus in 
this article is the close study of  {\it first order} deformations of a 
given irreducible representation $\rho$; that is, we will be considering the tangent 
space of the space of  (nearly ordinary) deformations 
of $\rho$, ${\Tang}_{\rho}(\Rrhobar)$.  For most of this article 
we will focus on the case where  $\rho$ factors through a finite 
quotient of $\Gal(\Kbar/K)$ --- we call such representations 
{\it Artin representations} --- and where the prime $p$ splits 
completely in $K$. If $\m$ denotes the prime ideal of $R$
corresponding to the representation $\rho$,
we will be looking at the linear mapping of $E$-vector 
spaces $$\tau: {\Tang}_{\m}(\Rrhobar) \longrightarrow {\Tang}_{0}(W),$$ 
the canonical projection  to the tangent space (at weight $0$) of 
the space of $p$-adic weights.  The projective space ${\PTang}_{0}(W)$ 
of lines in this latter $E$-vector space, ${\Tang}_{0}(W)$, has a natural 
descent to  $\Q$. A (nonzero) infinitesimal  weight $w \in {\PTang}_{0}(W)$ 
that generates a line that is $\Q$-rational  with respect to this 
underlying $\Q$-structure will be called a {\bf classical infinitesimal  
weight}. Similarly,   a (first-order) deformation of $\rho$ possessing 
a classical infinitesimal  weight  will be called 
an {\bf infinitesimally classical deformation}. (We restrict our
 attention to deformations of fixed determinant to avoid 
``trivial'' instances of infinitesimally classical deformations arising from
twisting by characters.)
We feel 
that (in the case where $K$ is not totally real) the mere occurrence of 
an infinitesimally classical deformation implies very strong consequences 
about the initial representation $\rho$. Our main result is that this is 
indeed so if we assume a hypothesis that we call the {\it Strong Leopoldt Conjecture}.  
Specifically, we prove the
following:

\begin{theorem} Assume the Strong Leopoldt Conjecture. 
Let $p$ be prime, and let  $K/\Q$ be a Galois extension in which $p$ splits completely.
Suppose that $\rho: \Gal(\Kbar/K) \rightarrow \GL_2(E)$ is continuous, irreducible,
 nearly ordinary,  has finite image, and admits
infinitesimally classical deformations. Then either:
  \begin{enumerate}
  
  \item 
There exists a character $\chi$ such that
 $\chi \otimes \rho$ descends either to an odd representation over a  totally real 
field, or descends to a field
containing at least one real place at which $\chi \otimes \rho$
is even.
\item  
The projective image of $\rho$ is dihedral, and
  the determinant character
descends to a totally real field $H^{+} \subseteq K$ with corresponding fixed field $H$
such that
  \begin{list}{\kern-0.2em{$($}\roman{Lcount}\kern+0.1em$)$}
    {\usecounter{Lcount}
    \setlength{\rightmargin}{\leftmargin}}
  \item $H/H^{+}$ is a CM extension.
  \item At least one prime above $p$ in $H^{+}$ splits in $H$.
  \end{list}
\end{enumerate}

\label{theorem:mainintro}
\end{theorem}

\Remarks
 
   \begin{enumerate}
\item
The restriction that $K/\Q$ be Galois is very mild since one is free
to  make a base change  as long as $\rho$ remains irreducible.

   \item
The \emph{strong Leopoldt conjecture} (\S \ref{section:theleopoldtconjecture})
is, as we hope to convince our readers, a natural
generalization of the usual Leopoldt conjecture, and is related
to other $p$-adic transcendence conjectures such as the
$p$-adic form of Schanuel's
conjecture.
   \item 
 We expect that the same conclusion holds  more generally for $p$-adic 
representations $\rho: \Gal(\Kbar/K) \rightarrow \GL_2(E)$ that  do 
not have finite image, but otherwise satisfy the conditions of 
Theorem~\ref{theorem:mainintro}. Explicitly,
\end{enumerate}

\begin{conjecture}\label{conj:conjecture} Let $p$ be prime, and 
let  $K/\Q$ be a Galois extension in which $p$ splits completely.
Suppose that $\rho: \Gal(\Kbar/K) \rightarrow \GL_2(E)$ is 
continuous, irreducible,
 nearly ordinary, is unramified except at finitely many places, and 
admits
infinitesimally classical deformations. Then either:
  \begin{enumerate}
  
  \item 
There exists a character $\chi$ such that
 $\chi \otimes \rho$ descends either to an odd representation over a  totally real 
field, or descends to a field
containing at least one real place at which $\chi \otimes \rho$
is even.
\item  The projective image of $\rho$ is  dihedral, and
  the determinant character
descends to a totally real field $H^{+} \subseteq K$ with corresponding fixed field $H$
such that
  \begin{list}{\kern-0.2em{$($}\roman{Lcount}\kern+0.1em$)$}
    {\usecounter{Lcount}
    \setlength{\rightmargin}{\leftmargin}}
  \item $H/H^{+}$ is a CM extension.
  \item At least one prime above $p$ in $H^{+}$ splits in $H$.
  \end{list}
\end{enumerate}

\end{conjecture}

 For some evidence for this, see Theorem~\ref{theorem:raviintro} below.
We refer to the first case as ``base change'' and  the second case as ``CM.''

  In the special case where $K$ is a quadratic imaginary field, the $E$-vector space 
of infinitesimal weights,  ${\Tang}_{0}(W)$, is then two-dimensional, 
and the nontrivial involution (i.e., complex conjugation) of $\Gal(K/\Q)$ 
acting on ${\Tang}_{0}(W)$ has unique  eigenvectors (up to scalar multiplication) 
with eigenvalues $+1$ and $-1$; refer to the points they give rise 
to in ${\PTang}_{0}W$ as the {\it diagonal} infinitesimal weight 
and the {\it anti-diagonal}, respectively.

\begin{corollary}\label{corquad} Assume the Strong Leopoldt Conjecture. 
Let  $K$ be a quadratic imaginary field and $p$ a prime that splits in $K$. 
Let $\rho: \Gal(\Kbar/K) \rightarrow \GL_2(E)$ satisfy the hypotheses 
of \ref{theorem:mainintro}, and in particular assume that $\rho$  admits an
infinitesimally classical deformation.  Assume that the image  of $\rho$ 
is finite and non-dihedral. Then one of the two cases below holds:

\begin{enumerate}
\item  The representation $\rho$  admits an
infinitesimally classical deformation with { diagonal} infinitesimal 
weight and  there exists a character $\chi$ such that
 $\chi \otimes \rho$ descends to an odd representation over $\Q$, or
 
 \item  the representation $\rho$  admits an
infinitesimally classical deformation with {anti-diagonal} infinitesimal 
weight and  there exists a character $\chi$ such that
 $\chi \otimes \rho$ descends to an even representation over $\Q$.
\end{enumerate}
\label{corollary:spec}
\end{corollary}

 The above Corollary~\ref{corquad} follows from Theorem~\ref{theorem:mainintro} together
 with the discussion of section~\ref{antidiagwt}.

Our result (Theorem~\ref{theorem:mainintro}), although comprehensive, 
is contingent
upon
 a conjecture, and is only formulated for representations 
with finite image. To offer another angle on the study of 
deformation spaces (of nearly ordinary Galois representations) 
containing few classical automorphic forms, we  start with a {\it residual} 
representation $\rhobar: \Gal(\Kbar/K) \rightarrow \GL_2(\F)$ 
where $\F$ is a finite field of characteristic $p$, and  $\rhobar$ 
is  irreducible  and nearly ordinary
at $v|p$ (and satisfies some supplementary technical hypotheses).
In  section~\ref{section:ordinaryfamilieswithnonparallelweights}, we prove the following 
complement to
Theorem~\ref{theorem:mainintro} when $K$ is an imaginary quadratic field:

\begin{theorem}
Let $K$ be an imaginary quadratic field in which $p$ splits.
Let $\F$ be a finite field of characteristic $p$.
Let $\rhobar: \Gal(\Kbar/K) \rightarrow \GL_2(\F)$ be a continuous
irreducible Galois
representation, let $\chi: \Gal(\Kbar/K) \rightarrow \F^{\times}$
be the mod-$p$ cyclotomic character, and
assume the following:
\begin{enumerate}
\item The image of $\rhobar$ contains $\SL_2(\F)$, and
$p \ge 5$.
\item If $v|p$, then $\rhobar$ is nearly ordinary at $v$
and takes the shape:
$\displaystyle{\rho|_{I_v}
= \left( \begin{matrix} \psi & * \\ 0 & 1 \end{matrix}
\right)}$,
 where $\psi \ne 1,\chi^{-1}$, and $*$ is tr\`{e}s ramifi\'{e}e if
$\psi = \chi$.
\item If $v \nmid p$, and $\rhobar$ is ramified at $v$, then
$H^2(G_v,\ad) = 0$.
\end{enumerate}
Then there exists a Galois representation:
$\displaystyle{\rho: \Gal(\Qbar/K) \rightarrow \GL_2(W(\F)[[T]])}$
lifting $\rhobar$ such that:
\begin{enumerate}
\item  The image of $\rho$ contains $\SL_2(W(\F)[[T]])$.
\item $\rho$ is unramified outside some finite set of primes $\Sigma$.
If $v \in \Sigma$, and $v \nmid p$, then $\rho|D_v$ is
potentially semistable; if $v|p$, then $\rho|D_v$ is nearly ordinary.
\item
Only finitely many specializations of $\rho$ have parallel weight.
\end{enumerate}
\label{theorem:raviintro}
\end{theorem}

\medskip

The relation between Theorem~\ref{theorem:raviintro} and
Conjecture~\ref{conj:conjecture} is the following.
Standard facts (\cite{Harder}, \S III, p.59--75)
imply that the cuspidal
cohomology of $\GL(2)_{/K}$ vanishes for non-parallel weights.
Thus,  if a one parameter family of nearly ordinary deformations of $\rhobar$
arise from classical cuspidal automorphic forms over $K$, they
\emph{a priori} must have all have parallel weight. Since families of
parallel weight have rational infinitesimal Hodge-Tate weights,
 Conjecture~\ref{conj:conjecture} implies that such families
are CM or arise from base change (and hence only exist if $\rhobar$
is also of this form). In contrast, Theorem~\ref{theorem:raviintro} at least
guarantees the \emph{existence} of non-parallel families for \emph{any}
$\rhobar$ (including those arising from base change).

\medskip

Although we mostly restrict our attention to the study of Galois
representations, the general philosophy of $R = \mathbf{T}$
theorems predicts that the ordinary Galois deformation rings
considered in this paper are isomorphic to the ordinary Hecke
algebras constructed by Hida in~\cite{HidaSL}. Thus
we are inclined to believe the corresponding automorphic
conjectural analogs of Conjecture~\ref{conj:conjecture} and
Theorem~\ref{theorem:raviintro}.  

\medskip

We now give a brief outline of this paper.
In section~\ref{section:deformationtheory}, we recall
some basics from the deformation theory of Galois representations,
and define the notion of infinitesimal Hodge-Tate weights.
In section~\ref{section:theleopoldtconjecture} we present
our generalization of the Leopoldt conjecture. In
section~\ref{section:descendinggrouprepresentations} we prove
some lemmas regarding group representations, and
in sections~\ref{section:artinrepresentations} 
and~\ref{section:ICD}, we prove
Theorem~\ref{theorem:mainintro}.
We prove Theorem~\ref{theorem:raviintro} in
section~\ref{section:ordinaryfamilieswithnonparallelweights}.
In section~\ref{section:extras}, we speculate further on
the connection between nearly ordinary Galois representations,
Hida's ordinary Hecke algebra, and classical automorphic forms.
We also describe some other problems which can
be approached by our methods, for example: establishing the rigidity of
 three dimensional Artin representations
of $\Gal(\Qbar/\Q)$. In section~\ref{section:extras}
we also give a proof of Theorem~\ref{theorem:hidafamilies},
and in the Appendix (section~\ref{section:appendix})  we prove some structural results
about generic modules used in sections~\ref{section:theleopoldtconjecture}
and~\ref{section:artinrepresentations}.

\subsection{Speculations on $p$-adic functoriality}

It has long been suspected that torsion classes
arising from cohomology
of arithmetic quotients of symmetric spaces
give rise to Galois representations, even when the associated
arithmetic quotients are not Shimura varieties.  In the
context of $\GL(2)_{/K}$ for an imaginary quadratic field $K$,
some isolated examples for $K = \Q(\sqrt{-1})$ were first noted
in~\cite{Germans2}. The first
conjectures for arbitrary reductive groups $G$ (and coefficient systems over a finite field)
were formulated by Ash~\cite{Ash}, at least for $G = \GL(n)_{/\Q}$.
Let $G_{/\Q}$ denote an arbitrary reductive group that is split at $p$,
$K_{\infty}$  a maximal compact of $G(\R)$, and
 $T_G$  a maximal torus of $G$.
After fixing a tame level, a construction of Hida assembles 
(ordinary) torsion classes (relative to some choice of local system)
for arithmetic quotients
$G(\Q) \backslash G(\A) / K^{\circ}_{\infty} K$, as $K$ varies over
$p$-power levels relative to the fixed tame level. If $\T_G$ denotes
the corresponding algebra of endomorphisms generated by Hecke operators
on this space (not to be confused with the torus, $T_G$), then $\T_G$ is finitely generated over the space
of weights
$\Lambda_G = \Z_p[[T_G(\Z_p)]]$.
Hida's control theorems~\cite{Hida}\footnote{Which presumably hold
in full generality, and have been established (by Hida himself)
in many cases,
$G = \GL(2)_{/K}$, $\GSp(4)_{/\Q}$ ---
the examples that arise in the discussion --- included.}
 imply that for (sufficiently)
 regular weights $\kappa$ in
$\Hom(\Lambda,\C_p)$, the specialization of $\T_G$ to $\kappa$ recovers
the classical space of ordinary automorphic forms of that weight.
It should be noted, however, that Corollary~\ref{corollary:spec} implies
that $\T_G$ can be  large even if the specialization to every
regular weight of $\Lambda_G$ is zero, and that $\T_G$  should
be thought of as an object arising from Betti cohomology (taking torsion cohomology classes into consideration as well) rather than anything
with a more standard automorphic interpretation, such as $(\mathfrak{g},K)$ cohomology.
Suitably adapted and generalized, the
conjectures alluded to above imply the existence of
a map $R_G \rightarrow \T_G$, where $R_G$ denotes an appropriate (naturally defined)
Galois deformation ring.
One of our guiding principles is that the map $R_G \rightarrow
\T_G$---for suitable definitions of $R_G$ and $\T_G$---should {\it exist} and be an isomorphism.
If $G = \GL(2)_{/\Q}$, then this conjecture has (almost) been
completely established, and many weaker results are also known 
for $G = \GL(2)_{/K}$ if $K$ is totally real
(such results are usually conditional, for example, on the modularity
of an associated residual representation $\rhobar$).

At first blush, our
main results appear to be negative,  namely,
they suggest that the ring $\T_G$ --- constructed from 
Betti cohomology --- has very little to do with automorphic
forms. 
We would like, however, to suggest a second possibility.
Let us define a (nearly ordinary)
emph{$p$-adic automorphic form} for $G$ to
be a characteristic zero point of $\Spec(\T_G)$. 
 Then
we conjecture that $p$-adic automorphic forms 
are subject to the same principle of functoriality as
espoused by Langlands for classical automorphic forms.
Such a statement is (at the moment) imprecise and vague ---
even to begin to formalize this principle would probably
require a more developed theory of $p$-adic local
Langlands than currently exists. And yet, in
the case of $\GL(2)_{/K}$  for  an imaginary quadratic field $K$, we
may  infer two  consequences of these 
conjectures which are quite concrete, and whose resolutions
will probably hold significant clues as to  what methods will
apply more generally.

The first example  concerns the possibility of
$p$-adic automorphic lifts. Fix an imaginary quadratic
field $K$. Let $G = \GL(2)_{/K}$
and $H = \GSp(4)_{/\Q}$.  Given a $p$-adic 
automorphic form $\pi$ for $G$ 
arising from a characteristic
zero point of $\Spec(\T_G)$, 
a $p$-adic principle of functoriality would predict the
following:
if $\pi$  satisfies a suitable
condition on the central character, it lifts to a $p$-adic
automorphic form for $H$.  The existence of such a lift is
most easily predicted from the Galois representation  $\rho$
(conjecturally) associated to $\pi$. Namely, if 
$\rho: \Gal(\Kbar/K) \rightarrow \GL_2(E)$ is a representation
with determinant $\chi$, and $\chi = \phi \psi^2$ for characters
$\phi$, $\psi$ of $K$, where $\phi$ lifts to a character of $\Gal(\Qbar/\Q)$,
then $\Ind^{\Gal(\Qbar/\Q)}_{\Gal(\Kbar/K)} 
(\rho \otimes \psi^{-1})$ has image landing inside
$\GSp_4(E)$ for some suitable choice of symplectic form, 
and will  be the Galois representation associated
to the conjectural $p$-adic lift of $\pi$ to $H$.
If the original $p$-adic automorphic form $\pi$ is actually
\emph{classical},  then this lifting is known to exist by work
of Harris, Soudry, and Taylor~\cite{taylorothers}.
The group $H = \GSp(4)_{/\Q}$ has the convenient property
that the corresponding arithmetic quotients are PEL Shimura
varieties. Consequently, there is a natural source of  Galois representations,
and, more relevantly for our purposes,
the Hida families associated to cuspidal ordinary forms
of middle degree ($3$, in this case) cohomology form a Hida family $\TH$
which is \emph{flat} over $\Lambda_H$. 
Yoneda's lemma implies the existence of maps between various
rings: from
 Galois induction, we obtain a surjective map
$R_H \rightarrow R_G$, and from $p$-adic functoriality, we predict
the existence of a surjective map $\T_H \rightarrow \T_G$.
All together we have a  diagram:
$$\begin{diagram}
R_H & \rTo & \T_H \\
\dOnto &  & \dEqualsC \\
R_G & \rToC & \T_G, \\
\end{diagram}$$
where only the solid arrows are known to exist. Naturally, we conjecture
that the horizontal arrows are isomorphisms, and that the diagram is
commutative.
 A calculation on tangent spaces suggests
that the relative dimensions of $\RG$ and $R_G$ over $\Z_p$ are
$2$ and $1$ respectively (factoring out twists), and thus
$\Spec(\T_G)$ should ``sit" as a \emph{divisor} inside 
$\Spec(\TH)$. 

If this picture is correct, we would be facing the following curious situation,
in light of  Corollary~\ref{corollary:spec}:
the space $\Spec(\TH)$ has a dense set of classical automorphic
points, and yet the divisor $\Spec(\T_G)$ contains components
with only finitely many classical points. Thus, although the
points of $\Spec(\T_G)$ are not approximated by classical
automorphic forms for $G$ (they cannot be, given the paucity of such forms), 
they \emph{are}, in some sense,
approximated by classical (indeed, holomorphic) automorphic forms for $H$. Such
a description of $\Spec(\T_G)$ would allow one to study classical
points of $\Spec(\T_G)$ (for example, automorphic forms
$\pi$ associated to modular elliptic curves over $K$) by variational
techniques.

Given the above speculation, there may be several approaches to defining  $p$-adic $L$-functions (of different sorts) on $\Spec(\T_G)$.  The first approach might follow the lead of the classical construction using modular symbols, i.e., 
one-dimensional homology of the appropriate 
symmetric spaces, but would make use of all the Betti homology including torsion classes.  The second approach might lead us to to a $p$-adic analogue of the ``classical" (degree $4$)  Asai
$L$-function  by constructing $p$-adic $L$-functions  on $\Spec(\TH)$ (banking on the type of {\it functoriality} conjectured above, and using the density of classical
points on $\Spec(\TH)$) and then restricting.

 One may also ask for a characterization
of the (conjectured) divisor given by the functorial lifting of $\Spec(\T_G)$ to $\Spec(\TH)$. 
For a \emph{classical} automorphic form $\pi$ on $H$, the
Tannakian formalism of $L$-groups (or, perhaps equally relevantly, 
a theorem of Kudla, Rallis, and Soudry~\cite{Kudla}) implies that $\pi$ arises from $G$
if and only if the degree $5$ (standard) $L$-function 
$L(\pi \otimes \chi,s)$ has a pole at $s = 1$. Hence, speculatively,
we conjecture that the divisor $\Spec(\T_G)$ in $\Spec(\TH)$ is cut out by the
polar divisor of a $p$-adic (two variable, and as yet only
conjectural) degree-$5$ $L$-function
on $\Spec(\TH)$ at $s = 1$. Moreover, the residue at this pole
will be the (conjectural) special value of the
Asai $L$-function interpolated along $\Spec(\T_G)$.

A second prediction following from $p$-adic functoriality
 is as follows. Once more, let
$G =\GL(2)_{/K}$, where $K$ is an imaginary quadratic
field, and now let $G'$ denote an inner form of $G$, specifically, arising from
a quaternion algebra over $K$. Then results of Hida~\cite{Hidaimag}
 provide us with \emph{two}
ordinary Hecke algebras $\T_G$ and $\T_{G'}$ which are one dimensional
over $\Z_p$. The theorem of Jacquet-Langlands implies that any classical
point of $\Spec(\T_{G'})$ corresponds (in the usual sense) to a classical point of
$\Spec(\T_G)$. If classical points were dense in $\Spec(\T_{G'})$, then
one would expect that this correspondence would force 
$\Spec(\T_{G'})$ to be identified with a subvariety of $\Spec(\T_{G})$.
In certain situations this expectation  can be translated
into a formal argument, and
one obtains a ``$p$-adic Jacquet-Langlands'' for $p$-adic
automorphic forms  as a consequence of
the classical theorem. This is because, in situations arising from
Shimura varieties, eigencurves are  uniquely determined by
their classical points, and are thus  rigid enough to be 
``independent of their constructor''~\cite{Chenevier}, \S 7. 
However, for our $G$ and $G'$, Corollary~\ref{corollary:spec}
exactly predicts
 that classical points will \emph{not} be dense, and hence 
no obvious relation between $\Spec(\T_G)$ and $\Spec(\T_{G'})$ follows
from classical methods.
Instead, we predict the existence of a genuine \emph{$p$-adic}
Jacquet-Langlands for $\GL(2)$ which does \emph{not} follow from the usual
classical theorem, nor (in any obvious way) from any general classical
form of functoriality, and that, by keeping track of the action of Hecke,
 this correspondence will identify
$\Spec(\T_{G'})$ as a subvariety of $\Spec(\T_{G})$.

It should be noted that the phenomena of non-classical points can even be seen
even in $\GL(1)_{/K}$, when $K$ is neither totally real nor CM. Namely,
there exist Galois representations $\Gal(\Kbar/K) \rightarrow E^{*}$ that
cannot be approximated by Galois representations arising from
algebraic Hecke characters.
For example, suppose the signature of
$K$ is $(r_1,r_2)$ with $r_2 > 0$, and that $K$ does not contain a CM field.
Then the only Galois representations unramified away from
$p$ arising from algebraic Hecke characters are
(up to a bounded finite character)
some integral power of the cyclotomic character. The closure of
such representations in deformation space has dimension one.
On the other hand, the dimension of the space of one dimensional
representations of $\Gal(\Kbar/K)$ unramified away from $p$ has
dimension \emph{at least} $r_2 + 1$,  equality  corresponding 
to the Leopoldt conjecture. (See also section~\ref{subsection:dimone}.)
One reason for
studying
$G = \GL(2)_{/K}$, however, is that
when $K$ is a CM field, $G$ is related via functorial maps
to groups associated to
Shimura varieties, such as $\GSp(4)_{/K^+}$ or $U(2,2)_{/K}$. Even if the general
conjectural landscape of arithmetic cohomology beyond Shimura varieties requires
fundamentally new ideas, some progress should be possible with the classical
tools available today if one is allowed some access to algebraic geometry.
The other simplest case to consider is to try to
 produce liftings from $p$-adic automorphic forms on $\GL(1)_{/L}$ to $\GL(2)_{/K}$,
 where
$K$ is an imaginary quadratic field, and $L/K$ a quadratic extension.
To put oneself into  a genuinely $p$-adic situation,
 one should insist that $L$ is not a CM field,
equivalently, that $L/\Q$ is not biquadratic.

Due to 
the lack----at present----of 
substantial progress on these conjectures,  we shall
limit 
discussion of these matters, in this paper at least,
 to this section.  However, we believe that it is useful to view the
results of this paper in the light of these speculations, as they are
a first step to acknowledging that
a  ``$p$-adic'' approach
will be necessary to studying  automorphic forms over $\GL(2)_{/K}$.

\medskip

\subsection{Acknowledgments}     

During the course of working on this paper, the authors have
discussed the circle of ideas (regarding the
conjectural landscape of Galois representations beyond Shimura varieties)
with many people; in particular, we  would like to thank                
Avner Ash, Matthew Emerton,
Michael Harris, Haruzo Hida, Ravi Ramakrishna, Dinakar Ramakrishnan, Damien Roy, Chris Skinner, 
Warren Sinnott,
Glenn Stevens, John Tate, Richard Taylor, 
Jacques Tilouine, 
Eric Urban, and Michel Waldschmidt for helpful remarks. We also  thank
 David Pollack
and William Stein for 
the data in section~\ref{section:auto}.

\newpage

\tableofcontents

\normalsize

\section{Deformation Theory}
\label{section:deformationtheory}

\subsection{Nearly Ordinary Selmer Modules}\label{ordsel}


Fix a prime number $p$, and let $(A,\m_A,k)$ be a complete local noetherian (separated) topological ring.
 We further suppose that the residue field $k$ is either

 \begin{itemize} 
\item an algebraic extension of ${\bf F}_p$, in which case $A$ is endowed with its standard $\m$-adic topology, or 
\item a field of characteristic zero complete with respect to a $p$-adic valuation 
 (such a field we will simply call a {\it $p$-adic field}, and also denote by $E$) in which case we require that the maximal ideal be closed and that the reduction homomorphism $A\to k$ identifies the topological field $A/m$ with $k$.
 \end{itemize}
If $M$ is a free $A$-module of rank $d$, $\Aut_A(M)\simeq \GL_d(A)$ is canonically endowed with the structure of topological group. If $\ell$ is a  prime number different from $p$, neither of the topological groups $M$ or $\Aut_A(M)$ contain  infinite closed  pro-$\ell$ subgroups. 
 
\begin{definition}\label{line}
  Let $V_A$ be a free rank two $A$-module. By a {\bf line} $L_A \subset V_A$ one means a   free rank one $A$-submodule  for which there is a complementary free rank one $A$-submodule ${\tilde L}_A\subset V_A$, so that $V_A = L_A \oplus {\tilde L}_A$.
  \end{definition}

Let $K/\Q$ be an algebraic  number field,
and assume that $p$  splits completely in $K$.

\begin{definition}\label{no}
Let $V_A$ be a free rank two $A$-module with
a continuous $A$-linear action of $G_K:=\Gal(\Kbar/K)$, and let
$$\rho_A: \Gal(\Kbar/K) \rightarrow \Aut_A(V_A) \simeq \GL_2(A)$$
be the associated Galois representation.
Such a representation $\rho$ is said to be \emph{nearly ordinary} 
at a prime $v|p$ 
if there is a choice of decomposition 
group  at $v$,  $\Gal(\Kbar_v/K_v) \simeq D_v \hookrightarrow G_K$, 
and a line  $L_{A,v} \subset V_A$ 
stabilized by the action (under $\rho_A$) of   $D_v$. 
If $D_v$ stabilizes more than one line, we will suppose that 
we have chosen one such line $L_{A,v}$,
which we refer to as the {\bf special line} at $v$.
\end{definition}

\medskip

\Remark 
The line $L_{A,v}$ will depend on the choice of $D_v\subset G_K$ but
 if one  choice of decomposition group $D_v$ at $v$ 
stabilizes a line $L_{A,v}$ then any other choice $D_v'$  
is conjugate to $D_v$, i.e., $D_v'=gD_vg^{-1}$ for 
some  $g \in G_K$ and therefore will stabilize the transported line $L_{A,v}':=gL_{A,v}$.  Let $I_v \subset D_v$ denote the inertia subgroup of the chosen decomposition group $D_v \subset G_K$ at the prime $v$.

\medskip

If $\rho_A$ is nearly ordinary 
at  $v$, then  there we have the decomposition
$$0 \rightarrow L_{A,v} \rightarrow V_A \rightarrow V/L_{A,v}
\rightarrow 0$$
of $A[D_v]$-modules. 
  Suppose we express a $G_K$-representation $\rho_A$ 
that is nearly ordinary at $v$ as a matrix representation 
(with respect to a basis of $V_A$ where the first basis 
element lies in the special line).
Then,
 the representation $\rho$ restricted to $D_v$ takes the shape:
$$\rho_A |_{D_v} = \left( \begin{matrix} \psi_{A,v} & * \\ 0 & \psi'_{A,v}
\end{matrix} \right)$$
where $\psi_{A,v}:D_v \to A^*$ is the character of $D_v$ given by its action on $L_{A,v}$,
 and $\psi_{A,v}'$ the character of $D_v$    given by its action on $V_A/L_{A,v}$. 
 
 \medskip
 
 Let us denote the reductions of the characters $\psi_{A,v},\psi_{A,v}'$ mod $\m_A$
 by $\psi_v,\psi'_v$.
  
 \begin{definition}
We say that the representation $\rho_A$ is {\bf distinguished} at $v$ 
if $\rho_A$ is nearly ordinary at $v$,
and  the pair of residual characters $\psi_v, \psi'_v:
D_v \to k^*$ are \emph{distinct}.
 If $\rho$ is  distinguished  at $v$  there are  at 
most two $D_v$-stable lines to choose as our special ``$L_{A,v}$", and 
if, in addition, $V_A$ is indecomposable as $D_v$-representation, 
there is only one such choice.

 \end{definition}
 
\bigskip

  {\bf Remark: }Let $A\to A'$ be a continuous homomorphism (and not necessarily local homomorphism) of the type of rings we are considering. If $V_A$ is a free rank two $A$-module endowed with continuous $G_K$-action, and if we set $V_{A'}:=V_A\otimes_AA'$ with induced $G_K$-action, then if $V_A$ is nearly ordinary (resp. distinguished) at $v$ so is $V_{A'}$. If $V_A$ is nearly ordinary at all places of $K$ dividing $p$. with $L_{A,v}$ its special line at $v$, and we take $L_{A',v}:=L_{A,v}\otimes_AA' \subset V_{A'}$ to be the special line at $v$ of $V_{A'}$ for all $ v \ | \ p$, we say that the natural homomorphism $V_A \to V_{A'}$ is a {\bf morphism of nearly ordinary $G_K$-representations}. In the special case where $A'$ is the residue field $k$ of $A$ and $A\to k$ is the natural projection, suppose  we are given $V$ a nearly ordinary, continuous, $k$-representation of $G_K$. Then  the isomorphism class of the pair $(V_A, V_k {\stackrel{\nu}{\to}} V)$ where $V_A$ is a nearly ordinary $A[G_K]$-representation and $\nu: V_k \simeq V$ is an isomorphism of nearly ordinary $k[G_K]$ representations, is called a {\bf deformation of the nearly ordinary $k[G_K]$-representation $V$ to $A$}.

We now consider the adjoint
representation associated to $\rho_A$; that is, 
the action of $G_K$ on  $\Hom_A(V_A,V_A)$ is 
given by $g(h)(x)= g(h(g^{-1}x))$ for $g \in G_K$ and $h \in \Hom_A(V_A,V_A)$.

 If  $W_A:=\Hom_A'(V_A,V_A) \subset \Hom_A(V_A,V_A)$ denotes the hyperplane 
in $\Hom_A(V_A,V_A)$ consisting of endomorphisms of $V_A$ of trace zero  
(equivalently, endomorphisms whose matrix expression in terms 
of any basis has trace zero),  then  $W_A$ is a $G_K$-stable free $A$-module of rank three, for which the submodule of scalar 
endomorphisms $A\subset\Hom_A(V_A,V_A)$ 
is a $G_K$-stable complementary subspace. 

\bigskip

If $v$ is a finite place of $K$ that is (nearly ordinary, and) distinguished for $\rho_A$,
 consider the following $D_v$-stable flag of free $A$-submodules of $W_A$  (each possessing a free $A$-module complement in $W_A$):
$$0 \subset \Wni_{A,v} \subset \Win_{A,v} \subset W_A,$$
where
\begin{itemize}
\item
$\Win_{A,v}:=  \Ker(\Hom_A'(V_A,V_A) \longrightarrow \Hom_A(L_{A,v},V_A/L_{A,v})),$ or, equivalently,
 $\Win_{A,v} \subset W_A$ is the $A$-submodule of endomorphisms of trace 
zero that bring the line $L_{A,v}$ into itself. We have the exact sequence of $A$-modules $$(*)\ \ \ \ 0 \to W_{A,v}^0 \to W_A \to W_A/W_{A,v}^0 \to 0$$ and note that since we are in a {\it distinguished} situation, $W_A/W_{A,v}^0 \simeq \Hom_A(L_{A,v},V_A/L_{A,v})$ is a $D_v$-representation over $A$ of dimension one given by the character ${\psi}_{A,v}'{\psi}_{A,v}^{-1}:D_v \to A^*$ whose associated residual character is nontrivial.
\item
$\Wni_{A,v}:= \Ker(\Win_{A,v} \longrightarrow \Hom_A(L_{A,v},V_A)),$ or equivalently,
 $\Wni_{A,v} := \Hom_A(V_A/L_{A,v}, L_{A,v}) \subset W_A$.  The subspace  $\Wni_{A,v}\subset W_A$ 
can also be described as the subspace of endomorphisms  that send the 
line $L_{A,v}$ to zero and are of trace zero. In the case when $A=k$ is a field, we may describe $\Wni_{k,v}$ as the $k$-vector subspace
of $\Win_{k,v}$ consisting
of nilpotent endomorphisms.
\end{itemize}

Sometimes we write $W_A = W_{A,v}$, when we wish to emphasize that we are considering the $A[G]$-module
$W_A$ as a $A[D_v]$-module.
The natural homomorphism $\Win_{A,v}\ {\stackrel{\epsilon}{\to}}\ \Hom_A(L_{A,v},L_{A,v}) = 
\End_A(L_{A,v}) = A$ fits in an exact sequence of 
$D_v$-modules

$$0 \to \Wni_{A,v}\  \to \ \Win_{A,v} \ \  {\stackrel{\epsilon}{\to}}\ \   A \to 0.$$
Choosing a basis of $V_A$ where the first basis element lies in $L_{A,v}$, we 
may (in the usual way)  identify $\Hom_A(V_A,V_A)$ with the space of $2\times 2$ 
matrices with entries in $A$. This identifies $\Win_{A,v}$ with the space of 
upper triangular matrices of trace zero, and the subspace $\Wni_{A,v}$   with matrices of the 
form $\left( \begin{matrix} 0 & * \\ 0 & 0
\end{matrix} \right)$.

In what follows below we shall be dealing with group cohomology with coefficients in modules of finite type over our topological $\Z_p$-algebras $A$; these can be computed using continuous co-chains.

\bigskip

Fix $\rho_A:G_K \to \Aut_A(V_A)$ a continuous $G_K$-representation that is \emph{nearly ordinary} and \emph{distinguished}
at all primes $v|p$. 

\bigskip

Let $\Sigma$ denote a
(possibly empty) set of primes in $K$, not containing
any prime above $p$.
For each place $v$ in $K$, we
define a  local Selmer condition
$\Lo_{A,v} \subseteq
H^1(D_v,W_{A,v})$
as follows:
\begin{enumerate}
\item If $v \notin \Sigma$ and ${v\nmid p}$, then
$\Lo_{A,v} = $ the image of $H^1(D_v/I_v,W_{A,v}^{I_v})$  in $H^1(D_v,W_{A,v})$ 
under the natural inflation homomorphism.
\item If $v \in \Sigma$, then 
$\Lo_{A,v} = H^1(D_v,W_{A,v})$.
\item If $v|p$, then let
$\Lo_{A,v} = \mathrm{Ker}(H^1(D_v,W_{A,v}) \rightarrow
H^1(I_v,W_{A,v}/\Win_{A,v}))$.
\end{enumerate}

\begin{lemma}\label{dvlemma}  Let $v$ be a prime of $K$. The following relations hold amongst local cohomology groups.
 \begin{enumerate}

\item $H^1\big(D_v/I_v, (W_{A,v}/\Win_{A,v})^{I_v}\big) = 0.$
\item The natural mapping $H^1(D_v,\Win_{A,v})\to H^1(D_v,W_{A,v})$ 
is an injection.
\item If $v|p$,  the image $H^1(D_v,\Win_{A,v})\hookrightarrow H^1(D_v,W_{A,v})$  is equal to the submodule $$\Lo_{A,v} =  \mathrm{Ker}(H^1(D_v,W_{A,v}) \to H^1(I_v,W_{A,v}/\Win_{A,v}))\subset H^1(D_v,W_{A,v}).$$
\end{enumerate}
 \end{lemma}
 
 \begin{proof}  We proceed as follows.
  \begin{enumerate}

\item  Either $(W_{A,v}/\Win_{A,v})^{I_v}=0$ in which case (1) follows, or else  (since $v$ is distinguished)  $D_v/I_v$ acts on $(W_{A,v}/\Win_{A,v})^{I_v}$ via a (residually) nontrivial character. Since $D_v/I_v \simeq {\hat {\bf Z}}$ it follows that $H^1(D_v/I_v, W_{A,v}/\Win_{A,v})=0$, giving (1).
\item Since $\rho_A$ is distinguished at $v$, 
the action of $D_v$ on the line $W_{A,v}/\Win_{A,v}$ is via a character that is residually nontrivial, and so 
 $H^0(D_v, W_{A,v}/\Win_{A,v})=0$. Therefore, by the long exact cohomological sequence associated to $(*)$ above, we have the equality
 $$H^1(D_v, \Win_{A,v} ) =  \mathrm{Ker}(H^1(D_v,W_{A,v}) \rightarrow
H^1(D_v,W_{A,v}/\Win_{A,v}))\subset H^1(D_v,W_{A,v})$$  establishing (2).
\item   The Hochschild-Serre Spectral Sequence applied to the pair of groups $I_v \lhd D_v$ as acting on $W_{A,v}/\Win_{A,v}$ gives that $$H^1\big(D_v/I_v, (W_{A,v}/\Win_{A,v})^{I_v}\big)\to H^1(D_v,W_{A,v}/\Win_{A,v}) \to H^1(I_v,W_{A,v}/\Win_{A,v})$$ is exact; (3) then follows from (1) and (2).

\end{enumerate}
\end{proof}

  If $c \in H^1(K,W_A)$,
let $c_v \in H^1(D_v,W_{A,v})$ denote is image under the restriction homomorphism to $D_v$.

\begin{definition}
\label{definition:selmer}
The $\Sigma$-Selmer $A$-module $H^1_{\Sigma}(K,W_A)$
consists of classes $c \in H^1(K,W_A)$
such that $c_v \in \Lo_{A,v}$ for each finite prime $v$. 
We denote this by $H^1_{\emp}(K,W_A)$ when $\Sigma$ is the empty set.
\end{definition}

{\bf Remarks.}  From the viewpoint of universal (or versal) deformation spaces, the  Selmer $A$-module $H^1_{\Sigma}(K,W_A)$ (taken as coherent ${\mathcal O}_{\Spec(A)}$-module) is the normal bundle of the $A$-valued section determined by $\rho_A$ in the appropriate versal deformation space.

\bigskip

  From now on in  this article (outside of 
  section~\ref{section:ordinaryfamilieswithnonparallelweights})
  we shall
  be exclusively interested  in Selmer modules when $A=E$ is a $p$-adic field. In this
  setting we shall omit the subscript $A$, a change that is recorded by the following
  definition.
  
 \begin{definition} If $A = E$ is a $p$-adic field, we write 
 $W$, $\Win$, $\Wni$ for $W_E$, $\Win_E$, $\Wni_E$, and
  $\Lo_v$ for $\Lo_{E,v}$. \end{definition}

\begin{lemma}\label{inertsigma} 
If ${v\nmid p}$, and if the action of inertia at $v$ on the $E$-module  $W_{v}$ factors through a finite group  (e.g., if $v$ is unramified), then $\Lo_{v} = H^1(D_v,W_{v})$.
\end{lemma}

\begin{proof}  We must show that if ${v\nmid p}$ , $H^1(D_v/I_v,W_{v}^{I_v})\to H^1(D_v,W_{v})$  is surjective. It suffices, then, to show that $H^1(I_v,W_{v})=0$. Let $N_v \lhd I_v$ denote the kernel of the (continuous) action of $I_v$ on $W_v$. We have by our hypothesis that $I_v/N_v$ is a finite cyclic group. 
The Hochschild-Serre Spectral Sequence applied to $N_v \lhd I_v$ acting on $W_{v}$ yields the exact sequence
$$H^1(I_v/N_v, W_{v}) \to H^1(I_v, W_{v}) \to H^1(N_v, W_{v}).$$  
Since  $I_v/N_v$ is finite and $W_{v}$ is a vector space over a field of 
characteristic zero, the left flanking group vanishes. 
Since the $N_v$-action on $W_{v}$ is trivial, we have 
that  $H^1(N_v, W_{v})=\Hom_{\rm cont}(N_v^{\rm ab}, W_{v})$. 
Since $N_v^{\rm ab}$ is a finite group times a pro-${\ell}$-group 
for $\ell \ne p$ it follows that $\Hom_{\rm cont}(N_v^{\rm ab}, W_{v})$ vanishes.
\end{proof}

\begin{proposition} 
If for all $v {\not |}p$ the inertia group at $v$ acts on $V$ 
through a finite group, then the Selmer $E$-vector space  $H^1_{\Sigma}(K,W)$ 
is independent of $\Sigma$. In particular, for any $\Sigma$ (as defined above), 
we have a canonical isomorphism of $E$-vector spaces  
$H^1_{\Sigma}(K,W) \simeq H^1_{\emp}(K,W)$.
\end{proposition}

\begin{proof}  By Lemma~\ref{inertsigma} for such primes $v {\not |}p$, i.e., when the inertia group at $v$ acts on $V_E$  through a finite group, it follows that $\Lo_{v} = H^1(D_v,W_{v})$ whether or not $v$ is in $\Sigma$.
\end{proof}

\begin{corollary}  When 
the action of $G_K$ on $V$ factors through a finite group, for any $\Sigma$ we have 
$$H^1_{\Sigma}(K,W) \simeq H^1_{\emp}(K,W).$$
\label{corollary:minimallevel}
\end{corollary}

\subsection{Universal deformations and Selmer groups}

Let $V = V_E$ in the notation of subsection~\ref{ordsel} above, and let $\rho=\rho_E: G_{K}\to \Aut_E(V)$ be a continuous $E$-linear representation unramified outside a finite set of places of $K$ and nearly ordinary and distinguished at all $v$ dividing $p$. Let $\Sigma$ be a (possibly empty) set of primes in $K$ as in subsection~\ref{ordsel}.
Set  $S$  to be the (finite) set consisting of all places $v$ of $K$ that are either in $\Sigma$, or divide $p$, or are archimedean, or are ramified for $\rho$.  
We may view $\rho$ as a (continuous) representation
 $$\rho: G_{K,S} \longrightarrow \Aut_E(V) \approx \GL_2(E),$$ where $G_{K,S}$ is the quotient group of $G_K$ obtained by division by the closed normal subgroup generated by the inertia groups of primes of $K$ not in $S$.  
 Suppose, now, that $\rho$ is absolutely irreducible and let ${\Rrhobar}(S)={\Rrhobar}$ denote the ring that is the universal solution to the problem of deforming $\rho$ to two-dimensional nearly ordinary $G_{K,S}$-representations over complete local noetherian topological algebras $A$ with residue field the $p$-adic field $E$.  Thus we have
  $$\rho^{\rm univ}: G_{K,S} \longrightarrow  \GL_2({\Rrhobar}),$$ a universal representation (nearly ordinary, and unramified outside $S$) with residual representation equal to $\rho$. For this, see \S 30 of \cite{Mazur1} but to use this reference, the reader should note that the case of finite residue fields rather that $p$-adic fields such as $E$ were explicitly dealt with there, and also although a variety of local conditions are discussed, including the condition of {\it ordinariness} at a prime $v$, the local condition of {\it nearly ordinariness} is not treated, so we must adapt the proofs given there to extend to $p$-adic residue fields such as $E$, and also to note that with little change one may obtain  same (pro-)representability conclusion---as obtained there for ordinariness---for { nearly ordinariness}. All this is routine. 
  If $\m$ denotes the maximal ideal of ${\Rrhobar}$, then ${\Rrhobar}/\m = E$, and the canonical map
  ${\Rrhobar} \rightarrow {\Rrhobar}/\m = E$ induces (from $\rho^{\rm univ}$) the representation~$\rho$.

  
  \medskip
  
  \Remark Let  $ {\mathcal O}_E \subset E$ be the ring of integers in the complete valued field $E$, and suppose that $\rho: G_K \to \GL_2(E)$ takes its values in $\GL_2({\mathcal O}_E)$. Put  $\F:= {\mathcal O}_E/\m_E{\mathcal O}_E$ and suppose that the residual representation
  $\rhobar: G_{K,S} \rightarrow \GL_2(\F)$  associated to $\rho$ is \emph{irreducible}, (or, more generally, satisfies
  $\mathrm{End}_{\F}(\rhobar) = \F$). Suppose, furthermore, that $\rhobar|_{D_v}$ is nearly ordinary and
  {distinguished} for each $v|p$. Then $\rhobar$ also admits a 
  universal nearly ordinary deformation ring $R(\rhobar)$. 
  Consider the map
  $R(\rhobar) \rightarrow {\mathcal O}_E$ corresponding to the representation
  $\rho$, and denote the kernel of this homomorphism by $P$.
  We may
  recover ${R}$ from $R(\rhobar)$ as follows.
  Let
  ${\widehat{R(\rhobar)}}$ denote the completion of $R(\rhobar)[1/p]$ at $P$.
  Let $E'$ denote the residue field of this completion ---  the field $E'$ is the smallest
  $p$-adic subfield of $E$ such that the image of $\rho$ in $\GL_2(E)$ 
  may be conjugated   into $\GL_2(E')$.
  Then there is an isomorphism
${R} \simeq  \widehat{R(\rhobar)} \otimes_{E'} E$
  (see~\cite{Kisin}, Prop 9.5).

  \medskip

  For each prime $v$ of $K$ not dividing $p$, and choice of inertia  subgroup $I_v \subset D_v \subset \Gal({\bar K}/K)$ denote by $$I_{v,\rho} \subset I_v$$ the kernel of the restriction of $\rho$ to $I_v$. Thus if $v$ is a prime unramified for $\rho$ and not in $S$, we have that $I_{v,\rho} = I_v$ is also in the kernel of the universal representation $\rho^{\rm univ}$. 
  
   \begin{definition} Say that a deformation of $\rho$ to a complete local ring $A$ with residue field $E$, $\rho_A:  G_{K,S} \longrightarrow  \GL_2(A),$ has {\bf restricted ramification} at a prime $v$ of $K$ if $I_{v,\rho}$ is in the kernel of  $\rho_A$.
   \end{definition} 
   
   Let $\Rrhobar_{\Sigma}$ denote the maximal quotient ring of $\Rrhobar(S)$ with the property that the representation $$\rho^{\Sigma}: G_{K,S} \longrightarrow  \GL_2(\Rrhobar_{\Sigma})$$ induced from $\rho^{\rm univ}$ via the natural projection has the property that it has restricted ramification at all (ramified) primes $v {\not |} p$ of $K$ that are not in $\Sigma$. Visibly, 
   $\rho$ is such a representation, and thus $\Rrhobar_{\Sigma} \ne 0$.
   %
   Let ${\mathcal X}_{\Sigma} = \Spec(\Rrhobar_{\Sigma})$.
   We will be particularly interested in $\Tang_{\m}{\mathcal X}^{\mathrm{det}}_{\Sigma}$, the Zariski tangent space of ${\mathcal X}_{\Sigma}$ at
   the maximal ideal $\m$ (corresponding to $\rho$).

  It is often convenient to also fix the determinant of the representation. 
  Let ${\mathcal X}^{\mathrm{det}}_{\Sigma} \subset {\mathcal X}_{\Sigma}$ denote the
  closed subscheme cut out by the extra requirement that the kernel (in $G_{K,S}$) 
  of the determinant of any $A$-valued point of ${\mathcal X}^{\mathrm{det}}_{\Sigma}$  is equal to the kernel of the determinant of $\rho$; we view this subscheme as {\it the  universal nearly ordinary deformation space of $\rho: G_{K,S} \to \GL_2(E)$ with restricted ramification outside $p$ and $\Sigma$, and with fixed determinant.} 
  If $\Sigma = \emptyset$ then we say that we are in a
 {\bf minimally ramified}  situation.
  
  Let $\Cat$ denote the category
of Artinian $E$-algebras, and consider the functor
$\D_{\rho, {\Sigma}}: \Cat \rightarrow
\mathrm{Sets}$ \  that associates to an Artinian $E$-algebra $A$ with residue field $E$ the set 
$\D_{\rho, \Sigma}(A)$ of all deformations of the Galois representation $\rho$  to $A$ that are nearly ordinary and of restricted ramification at primes not dividing $p$ and outside~$\Sigma$.

If we impose the extra condition that the deformation have fixed determinant, we obtain a sub-functor that we denote $\D^{\mathrm{det}}_{\rho, {\Sigma}}$: 

$$A\  \longmapsto \  \D^{\mathrm{det}}_{ {\Sigma}}(A)\ \subset\  \D_{{\Sigma}}(A).$$

\begin{proposition}\label{defprop}
The scheme ${\mathcal X}^{\mathrm{det}}_{\Sigma}$ (pro-)represents the functor
$\Ddet_{\Sigma}$: there is a functorial correspondence 
${\mathcal X}^{\mathrm{det}}_{\Sigma}(A) = \Ddet_{\Sigma}(A)$.
%
\end{proposition}

\begin{proof} This follows in a straightforward way from the definitions, given the discussion above (compare \S 30 of~\cite{Mazur1},
or~\cite{Gouvea}).
\end{proof}

Let $A=E[\epsilon] = E\oplus \epsilon\cdot E$ be the (Artinian local) $E$-algebra of 
dual numbers, where $\epsilon^2=0$. 
There is a natural way of supplying  the sets $\Ddet_{{\Sigma}}(A) \subset \D_{{\Sigma}}(A)$
 with  $E$-vector space structures (consult loc. cit. for this) and we have the following natural isomorphisms.

\begin{proposition}\label{tanselm}  
 We have identifications of $E$-vector spaces 
$$\Tang_{\m}{\mathcal X}^{\mathrm{det}}_{\Sigma}  \simeq \Ddet_{ {\Sigma}}(A)\simeq 
 H^1_{\Sigma}(K,W).$$
  Suppose, moreover, that for all $v {\not |} p$, the action of the inertia group $I_v$ on $W$ factors through a finite quotient group. Then $H^1_{\Sigma}(K,W)$
  is  independent of $\Sigma$.
\end{proposition}

\begin{proof} The first identification above is standard. See, for example, \S 30 of~\cite{Mazur1},
or~\cite{Gouvea}.  The second identification follows from a computation, the format of which is standard; it follows from  expressing of $\Aut_{E[\epsilon]}(V\oplus \epsilon\cdot V) \approx
\GL_2(E[\epsilon])$ as a semi-direct product of $\Aut_E(V)$ by $\End_E(V)= W\oplus E$, and then noting that a lifting $\rho_A$ of  the homomorphism $\rho$ to $A$  provides us, in the usual manner, with a continuous $1$-cocycle $c(\rho_A)$ on $G_K$ with values in $W\oplus E$, and if the determinant is {\it constant}  the projection of $c(\rho_A)$ to the second factor vanishes. The underlying cohomology class, $h(\rho_A)$, of $c(\rho_A)$ determines and is determined by the deformation class of the homomorphism $\rho_A$. The assignment $\rho_A \mapsto h(\rho_A)$ induces  a natural embedding $$\Ddet_{ {\Sigma}}(A)\subset 
 H^1(K,W) \subset  H^1(K,W)\oplus  H^1(K,E)  =  H^1(K,\End_E(V)).$$ If $h \in H^1(K,W)$  and $v$  is a place, denote  its restriction  to $D_v$ by $h_v \in H^1(D_v,W_v)$.
 
  \begin{sublemma} The subspace $\Ddet_{ {\Sigma}}(A)\subset 
 H^1(K,W)$ consists of the set of cohomology classes $h \in H^1(K,W)$ such that for each place $v$, $h_v$  lies in ${\mathcal L}'_v \subset H^1(D_v,W_v)$ where
 
 \begin{itemize} 
 \item if $v\ | \ p$ then ${\mathcal L}'_v= {\mathcal L}_v$,
 \item if $v$ is in  $\Sigma$, then ${\mathcal L}'_v = H^1(D_v,W_v)$, and
 \item if $v$ is none of the above, then ${\mathcal L}'_v  ={\rm image}{\big\{}H^1(D_v/I_{v,\rho}, W_v){\big\}} \subset H^1(D_v, W_v).$
 \end{itemize}
 \end{sublemma}

 \begin{proof} We must show that in any of the cases listed ${\mathcal L}'_v$ is the subspace that cuts out the ramification conditions at $v$ required for a deformation to lie in $\Ddet_{ {\Sigma}}(A)$. 
 \begin{itemize}
 \item  Suppose $v\ | \ p$.  A deformation of $\rho$ that is nearly ordinary at $v$ can be represented by a homomorphism $\rho_A$ that when restricted to $I_v$ preserves the line $L_{A,v} = L_v\oplus \epsilon L_v \subset V_A=V\oplus \epsilon V $. It is then evident that  the associated cocycle $c(\rho_A)$, when restricted to $I_v$, takes its values in  $\Win_{v}$,and so $h(\rho_A)_{v} \in {\mathcal L}_v$.  Thus  ${\mathcal L}'_v\subset {\mathcal L}_v$. The opposite inclusion is equally straightforward.
 \item For $v$ in  $\Sigma$ we have ${\mathcal L}'_v=  H^1(D_v,W_v)= {\mathcal L}_v$.
 \item  If $v$ is none of the above, the requirement that the homomorphism $\rho_A$ have restricted ramification at $v$ is that its restriction to $I_v$ induces a homomorphism from $D_v/I_{v,\rho}$ to $\GL_2(E[\epsilon])$ and thus that the associated deformation is identified with an element of  $H^1(D_v,W_v)$ in the image of $H^1(D_v/I_{v,\rho}, W_v)$.
 \end{itemize}
 \end{proof}
 
  
  Returning now to the proof of our proposition, we must show that ${\mathcal L}'_v = {\mathcal L}_v$ for $v {\not |} p$ and $v \notin \Sigma$.  We have  $${\mathcal L}_v = {\rm image}{\big\{}H^1(D_v/I_v, W_v^{I_v}){\big\}} \subset {\mathcal L}'_v  ={\rm image}{\big\{}H^1(D_v/I_{v,\rho}, W_v){\big\}} \subset H^1(D_v, W_v).$$  But by  Lemma~\ref{inertsigma},  ${\mathcal L}_v = H^1(D_v, W_v)$;  thus all the groups displayed in the line above are equal.
  
  \medskip
  
  If the representation $\rho$ restricted
   to $I_v$ for $v$ not dividing $p$ factors through a finite quotient group, then
 independence of $\Sigma$ follows from Corollary~\ref{corollary:minimallevel}.
\end{proof}

\subsection{Infinitesimal Hodge-Tate weights}
\label{subsection:infinite}

Recall that $p$ is completely split in $K$, and therefore,
 $D_v = \Gal(\Qbar_p/\Q_p)$. By local class field theory,
 the  Artin map gives us a homomorphism $\Z_p^*\hookrightarrow \Q_p^* \hookrightarrow D_v^{\rm ab}$, 
the latter group being the maximal abelian quotient of $D_v$.   
For $v|p$, we have that 
$\Lo_v  \simeq   H^1(D_v,\Win_v)$ by Lemma~\ref{dvlemma}. Consider  
the one-dimensional $E$-vector space ${\mathcal E}:= \Hom(\Z_p^*,E)$, 
and form  the $E$-linear functional  $\pi_v:\Lo_v \to {\mathcal E}$ 
given by composition
$$
\Lo_v  \simeq   H^1(D_v,\Win_v)  {\stackrel{\epsilon}{\longrightarrow}}
H^1(D_v,E) = \Hom(D_v^{\rm ab},E)  \to  \Hom(\Z_p^*,E)= {\mathcal E}.$$ 
(in the above discussion, $\Hom$ always means continuous homomorphisms).

\begin{definition}
\label{definition:HTW}
 Let $c \in H^1_{\Sigma}(K,W)$. For $v |  p$,
the {\bf infinitesimal Hodge-Tate weight of $c$ at $v$} is the
image of $c$ under the composition
$$\begin{diagram}
\omega_v: H^1_{\Sigma}(K,W) & \rTo &
 \Lo_v  & \rTo^{\pi_v} &  {\mathcal E}
\end{diagram}
$$
\end{definition}

 Putting the $\omega_v$'s together,
 we get the mapping of infinitesimal Selmer to infinitesimal weight space,
 
 $$\omega: = \bigoplus_{v|p}\omega_v: H^1_{\Sigma}(K,W) 
\to   \bigoplus_{v|p}{\mathcal E}  =  \{\bigoplus_{v|p}\Q\}\otimes_\Q{\mathcal E}.$$
 
 A vector $w$ in this latter vector 
space $\{\bigoplus_{v|p}\Q\}\otimes_\Q{\mathcal E}$ will be called {\it rational} 
if $w$ is expressible as $r\otimes e$ for  $r \in \{\bigoplus_{v|p}\Q\}$ 
and $e \in {\mathcal E}$.

\subsection{Motivations}

One of the motivating questions investigated in this article is
the following: under what conditions does the image of
$\omega$ contain non-zero rational vectors? 

If $K$ is a totally real field and $\rho$ is potentially semistable
and odd at all
real places of $K$, then one expects that $\rho$ is modular.
If $\rho$ is modular, and $K$ is totally real,
 the ordinary Hida families associated to $\rho$ produce
many deformations in which the Hodge-Tate weights vary
rationally with respect to one another. We conjecture that
in a strong sense the converse is almost true. Namely, that any
deformation $c \in H^1_{\Sigma}(K,W)$ with $\omega(c)$ rational
\emph{only} occurs if $c$ has arisen from the Hida family
of some odd representation over a totally real field,  the
only exception being when $c$ arises by induction from a
family of representations over a CM field, or --- in a final
exceptional case that only conjecturally
occurs when $\rho$ is an Artin 
representation --- when $\rho$ arises by base change from
some \emph{even} representation.

We will be paying special attention
in this article to \emph{Artin representations} $\rho$;
that is, to representations with finite image in $\GL_2(E)$.
The results we obtain  are, for the most
part, \emph{conditional} in that they depend on a certain
(``natural,'' it seems to us) extension of the classical
conjecture of Leopoldt.

%

\medskip

\begin{definition} 
\label{definition:ICD}
The representation $\rho$ admits
\emph{infinitesimally classical deformations} if
there exists a non-zero class $c \in H^1_{\Sigma}(K,W)$
with rational infinitesimal Hodge-Tate weights.
\end{definition}

\medskip

 \Remark  
If $\rho$ is an Artin representation then by Corollary~\ref{corollary:minimallevel}
there is an isomorphism $H^1_{\Sigma}(K,W) = H^1_{\emp}(K,W)$, so it
suffices to consider the case $\Sigma = \emp$.
 If $\rho$ is classical with distinct Hodge-Tate weights, then
one also expects that all infinitesimal deformations of $\rho$ are minimally ramified,
and hence that
 $H^1_{\Sigma}(K,W) = H^1_{\emp}(K,W)$.  This expectation arises
 from the automorphic setting, where an intersection of two Hida families in regular
 weight would contradict the semi-simplicity of the Hecke action.

\section{The Leopoldt conjecture}
\label{section:theleopoldtconjecture}

\subsection{Preliminaries}\label{prelim}

In this section we will let $p$ be a prime number and $E/\Q_p$  a field 
extension, but some of what we say will have an analog if we take $p$ 
to be $\infty$, i.e., if we consider our field $E$ to be an extension field of ${\bf R}$.  

Let $M/\Q$ be a finite Galois extension with Galois group $G$ of order $n$.  Denote by ${\mathcal O}_M^*$ group of units in the ring of integers of  $M$, and let $$U_{\glob}:= {\mathcal O}_M^*\otimes_{\Z}E$$ be the $E$-vector space obtained by tensor product with $E$. By the Dirichlet unit theorem,
$U_{\glob}$  is of dimension either $n-1$ or ${\frac{n}{2}}-1$, depending upon whether $M$ is totally real or totally  complex. We will think of $U_{\glob}$ as a (left) $E[G]$-module induced by the natural action of $G$. In the sequel, our modules over non-commutative rings  will be understood to be {\it left} modules unless we explicitly say otherwise  (and, on occasion, we will say otherwise).

 For $v$ a place of $M$ dividing $p$, let $M_v$ denote the completion 
of $M$ at $v$, and let
${\mathcal O}_{M_v}^*\subset M_v^*$ be the group of units in the ring of integers of the discrete valued field $M_v$. View $$ {\widehat {\mathcal O}}_{M_v}^*:= \lim_{\leftarrow}
  {{\mathcal O}}_{M_v}^*/({{\mathcal O}}_{M_v}^*)^{p^n},$$ the $p$-adic completion of $ { {\mathcal O}}_{M_v}^*$, as $\Z_p$-module, 
and form   
$$U_v:= {\widehat{\mathcal O}}_{M_v}^*\otimes_{\Z_p}E,$$ 
which we think of as $D_v$-representation space, where $D_v \subset G$ is the decomposition group at $v$. 

We have that $U_v$ is a free $E[D_v]$-module of rank one. The group $G$ acts transitively on the set of places $v|p$ and,  correspondingly, conjugation by $G$ is a transitive action  on the set of decomposition groups $\{D_v\}_{v|p}$. Also, if $g \in G$ and $v|p$, the automorphism $g:M \to M$ induces an isomorphism also denoted $g:U_v \cong U_{gv}$. 

\begin{definition} $U_{\loc}: = \prod_{v|p}U_v$.
\end{definition}
   We view $U_{\loc}$ in the natural way as $E[G]$-module, the action of an element $g \in G$ sending the vector $(\dots, x_v, \dots )$  in $\prod_{v|p}U_v$  to $(\dots, y_v, \dots )$ where $y_{gv}=gx_v$ for $v |p$.  
   
  Fixing a place $w | p$ and putting $D= D_{w}$, we have  natural 
identifications 
$$ \Ind_D^GU_{w} = E[G]\otimes_{E[D]}U_{w}\cong  \prod_{v|p}U_v=U_{\loc}$$ 
by the homomorphism induced from the 
rule $g\otimes x \mapsto (\dots, x_v, \dots ) \in  \prod_{v|p}U_v$,
 where, for $g\in G$ and $x \in U_{w}$, we have that $x_v$ 
(the $v$-th entry of the image) is $0$ if $v \ne gw$ and is $gx$ if $v = gw$. 
The $E[G]$-module $U_{\loc}$ is free of rank one.

\bigskip

Now let us consider  ``fields of definition" of the structure inherent to the $E[G]$-module $U_{\loc}$.  Let $F \subset E$ be a subfield (and we will be interested in the cases where $F$ is a field of algebraic numbers). 

If we choose a generator $u \in U_{w}$ of the $E[D]$-module $U_{w}$,
 then setting $U_{w,F}:= F[D]\cdot u \subset  U_{w}$ we see that 
$U_{w,F}$ is a free $F[D]$-module of rank one such 
that $U_{w,F}\otimes_FE \simeq U_{w}$.
 Refer to such a sub-$F[D]$-module of $U_{w}$ as an {\it $F$-structure} on $U_{w}$.

 Given any such $F$-structure on $U_{w}$ we have a canonical right $F[D]$-module 
structure on  $U_{w,F}$ as well. Namely, if $$x = a\cdot u \in U_{w,F}$$ for $a \in F[D]$, and if $b  \in F[D]$, the right-action of $b$ on $x$ is given by $x\cdot b : = ab u$.  This extends by base change to give a right action of $E[D]$ on  $U_{w}$. Given an  { $F$-structure} on $U_{w}$, we get a corresponding $F$-structure on 
$U_\loc = E[G]\otimes_{E[D]}U_{w}$ by setting 
$$U_{\loc,F} := F[G]\otimes_{F[D]}U_{w,F}$$  which we view, then, as a (left) $F[G]$-module and a right $F[D]$-module.

  These $F$-structures  will be called {\it  preferred $F$-structures on $U_\loc$} although, in fact, we will consider no other ones, and we identify $U_\loc$ with  $U_{\loc,F}\otimes_FE$, viewing it as a (left) $E[G]$-module and a right $E[D]$-module. The ambiguity in the choice of $F$-structure then boils down to which $(F[D])^*$-coset of generator of $E[D]$-module $ U_{w}$ we have chosen. 
  
   Suppose, then, we have
such a preferred $F$-structure given on $U_\loc$. By an {\bf $F$-rational, right $E[D]$-submodule}  $V \subset U_\loc$ we mean a right $E[D]$-submodule of $ U_\loc$ for which there exists a right $F[D]$-submodule $V_F \subset U_{\loc,F}$ such that $V = V_F\otimes_FE$. (For a systematic treatment of these issues, see the Appendix: section~\ref{app}.)

 For any $v|p$, the  natural injection  $M\hookrightarrow M_v$ induces a homomorphism of $E[D_v]$-modules  $$U_{\glob} {\stackrel{\iota_v}{\longrightarrow }}U_v$$ and the product of these,
    
$$\iota=\iota_p = \prod_{v|p}\iota_v:  U_\glob \longrightarrow  \prod_{v|p}U_v= U_{\loc}$$ is naturally a homomorphism of $E[G]$-modules.  We can also think of $\iota$ as the natural $E[G]$-homomorphism $$\iota: U_\glob \to \Ind_D^GU_{w}$$ induced from the 
$E[D]$-homomorphism $\iota_{w}: U_\glob \to U_{w}$.

The classical $p$-adic Leopoldt conjecture is equivalent to
the statement that $\iota$ is injective for $p$ any prime. This has been proven when $M/\Q$ is an abelian Galois extension~\cite{Brumer}.

 \bigskip

 \Remark In the situation where $p=\infty$,
 we have a {\it somewhat} analogous homomorphism $$\iota_{\infty}: U_\glob = {\mathcal O}_M^*\otimes \R \longrightarrow  \prod_{\text{$v$} \ {\bf infinite}}\R $$  given by  $\lambda(u) =$ the vector
 $\displaystyle{(\dots, \lambda_v(u), \dots) \in  \prod_{v | \infty}\R}$
 with $\lambda_v(u) = \log |u|_v$ for $v$ real, and given by
 $\lambda_v(u) = 2\log |u|_v$ for $v$ complex, which has been known (for over a century) to be injective.
 
\subsection{The Strong Leopoldt Conjecture}
 
 {\it Assume the classical Leopoldt Conjecture.}  The stronger version of Leopoldt's Conjecture that we shall formulate has to do with the manner in which the $E[G]$-submodule $$\iota(U_\glob) \subset U_\loc$$ is  skew to  the $F$-structures we have considered on $U_\loc$. 

\begin{definition}[Strong Leopoldt Condition for $E/F$] Let us say that $M/\Q$ satisfies  the {\bf Strong Leopoldt Condition for $E/F$} if the map $\iota$ is injective, and
furthermore for any preferred $F$-structure on $U_\loc$, and every $F$-rational right $E[D]$-submodule $Z \subset U_\loc$, and every $E[G]$-submodule $Y\subset U_\loc$ such that $Y$ is isomorphic to $U_\glob$ as $E[G]$-module, we have the inequality of dimensions
$$\dim(\iota(U_\glob)\cap Z)\ \le \ \dim(Y\cap Z).$$
\end{definition}

\bigskip

  For ease of reference, let us refer to the $E[G]$-submodules $Y\subset U_\loc$ that appear in this definition as the {\bf competitors}  (i.e., to $U_\glob$) and  the $F$-rational right $E[D]$-submodules $Z \subset U_\loc$ as the {\bf tests}.  So the Strong Leopoldt Condition for $E/F$ is satisfied if, so to speak,  $U_\glob$ ``wins every test with every competitor," or at least manages a draw.

\begin{conjecture} [{\bf Strong Leopoldt}]  Let $E/\Q_p$ be any algebraic extension field, and $F \subset E$ any subfield where  $F/\Q$ is algebraic. Let $M/\Q$ be any finite Galois extension. Then $M/\Q$ satisfies  the  Strong Leopoldt Condition for $E/F$.
\end{conjecture}

\Remarks

\begin{enumerate}
\item This conjecture is visibly a combination of the classical 
Leopoldt Conjecture and, as it seems to us, a fairly 
natural (genericity) hypothesis: the conjecture predicts that the $E$-vector 
space generated by global units --- as it sits in the $E$-vector space of $p$-adic 
local units --- is as {\it generic} as its module structure allows, relative to 
the preferred $F$-structures on the bi-module $U_\loc$ of  $p$-adic local units.  
This genericity condition is discussed in a general context in the 
Appendix (compare definition~\ref{definition:definitionofgeneric}).

\item
Note that the $\Z[G]$-module ${\mathcal O}_M^*$, or its torsion-free 
quotient, the {\it lattice} of global units, does {\it not} play a role in this 
Strong Leopoldt Condition, a condition that only  concerns itself with the 
tensor product of this unit lattice with the large field $E$. 
\item \label{item:3}
 Choosing model $F$-representations $V_{\eta}$ of $G$,    for  $\eta$ running 
through the distinct irreducible  $F$-representations of $G$,
 we have a canonical decomposition of the bi-module $U_{\loc,F}$ as $$U_{\loc,F} = 
\bigoplus_{\eta} V_{\eta} \otimes V_{\eta} ^\#$$ 
where for each $\eta$,  $V_{\eta} ^\#: = \Hom_{F[G]}(V_{\eta},U_{\loc,F})$ is 
viewed  as a right $F[D]$-module.  Put $V_{\eta, E}: = V_{\eta}\otimes_FE$, and 
similarly  $V_{\eta, E}^\#: = V_{\eta}^\#\otimes_FE$.  Assuming the 
classical Leopoldt conjecture we may identify $U_\glob$ with its 
image under $\iota_p$ in $U_\loc$ and 
write the image of $U_\glob$ in $U_{\loc}$  as a direct sum
$$U_\glob = \bigoplus_{\eta} V_{\eta} \otimes V_{\eta,\glob} ^\# 
\subset \bigoplus_{\eta} V_{\eta} \otimes V_{\eta,E} ^\#,$$ 
the ``placement" of $U_\glob$ in $U_\loc$ being determined by 
giving the vector subspaces $$V_{\eta,\glob} ^\# \subset  V_{\eta,E} ^\#$$ 
for each $\eta$. The Strong Leopoldt Conjecture, then, 
requires the $E[D]$-submodule $V_{\eta,\glob} ^\#$ 
to be generically positioned (in the sense described above) 
with respect to the $F[D]$-structure of  $ V_{\eta,E} ^\#$.
In the special case of Galois extensions $M/\Q$ 
where the classical Leopoldt conjecture holds and, for each $\eta$, 
we have that either $V_{\eta,\glob} ^\# = 0$ or  $V_{\eta,\glob} ^\# =  
V_{\eta} ^\#$  (i.e., where $U_\glob$ consists of a direct sum of 
isotypic components of the $E[G]$-representation $U_\loc$) the 
Strong Leopoldt Condition (trivially) holds because --- 
in this case --- there are no competitors $Y$ 
(in the sense we described above) other than $U_\glob$ itself.
\end{enumerate}

\begin{corollary}\label{abcor} If $M/\Q$ is an abelian extension,  
the Strong Leopoldt Conjecture holds.
\end{corollary}

 \begin{proof} The classical Leopoldt conjecture holds for abelian extensions of $\Q$, and $U_\glob$ consists of a direct sum of isotypic components of the $E[G]$-representation $U_\loc$.
\end{proof}
 
 \begin{corollary}\label{trcor} If $M/\Q$ is a totally real Galois extension for which the classical Leopoldt conjecture holds, the Strong Leopoldt Conjecture holds.
\end{corollary}

 \begin{proof} Since $M$ is totally real, Leopoldt's conjecture is equivalent to the
 statement that there is an isomorphism $U_{\loc}/U_{\glob} \simeq E$ as $G$-modules.
 Equivalently,  $U_\glob$ consists of
 a direct sum of all the non-trivial isotypic components of $U_\loc$ as an
$E[G]$-representation.

 \end{proof}

\medskip

The above are cases, then, where (given the classical conjecture) the Strong Leopoldt Conjecture holds vacuously.

\subsection{The Strong Leopoldt Conjecture and the $p$-adic Schanuel conjecture}\label{Leopshan}

 To get a sense of the nature of the Strong Leopoldt  conjecture in a relatively simple situation,  but a more interesting one than treated in Corollaries \ref{abcor} and \ref{trcor}, let us suppose that:

\begin{itemize}
\item  the decomposition group $D$ is trivial, 
\item the classical Leopoldt conjecture holds, and 
\item  There is an irreducible $G$-representation $\eta$, and a nonzero vector ${v^\#} \in  V_{\eta, E}^\#$ such that the $E[G]$-sub-representation $U_\glob \subset U_\loc$ is a direct sum 
$$ U_\glob\ \   = \ \ I\ \bigoplus\  \{V_{\eta, E}\otimes_E{v^\#}\} \ \subset\ U_\loc$$ where $I$ is an $E[G]$-representation that is a direct sum of {\it some} isotypic components in $U_\loc$.
\end{itemize}

(To make it interesting, of course,  we would want that this irreducible representation $\eta$ to be of dimension $>1$. We shall consider some explicit examples when this occurs in subsequent sections.)  
In the above situation the only possible {\it competitors} $Y$ are of the form 
$$ Y\ \   = \ \ I\ \bigoplus\  \{V_{\eta, E}\otimes_E{y^\#}\} \ \subset\ U_\loc$$ for some nonzero vector ${y^\#} \in  V_{\eta, E}^\#$.  One easily sees that
--- in this situation --- 
to test whether or not  the Strong Leopoldt Conjecture holds, it suffices to consider only {\it test vector spaces} $Z$ contained in  $V_{\eta} \otimes V_{\eta}^\#$. 
 Explicitly,  the Strong Leopoldt Conjecture will hold   if and only if  $$\dim\big(\{V_{\eta, E}\otimes_E{v^\#}\}\cap Z_E\big) \ \le \ \dim\big(\{V_{\eta, E}\otimes_E{y^\#}\}\cap Z_E\big)$$ for all  $F$-vector subspaces $Z \subset V_{\eta} \otimes V_{\eta} ^\#$.
  Let $d:=\dim_F(V_{\eta} ^\#)$ and choose an $F$-basis for  $V_{\eta} ^\#$, identifying $V_{\eta} ^\#$ with $F^d$, and (tensoring with $E$)   $V_{\eta, E} ^\#$ with $E^d$.  After this identification we may view the vector ${v^\#}$ as a $d$-tuple of elements of $E$ $$ {v^\#}:=  (e_1,e_2,\dots,e_d).$$  Without loss of generality we may assume that $e_d=1$.
\begin{theorem}\label{transcthm} In the situation described above,
\begin{enumerate}
\item If  the Strong Leopoldt Condition for $E/F$ holds, then the elements  $ e_1,e_2,\dots,e_d=1\in E$ are linearly independent over $F$.
\item If the transcendence degree 
of the field $F(e_1,e_2,\dots,e_{d-1})\subset E$ is $d-1$, then the Strong Leopoldt Condition for $E/F$ holds in this situation.
\end{enumerate}
\end{theorem}
\begin{proof} We begin with part one.
If  $e_1,e_2,\dots,e_d\in E$ satisfies a nontrivial $F$-linear relation explicitly given by
$\lambda(e_1,e_2,\dots,e_d) =0$, let $V_o^{\#} \subset V_\eta^{\#}$ be the kernel of $\lambda$; so $v^{\#} \in V_o^{\#}$. Define the $F$-vector test subspace $$Z:=   V_\eta\otimes_F V_o^{\#} \subset  V_\eta\otimes_F V_\eta^{\#},$$ and choose as ``competitor" 
the $E[G]$-sub-representation $$Y:= V_{\eta, E}\otimes_Ey^{\#} \subset  V_{\eta, E}\otimes_E V_{\eta, E}^{\#}.$$ where $y^{\#}\in V_{\eta, E}^{\#}$ is any vector such that $\lambda(y^{\#}) \ne 0$.   Then
$$1= \dim(\{V_{\eta, E}\otimes_Ev^{\#}\} \cap Z_E)\ > \ \dim(Y \cap Z_E) = 0.$$ 
This proves part one.

\

 Suppose now that the transcendence degree of the field $F(e_1,e_2,\dots,e_{d-1})\subset E$ is $d-1$. Form the polynomial ring $R:= F[X_1,X_2, \dots, X_{d-1}]$, and consider the $R$-modules $V_{\eta, R}:= V_\eta\otimes_F R$,  $V_{\eta, R}^{\#}:= V_\eta^{\#}\otimes_FR=R^d$, and $Z_R = Z\otimes_FR \subset  V_{\eta, R}\otimes_RV_{\eta, R}^{\#}$. Now consider the element $X^{\#}:=(X_1,X_2,\dots,  X_{d-1}, 1) \in R^d = V_{\eta, R}^{\#}$ and form the $R$-submodule
$$\{V_{\eta, R}\otimes_R X^{\#} \}\bigcap Z_R  \ \subset \  V_{\eta, R}\otimes_RV_{\eta, R}^{\#}.$$ By sending the independent variables $(X_1,X_2,\dots,  X_{d-1})$ to $(e_1,e_2,\dots,e_{d-1})$ we get an injective $F$-algebra homomorphism from $R$ to $E$ identifying the  base change to $E$  of the $R$-module $\{V_{\eta, R}\otimes_R X^{\#}\} \bigcap Z_R$ with the $E$-vector space $\{V_{\eta, E}\otimes_Ev^{\#}\} \cap Z_E$. 
For any $y^{\#} = (c_1,c_2,\dots,c_{d-1}, 1) \in  E^d$, 
by sending the independent variables $(X_1,X_2,\dots,  X_{d-1})$ to $(c_1,c_2,\dots,c_{d-1})$ we get an  $F$-algebra homomorphism from $R$ to $E$ identifying the  base change to $E$  of the $R$-module  $\{V_{\eta, R}\otimes_R X^{\#} \}\bigcap Z_R$
 with the $E$-vector space $$Y\cap Z_E=\{V_{\eta, E}\otimes_Ey^{\#}\} \cap Z_E.$$ 
  An appeal to the principle of upper-semi-continuity then guarantees that 
  $$\dim(\{V_{\eta, E}\otimes_Ev^{\#}\} \cap Z_E) \le \dim(\{V_{\eta, E}\otimes_Ey^{\#}\} \cap Z_E)$$  concluding the proof of our theorem.
\end{proof}

The Schanuel Conjecture is an (as yet unproved) elegant statement regarding the transcendentality of the exponential function evaluated on algebraic numbers (or of the logarithm,  depending upon the way it is formulated; we will give both versions below).  Its proof would generalize substantially the famous Gelfond-Schneider Theorem that establishes the transcendentality of $\alpha^\beta$ for algebraic numbers $\alpha$ and $\beta$ (other than the evidently inappropriate cases $\alpha =0,1$ or $\beta$ rational).  There is a $p$-adic 
analog to the Schanuel Conjecture, currently unproved, and a $p$-adic analog to the Gelfond-Schneider Theorem established by Kurt Mahler. 

Here, then, are the full-strength conjectures in two, equivalent, formats.

\bigskip

\begin{conjecture}[{\bf Schanuel} --- exponential formulation]
Given any $n$ complex numbers $z_1,\dots,z_n$ which are linearly independent 
over $\Q$, 
the extension field 
$\Q(z_1,\dots,z_n,\exp(z_1),\dots,\exp(z_n))$
 has transcendence degree at least $n$ over $\Q$.
\end{conjecture}

\begin{conjecture}  [{\bf Schanuel} --- logarithmic formulation]
Given any $n$ nonzero complex numbers $z_1,\dots,z_n$ whose logs (for any choice of the multi-valued
 logarithm)
 are linearly independent over the rational numbers $\Q$, then the extension field $\Q(\log z_1,\dots, \log z_n ,z_1,\dots,z_n)$ has transcendence degree at least $n$ over $\Q$.
\end{conjecture}

A somewhat weaker formulation of this conjecture is as follows:

\begin{conjecture}  [{\bf Weak Schanuel} --- logarithmic formulation]
  Given any $n$ algebraic numbers $\alpha_1,\dots,\alpha_n$ whose logs are linearly independent over the rational numbers $\Q$, then the extension field $\Q(\log \alpha_1,\dots, \log \alpha_n)$ has transcendence degree  $n$ over $\Q$.
\end{conjecture}

  We are interested in the $p$-adic version of the above conjecture, namely:
 
\begin{conjecture}  [{\bf Weak Schanuel} --- $p$-adic logarithmic formulation]

  Let $\alpha_1,\dots,\alpha_n$  be $n$ nonzero algebraic numbers contained in a finite extension field $E$ of $\Q_p$. Let $\log_p:E^* \to E $ be the $p$-adic logarithm normalized
so that $\log_p(p)=0$. If  $\log_p \alpha_1,\dots, \log_p\alpha_n$ are linearly independent over the rational numbers $\Q$, then the extension field $\Q(\log \alpha_1,\dots, \log \alpha_n) \subset E$ has transcendence degree  $n$ over $\Q$.
\end{conjecture}
 
  \begin{theorem}\label{weakStrong}  If $M/\Q$ satisfies the hypotheses of the beginning of this subsection, then the classical Leopoldt Conjecture plus the weak $p$-adic Schanuel Conjecture implies the Strong Leopoldt Conjecture.
  
  \end{theorem}

\begin{proof}  This follows directly from part 2 of Theorem~\ref{transcthm}.
\end{proof}

 In the next two subsections we examine cases of the above set-up.

\bigskip

\subsection{The Strong Leopoldt Condition for complex $S_3$-extensions of $\Q$}

  Consider the case where  $M/\Q$ 
is Galois with $G = S_3$, the symmetric group on three letters.
Suppose that $p$ 
splits completely in $M$, so the decomposition group $D$ is trivial, 
and  $M$ is totally complex. Let $E$ be an algebraic closure of $\Q_p$, 
and $F \subset E$ an algebraic extension of $\Q$ contained in $E$.  
  \begin{theorem}\label{S3SLC} The Strong Leopoldt Condition is true in this situation.
\end{theorem}

    We begin our discussion assuming the hypotheses and notation of the opening paragraph of this subsection.  We have a canonical isomorphism of  $F[G]$-modules
  $$U_{\loc, F} \simeq F \oplus F \oplus V_{\eta}\otimes_FV_{\eta}^{\#},$$ 
where the action of $G$ on the various summands is as follows:
the action on the first summand is trivial,
the  action on the second is via the sign representation, 
the action on the third is via action on the first tensor factor,
i.e., $g\cdot (x\otimes x^{\#}) = (gx)\otimes x^{\#}$ for $g \in G$, $x \in V_\eta$
and $x^{\#} \in V_{\eta}^{\#})$, where $V_\eta$ is the irreducible 
representation of $G=S_3$ on a two-dimensional vector space over $F$.
The $E[G]$-module $U_\glob$ is isomorphic to $V_{\eta, E}: = V_{\eta}\otimes_FE$.

 The homomorphism $\iota_p: U_\glob \to U_\loc$ 
is injective  (i.e., the classical Leopoldt Conjecture is --- in fact, 
trivially --- true) since the $G$-representation $U_\glob $ is irreducible 
and $\iota_p$ doesn't vanish identically.  Moreover, the image of $\iota_p$ 
is necessarily contained in the third summand of the decomposition displayed 
above; i.e., we may view  $\iota_p$ as being an 
injection of $E[G]$-modules  $$\iota_p: U_\glob \to V_{\eta, E}\otimes_EV_{\eta, E}^{\#}=(V_\eta\otimes_FV_\eta^{\#})\otimes_FE  \subset U_\loc,$$  identifying $U_\glob$ 
with the $E[G]$-sub-representation $V_\eta\otimes_Fv^{\#} \ \subset\ V_{\eta, E}\otimes_EV_{\eta, E}^{\#}$ for a nonzero vector  $v^{\#} \in V_{\eta, E}^{\#}$. In particular, we are in exactly the situation discussed in subsection~\ref{Leopshan}.

 Choose an $F$-basis $\{v_1^{\#} ,v_2^{\#} \}$ 
of $V_\eta^{\#}$ and write $v^{\#} = a_1v_1^{\#} +  a_2v_2^{\#}$ 
for specific elements $a_1, a_2 \in E$  we see that the subset of $E \cup \{\infty\}$ 
consisting in the set of  images of $\mu:=a_1/a_2$ under all 
linear fractional transformations with coefficients 
in $F$  (call this the subset of $F$-{\bf slopes} of $U_\glob$)   
is an invariant of  our situation. 

\begin{lemma}\label{SLCEQ} In the situation described above:
 \begin{enumerate}
\item  If the  Strong Leopold condition holds for $E/F$, then the set of $F$-slopes of $\lambda$ is not 
contained  in $F$. 
\item If  the set of $F$-slopes of $\lambda$ is not 
contained  in any quadratic extension of  $F$, then the  Strong Leopold condition holds for $E/F$.
\end{enumerate}
\end{lemma}

\begin{proof}
 If the set of $F$-slopes of $U_\glob$ is  
contained  in $F$, then the image of $U_\glob$ in $U_\loc$ is itself the base change to $E$ of  an $F[G]$-sub-representation in the preferred $F$-structure of $U_\loc$ which is an  $F$-vector subspace  (relative to the preferred $F$-structure) of $U_\loc$, and taking this $F[G]$-sub-representation itself as test object $Z$ --- and $Y \subset U_\loc$ as {\it any} irreducible two-dimensional $E[G]$ sub-representation different from $U_\glob$
 --- give  the dimension inequality $$\dim (U_\glob\cap Z)  > \dim (Y\cap Z) $$ that contradicts the Strong Leopoldt Condition. 
This proves part one.

\medskip

  Suppose now, that the set of $F$-slopes of $\lambda$ is not 
contained  in in any quadratic extension of  $F$. 
 As discussed in Remark~\ref{item:3} of the previous subsection, the $E[G]$-representation subspaces $Y \subset U_\loc$ that are isomorphic to $U_\glob$ (as $E[G]$-representations)  are  of the  form $V_{\eta, E}\otimes_Ey^{\#}$ for some nonzero vector $y^{\#} \in V_{\eta, E}^{\#}$.  Let $Z \subset  V_{\eta}\otimes_F V_{\eta}^{\#}$  be a test $F$-rational vector space, and let $Z_E = Z \otimes_F E$. The three cases that are relevant to us are when $\dim(Z_E)=1,2,3$. One checks that (since the set of $F$-slopes of $U_\glob$ is not 
contained  in $F$)  $V_{\eta, E}\otimes_Ev^{\#}$ does not contain an $F$-rational line, nor is annihilated by a nontrivial $F$-rational linear form. Therefore we need only treat the case where $Z_E$ is of dimension two. Now an inequality of dimensions of the form $ \dim(Y\cap Z_E) <   \dim(U_\glob\cap Z_E)$ can only happen  if $ \dim(Y\cap Z_E)=0$ and  $\dim(U_\glob\cap Z_E) = 1$, i.e.,
 if $Y+Z_E =  (V_{\eta,E}\otimes_E V_{\eta,E}^{\#})$ while $U_\glob+ Z_E $ is a proper subspace of $(V_{\eta,E}\otimes_F V_{\eta,E}^{\#})$. 

One can now make this explicit in terms of $4\times 4$ matrices as follows. We can choose an $F$-bases   $\{v_1,v_2\}$  of the vector space $V_\eta $ and $\{v_1^{\#},v_2^{\#}\}$  of the vector space $V_\eta^{\#}$  so that we may write $v^{\#} = v_1^{\#} + \mu v_2^{\#}$ and $y^{\#} = v_1^{\#} + \nu v_2^{\#}$ where by hypothesis $\mu$ is not 
contained in any quadratic extension of  $F$  (and we know nothing about $\nu$). We have a basis  $\{v_{i,j}:=v_i\otimes v_j^{\#}\}_{i,j}$ for  $V_{\eta,E} \otimes_E V_{\eta,E}^{\#}$ and in terms of this basis:

\begin{itemize}

\item $U_\glob$ is generated by $v_{1,1} + \mu v_{1,2}$  and $v_{2,1} + \mu v_{2,2}$,
 \item  any competitor $Y$ is generated by $v_{1,1} + \nu v_{1,2}$  and $v_{2,1} + \nu v_{2,2}$, 
 \item any test $Z$ is generated by $\sum_{i,j}a_{i,j}v_{i,j}$  and $\sum_{i,j}b_{i,j}v_{i,j}$  for $a_{i,j}, b_{i,j} \in F$.
 \end{itemize}
 
 We leave it to the reader to check (a)  that if $\mu$ is not contained in any quadratic extension field of $F$ the rank of the matrix
 
 \[\left(\begin{array}{cccc}
 1 & \mu & 0 & 0 \\
 0 & 0 & 1 & \mu \\
 a_{1,1} & a_{1,2} & a_{2,1} & a_{2,2} \\
 b_{1,1} & b_{1,2} & b_{2,1} & b_{2,2} \\
 \end{array}\right)\]

is greater than or equal to the rank of the matrix
 \[\left(\begin{array}{cccc}
 1 & \nu & 0 & 0 \\
 0 & 0 & 1 & \nu \\
 a_{1,1} & a_{1,2} & a_{2,1} & a_{2,2} \\
 b_{1,1} & b_{1,2} & b_{2,1} & b_{2,2} \\
 \end{array}\right)\]
 
 \noindent for any $\nu$, and (b) that this concludes the proof of our lemma.
\end{proof}
 
 \subsection{Proof of Theorem~\ref{S3SLC}}
 
 Given Lemma~\ref{SLCEQ} it  clearly 
suffices to show that the slope $\mu$ defined in the previous subsection is transcendental. Here, more concretely, is how to see this 
invariant $\mu$. If $K_1, K_2, K_3 \subset M$ are the 
three conjugate subfields of order 
three, let  $\epsilon_1, \epsilon_2, \epsilon_3$ be fundamental 
units  of these subfields ($\epsilon_i > 1$ in the unique 
real embedding $K_i \hookrightarrow {\bf R}$) and 
note that   $\epsilon_1 \epsilon_2 \epsilon_3 = 1$, 
and that the group $G$ acts as a full group of 
permutations of the set $\{\epsilon_1, \epsilon_2, \epsilon_3\}$.

 Fix a place $w$ of $M$ over $p$ and let ${\mathcal O}_{w}
=\Z_p$ denote the ring of integers in the $w$-adic completion of $M$. Let
 $\epsilon_{1,w}, \epsilon_{2,w}, \epsilon_{3,w}
\in {\mathcal O}_{w}^*=\Z_p^*$ denote the images of the units $\epsilon_1, \epsilon_2, \epsilon_3$. Let $\log_p:\Z_p^* \to \Q_p$ denote the classical $p$-adic logarithm. The $F$-slope of $U_{\glob}$ in $U_{\loc}$ is, up to $F$-linear fractional transformation, equal to $$\mu= \log_p(\epsilon_{1,w})/\log_p(\epsilon_{2,w}).$$  
 The transcendence of $\log_p(\epsilon_{1,w})/\log_p(\epsilon_{2,w})$ follows from 
Mahler's $p$-adic version of the Gelfond-Schneider theorem
(\cite{Mahler2}, Hauptsatz, p.275). Mahler proved that  for any two
algebraic numbers $\alpha$, $\beta$ in $\Qbar_p$, either
$\log_p(\alpha)/\log_p(\beta)$ is rational or transcendental. Since $\epsilon_{1}$ and $\epsilon_2$ are linearly independent in the torsion-free quotient of ${\mathcal O^*_M}$, the ratio  $\log_p(\alpha)/\log_p(\beta)$  is not rational, and hence is transcendental.

 \bigskip

\Remark The Archimedean analog of the above transcendence statement follows from the classical result of Gelfond-Schneider. Specifically, embed $M$ in ${\bf C}$ (in one of the three possible ways). For $i \ne j$  let $\log(\epsilon_{i})/\log(\epsilon_{j})$ refer to the ratios of any of the values of the natural logarithms of $\epsilon_{i}$ and of $\epsilon_{j}$; these ratios are transcendental. 
 A formally similar question concerning $p$-adic transcendence in the arithmetic of elliptic curves  is --- to our knowledge --- currently unknown. Namely, let $M/\Q$ be an $S_3$-extension,  $\sigma, \tau \in S_3$ denoting the elements of order $2$ and $3$ respectively. Let $v$ be a place of $M$ of degree one dividing a prime number $p$.  Now let $A$ be an elliptic curve over $\Q$ with Mordell-Weil group over $M$ denoted $A(M)$ and suppose that  the $S_3$ representation space $A(M)\otimes\Q_p$ is the two-dimensional irreducible representation of $S_3$.  We have the natural homomorphism $x \mapsto x_v$ of $A(M)\otimes\Q_p$ onto the one-dimensional $\Q_p$-vector space $A(M_v)\otimes\Q_p$.   Let $\alpha \in A(M)/{\rm torsion} \subset A(M)\otimes\Q_p$ be a nontrivial element fixed by $\sigma$ and let $\beta:= \tau(\alpha)$. Neither of the elements $\alpha_v, \beta_v \in   A(M)\otimes\Q_p$ vanish. We can therefore take the ratio $\alpha_v/\beta_v \in \Q_p$, this ratio being independent of the initial choice of $\alpha$ (subject to the conditions we imposed).  Is this ratio  $\alpha_v/\beta_v$ a transcendental $p$-adic number?

\medskip


\subsection{The Strong Leopoldt Condition for complex $A_4$-extensions of $\Q$}

  One gets a slightly different slant on some features of the 
Strong Leopoldt Condition if one considers the case where  $M/\Q$ 
is Galois with $G = A_4$, the alternating group on four letters.
Suppose, again, that $p$ 
splits completely in $M$, so the decomposition group $D$ is trivial, 
and suppose that $M$ is totally complex. Let $E$ be an algebraic closure of $\Q_p$, 
and $F \subset E$ an algebraic extension of $\Q$ contained in $E$.  
Given a preferred $F$-structure,
the $E[G]$-module $U_\loc$
then admits an isomorphism  $U_\loc \simeq F[G]\otimes_FE$, 
and any change of preferred $F$-structure amounts to a scalar 
shift --- multiplication by a nonzero element of $E$.  We have a canonical isomorphism of  $F[G]$-modules
  $$U_{\loc, F} \simeq F \oplus F \oplus F \oplus V_\eta\otimes_FV_\eta^{\#},$$ 
where the action of $G$ on  the first three summands is via the three one-dimensional representations of $A_4$, while the fourth summand, $ V_\eta\otimes_FV_\eta^{\#}$, is the isotypic component of the (unique) three-dimensional representation $V_\eta$; so $V_\eta^{\#}$ is again of dimension $3$.  Put, as usual,  $V_{\eta, E}: = V_\eta\otimes_FE$.
The $E[G]$-module $U_\glob$ is isomorphic to $ E \oplus E\oplus V_{\eta, E}$  where the action of $G$ on  the  first two summands is via the two nontrivial one-dimensional characters. Here too, the homomorphism $\iota_p: U_\glob \to U_\loc$ 
is injective  (i.e., the classical Leopoldt Conjecture is true since the $G$-representation space $U_\glob $ contains only representations with multiplicity one). We thus have an injection $\iota_p: U_\glob \to U_\loc$ identifying   $U_\glob$ with the  direct sum 
$$ U_\glob\ \   = \ \ I\ \bigoplus\  \{V_{\eta, E}\otimes_Ev^{\#}\} \ \subset\ U_\loc,$$
with $v^{\#} \in V_{\eta, E}^{\#}$  as discussed in subsection~\ref{Leopshan}.

\begin{corollary} Let $M/\Q$ be as in this subsection. The weak $p$-adic Schanuel Conjecture implies the Strong Leopoldt Conjecture for $M/\Q$ relative to $E/F$.
\end{corollary}

 Concretely, here is what is involved. Let $K$ be any non-totally real extension of $\Q$ of degree $4$ whose Galois closure $M/\Q$ is an $A_4$-field extension. Then for this field  we have $r_1= 0$ and $r_2=2$, so the rank of the group of units is $1$. Let $u$ be any unit in $K$  that is not a root of unity. 
Since $u\in K \subset M$ is fixed by an element of order three in $\Gal(M/\Q)$,
 it has precisely four conjugates   $u=u_1,u_2,u_3,u_4 \in  M$  (with $u_1u_2u_3u_4=1$). Let $v$ be a place of $M$ of degree one, and let $p$ be the rational prime that $v$ divides. By means of the embedding $M  \subset M_v =\Q_p$ we view those four conjugates as elements of $\Q_p^*$ and consider the four $p$-adic numbers
$$ \log_p(u_1), \log_p(u_2),\log_p(u_3),\log_p(u_4) \in \Q_p.$$ The sum of these four numbers is $0$, and the weak $p$-adic Schanuel Conjecture applied to  the numbers $ \log_p(u_1), \log_p(u_2),\log_p(u_3)$ implies the Strong Leopoldt Condition by Theorem~\ref{weakStrong}.

\bigskip

 {\bf Remark. } It would be interesting to work out the relationship between the weak $p$-adic Schanuel Conjecture and our Strong Leopoldt Conjecture, especially in some of the other situations where the Classical Leopoldt Conjecture has been proven. Specifically, consider the cases where $M/\Q$ is a (totally) complex Galois extension with Galois group $G$ such that for every irreducible character $\chi$ of $G$ we have the inequality $$(\chi(1)+\chi(c))(\chi(1)+\chi(c) -2) < 4 \chi(1).$$ Then (cf. Corollary 2 of section 1.2 of \cite{Lau}) the Classical Leopoldt Conjecture holds for the number field $M$.  This inequality is satisfied, for example, for certain $S_4$-extensions, and $\GL_2({\bf F}_3)$-extensions, of $\Q$.  See also Theorem 7 of \cite{Roy1} and conjectures formulated by Damien Roy in  \cite{Roy2}. We are thankful to Michel Waldschmidt and Damien Roy for very helpful correspondence regarding these issues.


\section{Descending Group Representations}
\label{section:descendinggrouprepresentations}

Let $E$ be an algebraically closed
field of characteristic zero.  
If $V$ is a finite dimensional
vector space over $E$, we will denote by $PV$ the associated 
projective space over $E$. Denote by $\GL(V)$
the  group $\Aut_E(V)$
of $E$-linear automorphisms of $V$, and  by $\PGL(V)$
the group $\Aut_E(PV)$ of linear projective automorphisms of $PV$ over $E$.

\medskip

Let $A$ be a  group.  The automorphism group $\Aut(A)$  of $A$
acts on the set of homomorphisms of $A$ to any group by the standard 
rule: $\iota \circ \phi(a):= \phi\cdot\iota^{-1}(a)$. 

\subsection{Projective Representations}

We start this section by considering \emph{projective} representations
of $A$ on $E$-vector spaces; more precisely: equivalence classes $\Phi$ of homomorphisms
$$\phi: A \rightarrow \PGL(V)$$ for $V$ a finite dimensional $E$-vector space, where equivalence means up to conjugation by an element in $\PGL(V)$.
Note that the representation $\Phi$ need not necessarily lift
to a (linear) representation of $A$ on $V$.

\medskip

The action of the automorphism group
$\Aut(A)$ on projective representations of $A$ 
factors
through the quotient $\Out(A) = \Aut(A)/\Inn(A)$, 
the group of \emph{outer automorphisms} of $A$.
 
\begin{definition} If $\Phi$ is a  projective 
representation of $A$, let $\Aut_{\Phi}(A) \subseteq \Aut(A)$
 denote the subgroup of automorphisms of $A$ that preserve $\Phi$.
Let $\Out_{\Phi}(A) \subseteq \Out(A)$ denote the image
of $\Aut_{\Phi}(A)$.
\end{definition}

\begin{lemma} Let $\phi:A \to \PGL(V)$ be a  homomorphism for which the associated projective representation $\Phi$ is irreducible. Then there is a unique
homomorphism $\twphi: \Aut_{\Phi}(A) \rightarrow \PGL(V)$ such that the composition
$$A \rightarrow \Inn(A) \rightarrow \Aut_{\Phi}(A)
 {\stackrel{\twphi}{\longrightarrow}}\PGL(V).$$
 is equal to $\phi$.
\label{lemma:extend}
\end{lemma}

\begin{proof}  
If $\alpha \in \Aut_{\Phi}(A)$, there is a unique intertwining
element $\twphi(\alpha)  \in \PGL(V)$ such 
that   $\alpha\circ \phi  =  \twphi(\alpha)\cdot \phi \cdot 
\twphi(\alpha)^{-1}$. Uniqueness comes from Schur's lemma,
given that $\Phi$ is an irreducible projective representation.
One checks that $\twphi$ is a homomorphism.  
If $a \in A$ then  $\phi(a)\in  \PGL(V)$ is ``a" (hence ``the") intertwiner 
for the automorphism given by conjugation by $a$, and hence $\twphi$ extends 
$\phi$.
\end{proof}

Note that if $\Phi$ is an irreducible projective representation of
$A$ on $PV$ then the center of $A$ acts trivially on $PV$ by Schur's lemma.

\label{section:projective}
\subsection{Linear Representations}
\label{section:linear}

In this section we consider some of the analogs 
of the results and constructions in section~\ref{section:projective}
for linear representations.
By linear representations (over $E$) we
mean equivalence classes $\Psi$ of homomorphisms
$$\psi:A \rightarrow  \GL(V).$$

\medskip

The action of $\Aut(A)$ and $\Out(A)$ on linear representations
of $A$ is compatible with the operation that associates
to any linear representation $\Psi$ its corresponding
projective representation $\Phi$.

\begin{definition} If $\Psi$ is a linear
representation of $A$, let $\Aut_{\Psi}(A) \subseteq \Aut(A)$
 denote the subgroup of automorphisms of $A$ that preserve the
associated linear representation $\Psi$.
Let $\Out_{\Psi}(A) \subseteq \Out(A)$ denote the image
of $\Aut_{\Psi}(A)$. 
\end{definition}

\medskip
 
Given  a group $B$ and a homomorphism $\phi:B  \to 
\PGL(V)$, the  pullback of $\phi$
 with respect to the exact sequence 
$$0 \to E^{\times} \to \GL(V) \to \PGL(V)\to 0$$ 
yields a central extension of $B$ by $E^{\times}$, 
and therefore a ``factor system," and,
in particular,
a well-defined class $\lambda(\phi) \in H^2(B,E^{\times})$ whose 
vanishing is equivalent to the existence of a linear representation 
lifting the representation $\phi$. 
 If $A \subseteq B$ is a normal subgroup such that the homomorphism 
$\phi$ restricted to $A$ lifts to 
a  homomorphism  $\psi: A \to \GL(V)$,
 then $\lambda(\phi)$ arises from a class in $H^2(B/A, E^{\times})$. 

\medskip 
 
 \begin{lemma}
\label{secondhalfoflemmaone}  
Let $A$ be a normal subgroup of $B$, and
let $\Psi$ be an irreducible linear representation of $A$.
Suppose that the natural
homomorphism $B  \to \Aut(A)$ induced from conjugation
has the property that its image lies in  $\Aut_{\Psi}(A)$.
 Then if the quotient group $B/A$ is finite cyclic,  
  the linear
representation $\Psi$ of $A$ on $V$ extends to a 
representation $\twPsi$ of $B$ on $V$ unique up to a character of $B/A$.
\label{lemma:extend2}
\end{lemma}

\begin{proof} 

Let $\Phi$ denote the irreducible projective
representation associated to $\Psi$, and let
 $\phi: A \rightarrow
\PGL(V)$ be a homomorphism in the class of $\Phi$. 
There is a natural inclusion of groups $\Aut_{\Psi}(A) \subseteq \Aut_{\Phi}(A) $.
Thus,
by Lemma~\ref{lemma:extend},
  the map $\phi$ factors as the 
composition $$A \rightarrow \Inn(A) \rightarrow \Aut_{\Phi}(A)
 {\stackrel{\twphi}{\longrightarrow}}\PGL(V).$$ Composition of  $B \rightarrow \Aut_{\Psi}(A) \subseteq \Aut_{\Phi}(A) $ with $\phi$  defines a map $\twphi: B 
\rightarrow \PGL(V)$. By Lemma~\ref{lemma:extend}
$\twphi$ extends $\phi$. 
 Suppose that $B/A$ is finite cyclic of order $n$. 
Let $\sigma \in B$ project to a generator of $B/A$. 
We begin by defining $ \twpsi(\sigma) \in \GL(V)$ to be a lifting of 
$ \twphi(\sigma) \in \PGL(V)$ 
with the property that $\twpsi(\sigma)^n = \psi(\sigma^n)$.
Such a choice is possible because $E$ is algebraically closed.
Moreover, such an equation is clearly necessary if $\twpsi$ is
to extend $\psi$; there are precisely $n$ such choices,
each differing by $n$-th roots of unity.
Each element of $B$ can be written uniquely 
in the form $\sigma^i x$ for $0 \le i < n$ and $x \in A$, and we define
$$\twpsi(\sigma^i x) = \twpsi(\sigma)^i \psi(x),$$
as we are forced to do if we wish $\twpsi$ to be a homomorphism
extending $\psi$.
We first observe that the identity above holds for \emph{all} $i$ by induction,
since 
$$\twpsi(\sigma^{i+n} x) = \twpsi(\sigma^i \sigma^n x)
= \twpsi(\sigma)^i \psi(\sigma^n x) = \twpsi(\sigma)^i \psi(\sigma^n)
\psi(x) = \twpsi(\sigma)^{i+n} \psi(x).$$

\begin{sublemma} There is an equality
$\psi(\sigma^{-1} x \sigma) = \twpsi(\sigma)^{-1} \psi(x) \twpsi(\sigma)$.
\label{sublemma:sub}
\end{sublemma}

\begin{proof} By assumption, conjugation of $A$ by $\sigma$ acts as
an element of $\Aut_{\Psi}(A)$, and thus the homomorphisms
$\psi$ and $\psi^{\sigma}$ (i.e., $\psi$ conjugated by $\sigma$)
 are related by an intertwiner in $\PGL(V)$ 
unique up to scalar multiple.
The same argument applied to the projectivization shows that
$\phi$ and $\phi^{\sigma}$
(i.e., $\phi$ conjugated by $\sigma$) are related by a unique intertwiner,
which is the projectivization of the corresponding linear one.
Yet by construction $\twpsi(\sigma)$ is the corresponding intertwiner, and
thus the equality
$\psi(\sigma^{-1} x \sigma) = \twpsi(\sigma)^{-1} \psi(x) \twpsi(\sigma)$
holds for \emph{any} lift $\twpsi(\sigma)$ of $\twphi(\sigma)$.
\end{proof}

\medskip

To show that $\twpsi$ is a homomorphism we are required to show that
$\twpsi(\sigma^p x \cdot \sigma^q y) = \twpsi(\sigma^p x) \twpsi(\sigma^q y)$.
Now
$$\twpsi(\sigma^p x \cdot \sigma^q y) = \twpsi(\sigma^{p+q} (\sigma^{-q} x \sigma^q) y)
= \twpsi(\sigma)^{p+q} \psi(\sigma^{-q} x \sigma^q)  \psi(y).$$
Applying Sublemma~\ref{sublemma:sub} $q$-times we may write this as
$$ \twpsi(\sigma)^{p+q} \twpsi(\sigma)^{-q} \psi(x) \twpsi(\sigma)^{q}
\psi(y) = \twpsi(\sigma)^{p} \psi(x) \twpsi(\sigma)^q \psi(y)
= \twpsi(\sigma^p x) \twpsi(\sigma^q y),$$
which was to be shown.
To conclude the proof of Proposition~\ref{secondhalfoflemmaone}  
we note that tensoring the homomorphism $\twpsi$ with the $n$ 
distinct characters of $B/A$ provides $n$ distinct homomorphisms 
extending $\psi$; these must be all of them by the discussion above. 
\end{proof}

\medskip

Let $V$ be a two-dimensional $E$-vector space, 
$\Gamma$  a finite group, and $\phi: \Gamma \hookrightarrow \PGL(V)$ an injective homomorphism such that the projective representation $\Phi$ that $\phi$ determines is irreducible  (and visibly: faithful).  Then $\Gamma$ is isomorphic to one of the groups in the collection
$\Sf = \{D_n, n \ \text{odd}, A_4, S_4, A_5\}$. Note that all these groups have trivial center.

If
$W = \Hom'(V,V)$ denotes the hyperplane in $\Hom(V,V)$ consisting of
endomorphisms of $V$ of trace zero, then $\phi$ induces a
\emph{linear} faithful action of $\Gamma$ on $W$.

\medskip

\begin{definition} 
\label{definition:defX}
Let $X \subseteq W$ be the unique vector subspace stabilized by the above action of $\Gamma$ on which 
$\Gamma$ acts 
irreducibly and faithfully. 
Explicitly,
\begin{enumerate}
\item If $\Gamma \simeq A_4,S_4,$ or $A_5$, then $W = X$.
\item If $\Gamma \simeq D_n$, $n$ odd, then
$W = X \oplus \eps$, where 
$\eps$ stands for the one-dimensional $\Gamma$ representation given by  the quadratic character of $\Gamma \simeq D_n$.
\end{enumerate}
\end{definition}

Denote by $\Psi$ the (irreducible, faithful, linear) representation of $\Gamma$ on $X$.

\bigskip

With $V$ a two-dimensional vector space as above, let  us be given $G$  a finite group, $H$ a normal
subgroup of $G$ equipped with an irreducible projective representation $\Phi'$ of $H$ on $PV$.
Let $\Ker \subset H$  be the kernel of this projection representation $\Phi'$, so that $\Phi'$ factors through a {\it faithful} representation, $\Phi$ of $\Gamma:=H/\Ker$ on $PV$. Therefore  $\Gamma$ and $\Phi$ are among the groups and representations classified in the discussion above; moreover, we have a corresponding linear representation $\Psi$ of $\Gamma$  on $X$, as described.  Denote by $\Psi'$ the linear representation of $H$ obtained by composing $\Psi$ with the surjective homomorphism $H \to \Gamma$. 

 Let $N(\Ker)$ be the normalizer of $\Ker$ in $G$, so that we have the sequence of subgroups
$$\{1\} \subset \Ker \subset H  \subset  N(\Ker) \subset G.$$
Conjugation on $H$ induces
a map $$N(\Ker) \rightarrow \Aut(\Gamma).$$

\begin{definition} Let $N_{\Phi}\subseteq N(\Ker)\subseteq G$ be the
pullback of $\Aut_{\Phi}(\Gamma)$  and let $N_{\Psi}\subseteq N(\Ker)\subseteq G$ be the
pullback of $\Aut_{\Psi}(\Gamma)$  under the above displayed homomorphism. 
\end{definition}

\medskip

\begin{lemma} 

We have:
$N_{\Phi} = N_{\Psi}$.
\label{lemma:projad}
\end{lemma}

\begin{proof} It suffices to show that
$\Out_{\Phi}(\Gamma) = \Out_{\Psi}(\Gamma)$ for
$\Gamma \in \Sf$.
We do this on a case
by case basis.
\begin{enumerate}
\item $\Gamma = A_4$, $\Out(A_4) = \Z/2\Z$. There is a unique projective
representation of $A_4$ of dimension $2$  and a unique linear representation
of dimension $3$, and thus $\Out_{\Phi}(\Gamma) = \Out_{\Psi}(\Gamma) = \Z/2\Z$.
\item $\Gamma = S_4$, $\Out(S_4) = 1$. In this case one trivially
has $\Out_{\Phi}(\Gamma) = \Out_{\Psi}(\Gamma) = 1$.
\item $\Gamma = A_5$, $\Out(A_5) = \Z/2\Z$. There are two inequivalent
projective representations of $A_5$ of dimension $2$  which are permuted by the outer
automorphism group. Correspondingly, the outer automorphism group 
permutes the two $3$-dimensional linear representations of $A_5$.
Thus $\Out_{\Phi}(\Gamma) = \Out_{\Psi}(\Gamma) = 1$.
\item $\Gamma = D_n$, $\Out(D_n) = \Out(C_n)/\kern-0.06em{\pm}1
= (\Z/n\Z)^{\times}/\kern-0.06em{\pm}1$.
The outer automorphism group acts freely transitively on the set of
$2$-dimensional representations of $D_n$, and thus neither $\Phi$ nor $X$
are preserved by any outer automorphisms. Hence $\Out_{\Phi}(\Gamma) = \Out_{\Psi}(\Gamma) = 1$.
\end{enumerate}
\end{proof}
In light of this lemma, we denote $N_{\Phi}$ (and correspondingly
$N_{\Psi}$) by $N$.

\medskip

\Remark  In 
section~\ref{subsection:inducedup}, we study inductions of the
(linear) representation $X$ to $G$; in particular, the decomposition
of $\Ind^{G}_{H} X$ into irreducibles with respect to $N = N_{\Psi}$.
The utility of Lemma~\ref{lemma:projad} is that from these calculations
 we are able  to make conclusions about
\emph{projective} liftings of the representation $V$ to $N = N_{\Phi}$,
for example, Lemma~\ref{lemma:descentabove}.

\medskip

We note the following consequence of Lemma~\ref{lemma:projad}.

\begin{lemma}\label{lift} The projective representation $\Phi'$ of $H$ lifts to a projective
representation of $N$. The linear representation $\Psi'$ of $H$ lifts to a linear
representation of $N$.
\end{lemma}

\begin{proof} Since the corresponding morphisms
$N \rightarrow \Aut_{\Phi}(\Gamma)$ and $N \rightarrow \Aut_{\Psi}(\Gamma)$
map $H$ to $\Inn(\Gamma) = \Gamma$, it suffices to show that
the representation $\Phi$ 
and $\Psi$ considered as representations
of $\Gamma$ lift to $\Aut_{\Phi}(\Gamma)$ and $ \Aut_{\Psi}(\Gamma)$, respectively.
The claim for $\Phi$ follows from Lemma~\ref{lemma:extend}, and the
claim for $\Psi$ follows from Lemma~\ref{lemma:extend2}, and the fact
that $\Out_{\Phi}(\Gamma)$ has order one or two for all $\Gamma$ and
is thus cyclic.
\end{proof}

\medskip

\Remark Note that the action on $N$ factors through $\Inn(\Gamma) = \Gamma$
in all cases with the possible exception of $\Gamma = A_4$, when
the action of $N$ may factor through $S_4$ or $A_4$, depending whether
the map $N(\Ker) \rightarrow \Out(\Gamma) $ is  surjective
or trivial.

\medskip

\Remark
It should be noted that the proof of
Lemma~\ref{lemma:projad} is not a formal consequence of any general
relation between irreducible projective representations and
the linear representations corresponding to their adjoints.
In particular, $N_{\Phi} \ne N_{\Psi}$ for a general
$n$-dimensional projective representation $\Phi$ and
its corresponding linear adjoint representation $\Psi$
of dimension $n^2 - 1$, even if the latter representation is
irreducible.
For example, the group $\Gamma = A_6$ admits four inequivalent
faithful irreducible projective representations of dimension $3$, which as a set
are acted upon freely transitively by $\Out(A_6) = \Z/2\Z \oplus \Z/2\Z$.
Thus, if $PV$ is any $3$-dimensional irreducible
projective representation of $A_6$, then $\Out_{\Phi}(A_6) = 1$.
On the other hand,
$A_6$ has only two inequivalent irreducible $8$-dimensional
representations. Since one of them 
is the linear representation $\Psi$ of $A_6$ on
$\Hom'(PV,PV)$,  there exists at least
one outer automorphism of $A_6$ which preserves $\Psi$, and hence
$\Out_{\Psi}(A_6) = \Z/2\Z$.
In particular, there exists an 
$8$-dimensional representation of $A_6.2$ lifting $\Psi$ which does not
arise from the adjoint of a three dimensional projective
representation. By $A_6.2$ we mean the extension corresponding
to an ``exotic'' automorphism of $A_6$ rather than 
$S_6$
($S_6$ admits no eight dimensional
irreducible representations).

\subsection{Representations Induced  from Normal Subgroups}
\label{subsection:inducedup}

Let $\Gamma \in \Sf$, and let $X$ be the representation space
of definition~\ref{definition:defX} and, as in the previous section, we let $\Psi$ denote the ($E$-linear, irreducible, faithful) representation of $\Gamma$ on $X$.
Let $G$ be a finite group, and $H$ a normal subgroup of $G$
equipped with a surjective homomorphism to $\Gamma$. View the representation space $X$ as a representation space for
 $H$ via this surjective homomorphism, and --- 
as in the previous section --- denote by $\Psi'$ the ensuing representation of $H$ on $X$.
For a given $g \in G$, let $\Xg:=gX$ denote the vector
space whose vectors are written $\{gx\ | x\in X\}$ and such that the natural mapping $X \to gX$  ($x\mapsto gx$) is an isomorphism of $E$-vector spaces. We define a representation ``$\Psi^g$" of $H$ on $\Xg$  by composing the $\Psi$-action of $H$  on $X$ with the automorphism of $H$ given by
conjugation by $g$.  Specifically, for $h \in H$ and $gx$ an element of $\Xg$ the action is given by the rule: $h {\star} gx:= gx'$ where $x':= (g^{-1}hg)\cdot x \in X$ for $x \in X$ and $h \in H$. In the special case when $g=1$, the space  $\Xg$   is identified with $X$.

Let ${\mathcal R} \subset G$ be a representative system for left cosets of $H$ in $G$
and  let $g \in G$.  We have an isomorphism of $E[G]$-modules:  $$\Ind_H^G\Xg= E[G]\otimes_{E[H]}\Xg \ {\stackrel{\simeq}{\longrightarrow}}\ \bigoplus_{\gamma\in {\mathcal R}}\gamma\Xg$$ where the action of $G$ on the right-hand side is given by the evident formula.  In particular, we have:   $$\Ind_H^GX \ {\stackrel{\simeq}{\longrightarrow}}\ \bigoplus_{\gamma\in {\mathcal R}}\gamma X$$ where, again, the action of $G$ is given by the evident formula; the action of $H$ --- in contrast --- 
has the feature of preserving the direct sum decomposition on the right-hand side of the display above, the representation of $H$ on the summand $\Xg$ being $\Psi^g$ as described above.

\begin{lemma} Let $V \subseteq \Ind^{G}_{H} X$ be
an irreducible sub-representation, and let
$g$ be any element of $G$. We have natural isomorphisms of $E$-vector spaces
$$\Hom_H(V,\Xg) {\stackrel{\simeq}{\longrightarrow}} \Hom_G(V,\Ind^{G}_{H} \Xg)
 {\stackrel{\simeq}{\longrightarrow}}\Hom_G(V,\Ind^{G}_{H} X) \ne 0.$$
\label{lemma:frob}
\end{lemma}

\begin{proof} The first isomorphism comes from  Frobenius
reciprocity; the second  is induced from the inverse to the natural $E$-vector space isomorphism $X \simeq \Xg$ that intertwines the $H$ representation $\Psi$ on $X$ to $\Psi^g$ on $\Xg$. 
\end{proof}
\Remark Comparing the isomorphisms above for different elements $g, g' \in G$ gives us canonical isomorphisms of $E$-vector spaces
  $$\Hom_H(V,\Xg) {\stackrel{\simeq}{\longrightarrow}} \Hom_G(V,\Ind^{G}_{H} \Xg)
 {\stackrel{\simeq}{\longrightarrow}}\Hom_H(V,X_{g'}) {\stackrel{\simeq}{\longrightarrow}} \Hom_G(V,\Ind^{G}_{H} X_{g'}).$$ 

Recall the group $N$ that fits into the sequence $$\{1\} \subset \Ker  \subset H  \subset N  \subset G$$ as discussed in the previous subsection.

\begin{lemma}
\begin{enumerate}
\item  Let $U \subseteq \Ind^{N}_{H} X$
be an irreducible representation of $N$.
Then $\Ind^{G}_{N} U$ is an irreducible
representation of $G$.
\item
Let $V \subseteq \Ind^{G}_{H} X$
be an irreducible representation of $G$.  Then there is an irreducible representation $U \subseteq \Ind^{N}_{H} X$ of $N$ and an isomorphism $$ V {\stackrel{\simeq}{\longrightarrow}} \Ind^{G}_{N} U$$ of $G$-representations. (I.e., {\it all irreducible representations in $\Ind^{G}_{H} X$ ``come from" irreducible sub-representations of $N$ acting on $\Ind^{N}_{H} X$.})
\end{enumerate}
\label{lemma:irredform}
\end{lemma}

\begin{proof} By definition, $\Ind^{N}_{H} X$ restricted
to $H$ is isomorphic to
$\bigoplus_{g \in {\mathcal R}(N/H)} \Xg \simeq X^{[N:H]}$ where ${\mathcal R}(N/H) \subset N$ is a representative system of $H$-cosets.
Thus $\dim(U) = \dim(X) \cdot \dim \Hom_{H}(X,U)$.
There is at least one irreducible representation $V$
contained in $\Ind^{G}_{N} U$ such that $\Res_{N}V$ contains a copy of $U$.
By Lemma~\ref{lemma:frob} it follows that
$$\dim \Hom_{H}(V,\Xg) = \dim \Hom_{H}(V,X)
\ge \dim \Hom_{H}(U,X)$$
for all $g$, and thus as the number of distinct representations
$\Xg$ is equal to $G/N$,
$$\dim(V) \ge [G:N] \dim(U) = \dim(\Ind^{G}_{N} U),$$
from which it follows
that $V = \Ind^{G}_{N} U$ is irreducible.
On the other hand, if $V \subseteq \Ind^{G}_{H} X$ is
irreducible, then $U \subseteq V|_N$ for some irreducible
$N$-representation $U$, and then $\Hom_G(V,\Ind^{G}_{N}U)$
is non-zero and $V = \Ind^{G}_{N}U$.
\end{proof}

\subsection{Elements of order two}

 As in the previous subsection, let $G$ be a finite group, $H \subset G$ a normal subgroup equipped with a homomorphism onto $\Gamma$ with kernel denoted $\Ker$ and let  $\Ker \subset H \
\subset N \subset G$ be as described previously; we also have the faithful irreducible projective representation $\Phi$ of $\Gamma$ on $V$, and the corresponding  linear representation $\Psi$ of $\Gamma$ on $X$.  By an {\it involution} in a group we mean an element of order two.
 
\begin{definition} An involution $c \in N \subset  G$ will be called \emph{\even} if $c$ lies in the kernel
of the map $N \longrightarrow \Aut(\Gamma)$.
\end{definition}

Let $\mathcal C \subset G$ be the conjugacy class 
of an involution in $G$; thus if $c \in \mathcal C$, $\mathcal C =\{gcg^{-1}\ | \ g\ \in G\}$.

\begin{lemma} Suppose that:
\begin{enumerate}
\item No $c \in \mathcal C$ is  \even.
\item There exists at least one $c \in  \mathcal C$ such that
$c \notin N$.
\end{enumerate}
Then if $V \subseteq \Ind^{G}_{H} X$ is an irreducible
representation then:
\begin{enumerate}
\item If $\Gamma = D_n$,
$\dim(V|c = 1) = \dim(V|c = -1) = \frac{1}{2} \dim(V)$.
\item If $\Gamma \in \{A_4,S_4,A_5\}$,
$\dim(V|c = -1) < \frac{2}{3} \dim(V)$.
\end{enumerate}
\label{lemma:twothirds}
\end{lemma}

\begin{proof}
Note that $\dim(V|c=-1)$ does not depend on the choice
of $c \in \mathcal C$, and so we assume that $c \notin N$.
Since $H$ is normal in $G$, $H$ has index two
in $\langle H,c \rangle$. Let $U$ be
an irreducible constituent of $\Res_{\Hc} V$.
The representation $\Res_{H} U$ is
a product of conjugates $\Xg$ of $X$ for some
collection of elements $g \in G$ (not
necessarily distinct). By Frobenius reciprocity
it follows that $U$ injects into
$\Ind^{\Hc}_{H} \Xg$ for some $g$.
If $\Ind^{\Hc}_{H} \Xg$ is irreducible, then
$$\dim(U|c=1) = \dim(U|c=-1),$$ and we are done.
If $\Ind^{\Hc}_{H} \Xg$ is reducible, then it
decomposes as a direct sum of two representations
each of dimension $\dim(X) \in \{2,3\}$.
The relations
$\dim(U|c=-1) = \frac{1}{2} \dim(U)$ if $\dim(X) = 2$ or
$\dim(U|c=-1) \le \frac{2}{3} \dim(U)$ if $\dim(X) = 3$
are then equivalent to the condition that
$c$ does not act as $+1$ or $-1$ on $U$.  
Now
$$U = \Res_{H}(U) \subset \Res_{H} \Ind^{\Hc}_{H} \Xg
= \Xg \oplus c X_{g} = \Xg \oplus X_{cg}.$$
Since $U$ is clearly neither $\Xg$ nor $c \Xg$ (as they
are not preserved by $c$) it follows that there
must be an isomorphism of $H$ representations
$\Xg \simeq X_{cg}$, or equivalently, an isomorphism of
$H$ representations $X \simeq X_{g^{-1} c g}$. It follows that
$g^{-1} c g \in N$.
If, in addition, $c$ acts centrally on $U$, then the irreducible
$\langle H,c \rangle$ subrepresentations of $X_g \oplus c X_g$
are given explicitly by the diagonal and anti-diagonal subspaces.
The corresponding isomorphism
$X \rightarrow X_{g^{-1} c g}$ is therefore the identity. It follows
that the image of $g^{-1} c g$ in $\Aut(\Gamma)$ is trivial, and
hence $g^{-1} c g$ is \emph{\even}, a contradiction.

\medskip

To prove the \emph{inequality} in the case
when $\dim(X) = 3$, it suffices to show that
there exists at least one representation
$U \subseteq \Res_{\Hc} V$ with $U = \Ind^{\Hc}_{H} \Xg$,
since then $\dim(U|c=-1) = \frac{1}{2} \dim(U)
< \frac{2}{3} \dim(U)$. For this, note that there
is at least one  $U$ such that $X \subseteq \Res_{H} U$.
Yet if $\Ind^{\Hc}_{H} X$ is reducible, then by arguing as
above we deduce that $c \in N$.
By assumption $c \notin N$ and we are done.
\end{proof}

\section{Artin Representations}
\label{section:artinrepresentations}

Fix a Galois extension $K/\Q$, an algebraically closed extension $E/\Q_p$,
and an
irreducible representation
$$\rho: G_{K} \rightarrow  \GL_2(E)$$
with finite image. Let $\Proj(\rho)$ denote the associated
two dimensional projective representation, and
$\Ad^0(\rho)$ the associated three dimensional linear representation.
Let $L$ denote the fixed field of the kernel of $\Proj(\rho)$, which
is also the fixed field of the kernel of $\Ad^0(\rho)$.
By a well known classification, the image of $\Proj(\rho)$
belongs to the set $\Sf: =
\{A_4,S_4,A_5,D_n, n \ge 3, \text{odd}\}$.
We make the following assumptions about the fields $K$ and $L$.
\begin{enumerate}
\item $p$ is totally split in $K$, and $L/K$ is unramified at all
primes above $p$.
\item If $v|p$, the eigenvalues $\alpha_v$, $\beta_v$ of $\rho(\Frob_v)$ are distinct
($\rho$ is $v$-distinguished).
\end{enumerate}
Fix once and for all an ordering of the pair $\alpha_v$, $\beta_v$, and let
$L_v = \alpha_v E$.

\medskip

\begin{definition} Say  that the representation $\rho$ \emph{descends to
$K^{+} \subseteq  K$}  if and only if
there exists a Galois representation
$\varrho : G_{K^{+}} \rightarrow \GL_2(E)$  such that
$\varrho |_{G_{K}} = \rho \otimes \chi$ where $\chi$ is a character of $K$.
Similarly, say $\Proj(\rho)$ descends to $K^{+}$ if it is the restriction
of some two dimensional projective representation of $G_{K^{+}}$.
\end{definition}

\Remark The image of $\varrho$ can be strictly larger than
the image of $\rho$, as can be seen by taking $\varrho$ to be a
$G_{\Q}$-representation with projective image $S_4$ and taking
$K$ to be the corresponding quadratic subfield.

\medskip

The goal of this section is to give the proof of
the main theorem:

\begin{theorem}\label{mainth} Suppose that $\rho$ admits infinitesimally classical deformations
of minimal level.
Assume the strong Leopoldt conjecture.
Then one of the following holds:
\begin{enumerate}
\item There exists a character $\chi$ such that
 $\chi \otimes \rho$ descends to an odd representation over a  totally real field.
\item The projective image of $\rho$ is dihedral. The determinant character
descends to a totally real field $H^{+} \subseteq K$ with corresponding fixed field $H$
such that
  \begin{list}{\kern-0.2em{$($}\roman{Lcount}\kern+0.1em$)$}
    {\usecounter{Lcount}
    \setlength{\rightmargin}{\leftmargin}}
  \item $H/H^{+}$ is a CM extension.
  \item At least one prime above $p$ in $H^{+}$ splits in $H$.
  \end{list}
\item The representation $\chi \otimes \rho$ descends to a field
containing at least one real place at which $\chi \otimes \rho$
is even.
\end{enumerate}
\end{theorem}

\subsection{Notation}

This notation
will be used for the remainder of the paper.
Recall that $L$ is the fixed field of the kernel of $\Proj(\rho)$.
Let $M$ be the Galois closure of $L$, and let $G = \Gal(M/\Q)$.
Let $H = \Gal(M/K)$. The group $H$ is normal in $G$ by construction.
Let $D$ be the decomposition group at some fixed prime $v$ dividing $p$.
The representation $\Proj(\rho)$ gives rise to a projective representation $\Phi$ of $H$,
that factors through  a group $\Gamma \in \Sf$. 
Recall that $N = N_{\Phi}$ is the group defined as in 
Lemma~\ref{lemma:projad}.  Let $K^{+}$ be the fixed field of $N$, so that
\begin{lemma}
 $\Proj(\rho)$ descends to $K^{+}$.
\label{lemma:descentabove}
\end{lemma}

\subsection{Lifting projective representations}

Since the deformation theory of $\rho$ only depends on
$\Proj(\rho)$, we turn our attention to projective representations.
Recall the following theorem of Tate.

 \begin{theorem}[{\bf Tate}]\label{tateproj} Let $K$ be a number field and $E$ an  algebraically closed field. Any continuous homomorphism $G_K \to \PGL_n(E)$ lifts to a continuous homomorphism $G_K \to \GL_n(E)$. 
 \end{theorem}
 
 \begin{proof} Tate's result will follow if we show that $H^2(G_{K},E^*)$ vanishes (here the action of $G_K$ on $E^*$ is trivial). Since $E$ is algebraically closed, the multiplicative group $E^*$ is isomorphic to an extension of a uniquely divisible group by a group isomorphic to $\Q/\Z$, so $H^2(G_{K},E^*) = H^2(G_{K},\Q/\Z) = H^3(G_{K},\Z)$. But the latter group vanishes. For this, see page 77, Cor. 4.17 in \cite{Milne}, where it is shown that  $H^r(G_{K},\Z)$ vanishes for all number fields $K$ and all odd values of $r$.
\end{proof}

\begin{lemma} The representation $\chi \otimes \rho$ descends to $K^{+}$
for some character $\chi$.
\end{lemma}

\begin{proof} By Lemma~\ref{lemma:descentabove},
%
%
the representation
 $\Proj(\rho)$ can be descended to a representation:
$$\Proj(\rho)^+: G_{K^{+}} \rightarrow \PGL_2(E),$$
then by Theorem~\ref{tateproj}, the representation $\Proj(\rho)^+$ lifts
to a linear representation $\varrho^+$ of $G_{K^{+}}$.   Since
any two lifts of a projective representation are equivalent up
to a character, it follows that $\varrho: = \varrho^+|_{G_K} = \rho \otimes \chi$, and
the result follows.
\end{proof}

\subsection{Complex conjugation in $G$}

Let ${\mathcal C}$ be the conjugacy
class of complex conjugation elements of $G$.

\begin{proposition}  One of
the following  conditions hold:
\begin{enumerate}
\item Every $c \in {\mathcal C}$ is not \even,
\item There exists at least one $c \in {\mathcal C}$ such that
$c \notin N$,

\item The projective representation $\Proj(\rho)$ descends to a representation over a field
containing at least one real place which is even,
\item The projective representation $\Proj(\rho)$ descends to an odd representation
over a totally real field.
\end{enumerate}
\end{proposition} 

\begin{proof}
Suppose that every $c \in {\mathcal C}$ lies in $N$.
Then the fixed field $K^{+}$ of $N$ is totally real, and
$\Proj(\rho)$ descends to $K^{+}$ (and is either odd or
even at each place). If $c$ is \even, then consider the
descended representation $\varrho = \chi \otimes \rho$ at the corresponding
real place. Since $c$ acts trivially by conjugation on
$\Gamma$, it follows that $c$ acts centrally on this representation,
and thus via $\pm 1$ (by Schur's lemma). Hence, possibly after a twist, $\varrho$ is even
at this place.
\end{proof}

\subsection{Selmer Groups of Artin Representations}
\label{subsection:selmer2}

Recall (from section~\ref{section:theleopoldtconjecture}) the following notation:

\begin{itemize}
\item
 the $E[G]$-module formed from {\it global units} of $M$:  $U_{\glob}:= {\mathcal O}_M^*\otimes_{\Z}E$; 
 \item
 the $E[D_v]$-module formed  from {\it local units}  of $M_v$ for places $v$ dividing $p$:
$U_v:= {\widehat{\mathcal O}}_{M_v}^*\otimes_{\Z_p}E;$
\item
 and  the natural homomorphism
$$\iota=\iota_p = \prod_{v|p}\iota_v:  U_\glob \longrightarrow  \prod_{v|p}U_v=\oplus_{v|p}U_v=  U_{\loc},$$
which the classical Leopoldt conjecture asserts is injective. 

\end{itemize} 
Assume the classical Leopoldt conjecture, and form the exact sequence
$$0 \to  U_\glob\ {\stackrel{\iota_p}{\longrightarrow}}\  U_{\loc} \to \Gamma_M \to 0.$$ The $E$-vector space  $\Gamma_M$ can be identified, via Class Field Theory,  with the tensor product of the Galois group of the maximal $\Z_p$-power extension of $M$ over $\Z_p$ with $E$. 

   Let $F \subset E$  be fields as in subsection~\ref{prelim}, and let $\Lambda_F:=F[G]$ and $\Lambda:=\Lambda_F\otimes_FE$.
   Make the identification of  left-$\Lambda$-modules $\Lambda \simeq U_\loc$  (compatible with preferred $F$-structures; see subsection~\ref{prelim}).  Consequently, given the classical Leopoldt conjecture,  $U_\glob$ will be identified with some left ideal $I_{\glob}  \subset \Lambda$, and $\Gamma_M $ above will be identified with the quotient  $\Lambda/I_{\glob} $.  After these identifications, and setting $I:=I_{\glob}$, the exact sequence displayed above may be (cryptically) written:
   $$0 \to  I\ \to \  \Lambda \to \Lambda/I \to 0.$$

Recall (from subsection~\ref{ordsel} that  $W:=\Hom'(V,V) \subset \Hom(V,V)$ denotes the $G_K$-stable hyperplane in $\Hom(V,V)$ consisting of endomorphisms of $V$ of trace zero. By the inflation-restriction sequence there is an isomorphism
$$H^1(K,W) \simeq H^1(M,W)^{H} = H^1(M,\Hom'(V,V))^{H}.$$
Let $\HS$ denote the Selmer group without any local conditions at $p$.
Since we may consider minimal deformations $\Sigma = \emptyset$
(see Corollary~\ref{corollary:minimallevel}, and the remarks following
Definition~\ref{definition:ICD}),
the group $\HS$ is therefore isomorphic to
$$\Hom_E(\Gamma_{M},\Hom'(V,V))^{H} \simeq \Hom_E(\Lambda/I,\Hom'(V,V))^{H},$$
 the right-hand displayed module coming from the identification made above. The
away from $p$ Selmer group $\HS$ fits into an exact sequence
$$0 \to \HS \to \Hom_E( U_{\loc} ,\Hom'(V,V))^{H} \to \Hom_E( U_\glob,\Hom'(V,V))^{H},$$ 
which --- after our identifications --- can be written
$$0 \to \Hom_E(\Lambda/I,\Hom'(V,V))^{H} \to \Hom_E( \Lambda,\Hom'(V,V))^{H} \to \Hom_E( I,\Hom'(V,V))^{H}.$$ 
 We have a canonical identification
$$\Hom_E(U_\loc,\Hom'(V,V)) = \Hom_E(\oplus_{v | p}U_v,\Hom'(V,V))= \prod_{v|p} \Hom_E( U_v ,\Hom'(V,V)).$$  For each $v|p$ we have the natural homomorphism $\Hom'(V,V) \rightarrow \Hom(L_v,V/L_v)$.

Recall that $\Win_v:=\Ker(\Hom'(V,V) \rightarrow \Hom(L_v,V/L_v))$, and that we will be using these mappings to cut out the ordinary piece of the Selmer modules. Recall, as well, the local Selmer conditions
$\Lo_v \subseteq
H^1(D_v,W)$ defined in subsection~\ref{ordsel},

\begin{lemma} Let $\gamma \in \HS$. Then $\gamma \in \Lo_v$ for
$v|p$ if and only if the corresponding morphism
$\phi \in \Hom_H(\Lambda,\Hom'(V,V))$
lies in the kernel of the composite map
$$\begin{diagram}
\Hom_H(\Lambda,W) \simeq \Hom_H(U_\loc,W)
& \rTo & \Hom(U_v,W) \\
  & \rdTo & \dTo & \\
 & &
 \Hom(U_v ,W/\Win_v) \end{diagram}$$
Denote the intersection in $\Hom_H(\Lambda,W)$ of these kernels for
all $v$ by $\Hom^0_{H}(\Lambda,W)$.
\end{lemma}

\begin{proof}
 This is simply a matter of tracing the
definition of $\Lo_v$.
\end{proof}

\medskip

\begin{lemma}
The group
$\Hom^0_H(\La,W)$ is 
the subset of homomorphisms $\phi$ such that
$\phi([\sigma]) \in \Win_{\sigma(v)}$ 
for all $\sigma \in G$.
\end{lemma}

By Frobenius reciprocity we may  re-write the
Selmer group $\HS$ as
$$\Hom_{G}(\Lambda/I,\Ind^{G}_{H} W).$$
  Recall that we have a chosen place $v$ of $M$ above $p$, and we have also chosen a representative system ${\mathcal R}(G/H)\subset G$ of left $H$-cosets in $G$. 
\begin{lemma} After the identification we have made, we have that
$$H^1_{\Sigma}(K,W)
= \Hom_G(\Lambda/I,\Ind^{G}_{H} W)
\bigcap \Hom^0_{G}(\Lambda,\Ind^{G}_{H} W),$$
where the adornment $0$ in the second homomorphism group
indicates morphisms $\phi: \Lambda \to \Ind^{G}_{H} W$ such that
$$\phi([1]) \in \bigoplus_{\sigma \in{\mathcal R}( G/H)}\kern-0.6em{\Win_{\sigma(v)}}$$
\end{lemma}

\begin{definition}
Let $Y = \Ind^{G}_{H} W$ and $\Yin
= \bigoplus_{\sigma \in {\mathcal R}(G/H)} \Win_{\sigma}$.
\end{definition}

\begin{corollary}
In the notation of section~\ref{section:ext},
we have a canonical isomorphism
$$H^1_{\Sigma}(K,W) \simeq \Hom_G(\Lambda/I,Y;\Yin).$$
\end{corollary}

If $\gamma \in H^1_{\Sigma}(K,W)$, then the infinitesimal
Hodge-Tate weights are given as follows. Let $\phi$ denote
the corresponding element of $\Hom_{G}(\Lambda,Y)$, so
$\phi([1]) \in \Yin$.
 Recall there is a canonical
homomorphism $\Win_v\ {\stackrel{\epsilon}{\to}}\ \End_E(L_v) = E$.
The composite map
$$\pi: \HS \rightarrow
\Yin \rightarrow \bigoplus_{\sigma \in {\mathcal R}(G/H)} E,$$
is exactly the map of Definition~\ref{definition:HTW}.

\begin{definition} 
\label{definition:zni}
Let
$\Yni$ denote the kernel of the map
$\Yin \rightarrow \bigoplus_{\sigma \in {\mathcal R}(G/H)} E$.
\end{definition}

The spaces
$\Yin$ and $\Yni$ are, by construction, left rational
$D$-submodules of $Y$.
To orient the reader, choosing an appropriate
local basis everywhere,
the spaces $\Yin$ and $\Yni$ can be thought of
as $|G/H|$ copies of the upper triangular trace
zero matrices and upper nilpotent matrices respectively.
The Bloch--Kato Selmer module  (usually denoted $\HN$) is identified with
$ \Hom_G(\Lambda/I,Y;\Yni).$

\subsection{Involutions respecting irreducible subspaces}

Now that we have identified the various Selmer groups of interest to us as vector
spaces of the form $ \Hom_G(\Lambda/I,Y;\Yin)$ and $ \Hom_G(\Lambda/I,Y;\Yni)$, our next
goal is to prove our main theorem by making use of the Strong Leopoldt Conjecture which allows us to compare the dimensions of these spaces with the dimensions of spaces of homomorphisms $\Hom_G(\Lambda/J,Y;Z)$, where $J$ is a 
left $E[G]$-ideal  such that we have an isomorphism of  left $E[G]$-modules $J  \simeq I$ and  $Z$ is an  appropriate  left rational
$D$-submodule (test object)
related to $\Yin$ and $\Yni$. (Note that the appropriate test objects are left modules,
as the genericity hypothesis will be applied in the dual context of Lemma~\ref{lemma:maingeneric}.)
To do this recall by Lemma~\ref{lemma:type} we need to
construct certain subspaces $Y'_i$ for each $V_i$-isotypic
component $Y_i$ of $Y$. Since the irreducible representations
$V \subseteq Y$ are not apparent from the representation
of $Y$ as $\Ind^{G}_{H} W$, we construct the appropriate
spaces $Y'_i$ by choosing involutions on certain subspaces
of $Y$ which we show stabilize each irreducible representation
$V \subseteq Y$.   Recall that $X \subseteq W$ is
the unique irreducible representation in $W$ on which
$\Gamma$ acts faithfully; it equals $W$ unless
$\Gamma = D_n$ when $X$ has dimension two.
We may write
$$\Ind^{G}_{H} X = \Ind^{G}_{N} \Ind^{N}_{H}X 
= \bigoplus_{g \in {\mathcal R}(G/N)} g \Ind^{N}_{H} X$$ where ${\mathcal R(G/N)}\subset G$ is a representative system for the cosets of $N$ in $G$. 
Since  $N/H$ is not necessarily normal in $G/H$, the
vector space $g \Ind^{N}_{H} X$ is not a representation
of $N$, but rather of $gNg^{-1}$.  Explicitly we have
an isomorphism
$$g \Ind^{N}_{H} X \simeq \Ind^{g N g^{-1}}_{H} \Xg.$$
Recall by Lemma~\ref{lemma:irredform} 
that if $V \subseteq \Ind^{G}_{H} X$ is irreducible, then
$V = \Ind^{G}_{N} U$ for some irreducible $U \subseteq \Ind^{N}_{H} X$.
Thus
$$V = \Ind^{G}_{H} U = \bigoplus_{g \in {\mathcal R}(G/N)} U_g,
\qquad
U_g \subseteq \Ind^{g N g^{-1}}_{H} \Xg.$$
Here $U_g$ is an irreducible $g N g^{-1}$ submodule of
$\Ind^{g N g^{-1}}_{H} \Xg$.
From this we  deduce the following:

\begin{lemma} To construct an involution $\iota$ on
$\Ind^{G}_{H} X$  preserving the irreducible $G$
representations it suffices to construct a collection
of involutions $\iota_g$ on
$\Ind^{g N g^{-1}}_{H} \Xg$
for each $g \in {\mathcal R}(G/N)$ which
 preserve the  irreducible $g N g^{-1}$ representations
inside
 $\Ind^{g N g^{-1}}_{H} \Xg$.
\label{lemma:construct}
\end{lemma}

\begin{proof} Let $\iota = \bigoplus \iota_g$.
Then $\iota_g$ preserves each $U_g$ because
by definition $U_g$ is a irreducible $g N g^{-1}$
submodule of $\Ind^{g N g^{-1}}_{H} \Xg$. Thus the map
$\iota$ preserves their direct sum, which is
$\Ind^{G}_{N} U$. By Lemma~\ref{lemma:irredform}
the module $\Ind^{G}_{N} U$ is irreducible, and moreover
all irreducibles are of this form.
\end{proof}

\subsection{Construction of involutions respecting irreducible subspaces}
Let ${\mathcal R}(N/H)$ be a representative system of left cosets of $H$ in $N$. By definition of the group $N$ there is an isomorphism
$$\Ind^{g N g^{-1}}_{H} \Xg = \bigoplus_{n \in {\mathcal R}(N/H)}
g \cdot n \cdot X, \qquad
\Res_{H} \Ind^{g N g^{-1}}_{H} \Xg \simeq (g X)^{[N:H]}.$$
To construct $\iota_g$ we begin by constructing an involution
on $g X$ for any $g$. We do this by more generally
constructing an involution on $g W$ for any $g$
that preserves the decomposition $g W = g X \oplus g \eps$ when
$\Gamma = D_n$.

\medskip

Let $v$ be the place above $p$ ``corresponding to" $g$ in ${\mathcal R}(G/N)$; i.e.,
if $D = \langle \phi \rangle$ is the decomposition group corresponding to our fixed choice
of place above $p$ then
the decomposition group $D_v$ is isomorphic to
$g D g^{-1} = \langle g \phi g^{-1} \rangle$. 
Let $L_v \subseteq W_v$ denote the special
$D_v$-invariant line. By construction of induced
representations, the action of $H$ on
$W_v \subseteq Y$ is conjugated by $g$, and thus
$L_v \subseteq Y$ is $\phi$-invariant.
Thus for the rest of this subsection we drop the subscript
$v$ from $L_v$ and from the sequence of subspaces
$\Wni_v \subset \Win_v \subset W_v$.
Recall that these spaces fit into the following
exact diagram:
$$\begin{diagram}
0 & \rTo & \Wni & \rEquals & \Hom(V/L,L) & & & \\
  &      & \dInto &     & \dInto & & & \\
0 & \rTo & \Win & \rTo & W & \rTo & \Hom(L,V/L) & \rTo & 0 \\
  &      & \dOnto &       & \dOnto &    & \dEquals &  &  \\
0 & \rTo & \Hom(L,L) & \rTo & \Hom(L,V) & \rTo & \Hom(L,V/L) & \rTo & 0 \\
\end{diagram}.$$
There is a canonical identification $\Hom(L,L) = E$.

\begin{lemma} Let $\psi$ be any non-trivial element of order two
in $\Gamma$ such that the commutator $[\psi,\phi] \neq 0$.
Then
$$W^{\psi = -1} \cap \Wni = 0.$$
Moreover, $\dim(W^{\psi = -1}) = 2$. If $\Gamma = D_n$, then
\begin{enumerate}
\item $\dim(X^{\psi = -1}) = 1$, $\dim(\eps^{\psi = -1}) = 1$.
\item $\dim(W^{\psi = -1} \cap \Win) = 1$, and the projection map
$W^{\psi = -1} \cap \Win \rightarrow \eps$
is an isomorphism.
\item Suppose that $\phi$ has order two.
In the diagram:
$$\begin{diagram}
              &       & W^{\psi = -1} \cap \Win  &      &     \\
              & \ldTo &       & \rdTo&      \\
\Hom(L,L) = E &       & \rDotsto^{\eta_{\psi}}      &      & \eps \\
\end{diagram}
$$
the down arrows induced from the natural surjections
$\Win \rightarrow \Hom(L,L)$ and $W \rightarrow \eps$ are isomorphisms,
and there is a corresponding non-zero map
$\eta_{\psi}: E \rightarrow \eps$ making the diagram commute.
 If $\psi$, $\xi$ are two
 elements of $\Gamma$ of order two distinct from $\phi$
and from each other,
then $\eta_{\psi} \neq \eta_{\xi}$.
\end{enumerate}
\label{lemma:inv}
\end{lemma}

\begin{proof}
Let $\psi$ be any element of order two. 
If $W^{\psi = -1} \cap \Wni \ne 0$, then  in particular
$\psi$ preserves $\Wni = \Hom(V/L,L)$. Yet
$\psi \Wni = \Hom(V/\psi L,\psi L)$, which can equal
$\Wni$ if and only if $\psi L = L$. Yet $\phi L = L$,
and thus $[\psi,\phi]$ acts trivially on $L$. 
In particular (lifting $\phi$, $\psi$ to
$\GL(V)$), $[\psi,\phi]$ is an element of finite order
of determinant one with a fixed line $L$, which is
therefore unipotent and hence trivial since
$[\psi,\phi]$ has finite order.

Now let us assume that $\Gamma = D_n$. If
$\psi$ is any element of order two, then
$\eps^{\psi = -1} = \eps$ has dimension one, and thus
$X^{\psi = -1}$ also has dimension one.
The identity $\dim(W^{\psi = -1} \cap \Win) = 1$ is automatic
by the first part of the lemma and for dimension reasons.
If $\phi$ has odd order, then $\eps = W^{\phi = 1} \subseteq \Win$
and $\eps \subseteq W^{\psi = -1}$, and so
$W^{\psi = -1} \cap \Win = \eps$, and in particular the projection
onto $\eps$ is non-trivial. Thus, from now on, we assume that
$\phi$ has order two.

\begin{sublemma} The maps $\Win  \rightarrow \Hom(L,L) = E$ and
$\Win \rightarrow \eps$ are $D$-module homomorphisms
with distinct kernels, and together
 identify $\Win$ with $E \oplus \eps$.
\label{sublemma:one}
\end{sublemma}

\begin{proof} By construction the maps
are $D$-module homomorphisms.
Since  $E$ is a $\phi = 1$ eigenspace and
$\eps$ is a $\phi = -1$ eigenspace, the result is immediate.
\end{proof}

\medskip

Consider the composite maps
$\Win \cap W^{\psi = -1} \rightarrow \Hom(L,L) = E$ and
$\Win \cap W^{\psi = -1} \rightarrow \eps$.

\begin{sublemma} These maps are isomorphisms of vector spaces.
\label{sublemma:two}
\end{sublemma}

\begin{proof} The kernel of the first map is contained inside the
kernel of the map $\Win \rightarrow \Hom(L,L)$, which
is $\Wni$. Since $\Wni \cap W^{\psi = -1}=0$,  the
first map is injective, and hence an isomorphism by
dimension counting.

The projection $\Win \rightarrow \eps$ is  a map
of $D$-modules, and by Sublemma~\ref{sublemma:one} the
kernel is $D$-stable (equivalently: $\phi$-stable).
Thus the kernel of $\Win \cap W^{\psi = -1} \rightarrow
\eps$ is both $\phi$ and $\psi$ stable. If
it is non-zero, it must be a line which lands
completely within $X$, since $X$ is the kernel 
of the projection $W \rightarrow \eps$.
Yet $[\phi,\psi]$ would fix this line and at the
same time be a determinant one matrix with finite order,
a contradiction.
Thus the second map is injective, and hence an isomorphism.
\end{proof}

This lemma concludes part $(2)$, since we have
shown that $\Win \cap W^{\psi = -1} \rightarrow \eps$
is surjective. It remains to prove part $(3)$.
The existence of $\eta_{\psi}$ is guaranteed
by Sublemma~\ref{sublemma:two}.
We note that the identification
$\Win = E \oplus \eps$ implies that
$\Win \cap W^{\psi = -1} \subset \Win$ is exactly
the \emph{graph} of the homomorphism $\eta_{\psi}$.
Now  two
homomorphisms $\eta_{\psi}$ and $\eta_{\xi}$  are equal
if and only if their graphs are the same.
Suppose that
$$\Win \cap W^{\psi = -1} = \Win \cap W^{\xi = -1}.$$
If the projection of these spaces to $X$ were zero,
then $\Win \cap W^{\psi = -1}$  would lie in $\eps$, and then 
(since $\Win = E \oplus \eps$) their projection to
$E$ would be zero, 
a contradiction.
Thus the projection to $X$ defines a line that
is preserved by $\psi$ and by $\xi$. Once more we
arrive at a contradiction unless $\psi$ and $\xi$ commute.
Yet any pair of elements in $D_n$ ($n$ odd) of order two
do not commute, and thus we are done.
\end{proof}

\medskip

Thus, for each $g \in G$, we have constructed an involution
on $\Wg$ by some element $\psi \in \Gamma$.
For a fixed element $g$ make a choice of $\psi$, and then
lift  $\psi$ to an element $h \in H$.
By definition, if
$n \in N$ then $g \cdot n \cdot W \simeq \Wg$.
Thus the action of $h$ and $\phi$
on  $g \cdot n \cdot W$ is the same as that on f$g W$, and
in particular, the conclusions of 
 Lemma~\ref{lemma:inv} apply
to all the constituents of
$\Res_{H} \Ind^{g N g^{-1}}_{H} \Wg$.

\begin{definition} 
Define the involution
$\iota_g$ on
$\Ind^{g N g^{-1}}_{H} \Wg$ to be map obtained from
multiplication
by $h$. 
\end{definition}

By Lemma~\ref{lemma:construct} we have constructed at least one                  
involution $\iota$ on $\Ind^{G}_{H} X$  by compiling the $\iota_g$'s for $g \in {\mathcal R}(G/N)$. 

If $\Gamma = D_n$, note that  $\Gamma_n$ always
has at least three distinct non-commuting elements of order two, and 
they all act as $-1$ on $\eps$, and so extend to involutions  $\iota_g$ on $\Ind^{G}_{H} W$
preserving $G$-irreducible representations.

\medskip

   Fix such an involution $\iota: \Ind^{G}_{H} X \to \Ind^{G}_{H} X$.
\medskip

\begin{definition} Let $Y':= Y^{\iota = -1} = (\Ind^{G}_{H} W)^{\iota = -1}$.
\end{definition}

\medskip

\begin{lemma}  Let $V, V' \subset Y  = \Ind^{G}_{H} W$ be
two irreducible $E[G]$-submodules.
Then any isomorphism  of $E[G]$-modules
$V \rightarrow V'$ induces an isomorphism of
vector spaces $Y' \cap V \simeq Y' \cap V'$. 
Moreover, $Y'$ is generated by elements lying in
the intersections $Y' \cap V$ for all irreducible
$E[G]$-submodules $V \subset Y$.
\label{lemma:quick}
\end{lemma}

\begin{proof}  If $\Gamma =D_n$ and $V \subseteq \Ind^{G}_{H} \eps$,
the result is trivial because then $Y' \cap {V} = {V}$
and $Y' \cap {V}' = {V}'$.
Any irreducible $V \subseteq \Ind^{G}_{H} X$ is of the form
$\Ind^{G}_{N} U$ for some irreducible $N$-representation
$U \subseteq \Ind^{N}_{H} X$. Thus any isomorphism of
$E[G]$-modules
${V} \rightarrow {V}'$ arises from an isomorphism of $E[N]$-modules
${U} \rightarrow {U}'$. Yet $Y' \cap {U} = U^{\iota = -1}$
is canonically defined in terms of the abstract
$N$-structure of $U$ and
thus is preserved under isomorphism. Finally, since $\iota$ preserves
each irreducible representation $V$ inside $Y$ (by Lemma~\ref{lemma:construct},
and the subsequent construction),
 it follows that
any element of 
$Y' = Y^{\iota = -1}$  can be decomposed as a sum of $\iota = -1$
eigenvectors of irreducible representations $V \subset Y$, and hence
into elements lying in the subspaces 
$Y' \cap V$.
\end{proof}

Note that as an $H$ representation, $W$ is either irreducible or
the sum of two distinct irreducibles, and hence cyclic. Thus
$Y = \Ind^{G}_{H} W$ is cyclic as a $G$-module, and hence admits
an injection $Y \rightarrow \La$ compatible with rational structures.
It follows from Lemma~\ref{lemma:quick} 
and Lemma~\ref{lemma:type} that $Y'$ is of the form
$\Hom_G(\La/J,Y)$ for some ideal  $J\subset \La$.
Thus, in the language of section~\ref{section:ext}, we may write

\begin{corollary}\label{decomp} There is a decomposition
$\displaystyle{Y' = \bigoplus S_i \otimes {T_i}^* \ \subset  \ 
Y = \bigoplus V_i \otimes {T_i}^* \subset \La}$, 
where, for each $i$, $ S_i  \subset  V_i$ is an $E$--subspace, and
${T_i}^* \subseteq {V_i}^*$ is a rational $E$--subspace.
\label{corollary:slow}
\end{corollary}

\subsection{The dimensions of the subspaces $S_i$}

\begin{lemma}
\label{lemma:dimensioninequalities}
Let $U_g \subseteq \Ind^{g N g^{-1}}_{H} \Xg$ be an irreducible
$g N g^{-1}$-sub-representation. Then
\begin{enumerate}
\item If $\Gamma \in
\{A_4, S_4, A_5\}$, then
$$\dim(U^{\iota = - 1}_g) = \frac{2}{3} \cdot \dim(U_g).$$
\item If $\Gamma = D_n$, then
$$\dim(U^{\iota = - 1}_g) = \frac{1}{2} \cdot \dim(U_g).$$
\end{enumerate}
If $V \subseteq \Ind^{G}_{H} X$ is an irreducible $G$-representation,
then
\begin{enumerate}
\item If $\Gamma \in
\{A_4, S_4, A_5\}$, then
$$\dim(V^{\iota = - 1}) = \frac{2}{3} \cdot \dim(V).$$
\item If $\Gamma = D_n$, then
$$\dim(V^{\iota = - 1}) = \frac{1}{2} \cdot \dim(V).$$
\end{enumerate}
Finally, $Y' \cap \Yni = 0$, where
$\Yni$ is as in Definition~\ref{definition:zni}.
\end{lemma}

\begin{proof} The representation
$U_g|H$ is equal to a number of copies of $\Xg$, upon
which $\iota_g$ is acting by construction by our
chosen element of order two  $\psi$.
Thus the result follows directly from Lemma~\ref{lemma:inv},
in particular the claim that $\dim(W^{\psi = -1}) = 2$
and $\dim(X^{\psi = -1}) = 1$ when $\Gamma = D_n$.
For $V$ the result follows in the same way by restricting
to each $U_g$.
Note that
for two distinct $g, g' \in G/H$ the modules $gW$,
$g'W$ may be isomorphic but the ordering of
the eigenvalues $\alpha$, $\beta$ may be reversed,
(i.e., $L_v$ may depend on $g \in G/H$ rather than
$g \in G/N$).
However, the proof of Lemma~\ref{lemma:inv} only
uses the fact that $\phi L_v = L_v$, and so is
symmetric with respect to the two possible choices
of special lines.
Finally, the fact that $Y' \cap \Yni = 0$ also follows
from Lemma~\ref{lemma:inv}, in particular the claim
that $W^{\psi = -1} \cap \Wni = 0$.
\end{proof}

In terms of the decomposition of $Y'$ in Corollary~\ref{corollary:slow},
we may read off information concerning the dimension of the
spaces $S_i$ from Lemma~\ref{lemma:dimensioninequalities}.

\begin{corollary} In the decomposition
$$Y' = \bigoplus S_i \otimes {T_i}^*  \subseteq\ Y = \bigoplus V_i \otimes {T_i}^* \subseteq \La$$ of Corollary~\ref{decomp} (with rational $E$-subspaces 
${T_i}^* \subseteq {V_i}^*$) we have that 
$S_i \subseteq V_i$ is a subspace such that:
\begin{enumerate}
\item If $\Gamma \in \{A_4,S_4,A_5\}$, then $\dim(S_i) = \frac{2}{3} \dim(V_i)$.
\item If $\Gamma = D_n$ and $V_i \subseteq \Ind^{G}_{H} X$, then
$\dim(S_i) = \frac{1}{2} \dim(V_i)$.
\item If $\Gamma = D_n$ and $V_i \subseteq \Ind^{G}_{H} \eps$, then
$\dim(S_i) = \dim(V_i)$.
\end{enumerate}
\end{corollary}

\section{Infinitesimally Classical Deformations of Artin Representations}
\label{section:ICD}

Recall that we are working with a Galois extension $K/\Q$, an algebraically closed extension $E/\Q_p$,
and an
irreducible representation
$$\rho: G_{K} \rightarrow  \GL_2(E)$$
with finite image. Recall that  $\rho$ admits an infinitesimally classical
deformation if and only if there is a  rational vector
$\lambda$ in $E^{[K:\Q]} = \Yin/\Yni$ that lies in the image of a non-zero element in 
$H^1_{\Sigma}(K,\Ad^0(\rho))=H^1_{\Sigma}(K,W)$ under $\pi$.

The subspace $\Yni \subset Y$ is rational, and by construction
is a left $D$-submodule of $Y$.
Furthermore, the same is true, by construction,
of $\Yni(\lambda):= \pi^{-1}(\lambda E)$  (which is indeed a $D$-module
since the action of $D$ on $\Yin/\Yni$ is trivial).
 If $\rho$ admits such  an infinitesimally classical deformation $\lambda$,
 we get the following inequality, which we record for future purposes as a lemma:
\begin{lemma} If $\rho$ admits an infinitesimally classical deformation corresponding
to $\lambda$, then 
$$\Hom_{G}(\Lambda/I,Y;\Yni(\lambda))\ne 0$$
where $\Yni(\lambda)$ is a rational $D$-submodule of $Y$.
\label{lemma:theend}
\end{lemma}

  We shall prove the main theorem (Theorem~\ref{mainth}) by contradiction.  Specifically, we will assume the following
  
\begin{hypothesis}\label{hypoth}  The Artin representation $\rho$ admits  an infinitesimally classical deformation $\lambda$ and yet there is no character $\chi$ such that
 $\chi \otimes \rho$ descends to either to an odd representation over a  totally real field or to a field
containing at least one real place at which $\chi \otimes \rho$
is even. If the projective image of $\rho$ is dihedral, assume further that the determinant character does not
descends to a totally real field $H^{+} \subseteq K$ with corresponding fixed field $H$
such that
  \begin{list}{\kern-0.2em{$($}\roman{Lcount}\kern+0.1em$)$}
    {\usecounter{Lcount}
    \setlength{\rightmargin}{\leftmargin}}
  \item $H/H^{+}$ is a CM extension.
  \item At least one prime above $p$ in $H^{+}$ splits in $H$.
  \end{list}

 \end{hypothesis}
and end with a contradiction to this hypothesis.

\subsection{Construction of the subspace $Y'(\lambda) \subset Y'$}

We shall set about constructing a
space $Y'(\lambda)$  by removing
selected vectors from
$Y' = \Yiota$.  If $\lambda = 0$ then we let $Y'(\lambda) = Y'$,
and all the proofs below continue to go through (alternatively,
one may note that in this case Lemma~\ref{lemma:theend} is true
with \emph{any} choice of $\lambda$).
Henceforth we assume that $\lambda \ne 0$.
Since $Y' \cap \Yni = 0$, the intersection
$Y' \cap \Yni(\lambda)$ is at most one dimensional.
Since $\dim(\Yiota) + \dim(\Yni) = \dim(Y)$,
this dimension
is exactly one.
 Let $v_{\lambda}$ denote a
a non-trivial vector
in this intersection. The vector
$v_{\lambda}$  projects to a non-trivial isotypic
component corresponding to some
 representation $V_i$, which we fix once and for all,
and denote by $\Vdag$.
If $\Gamma = D_n$, we make the additional hypothesis that
$\Vdag \subseteq \Ind^{G}_{H} \eps$. That this is
possible is the content of Lemma~\ref{lemma:inv} part two,
since the projection of $v_{\lambda}$ to $\Win_v$
must land outside $\Wni_v$ for at least one $v$,
and  the map $\Win_v \cap W^{\psi = -1}_{v} 
\rightarrow \eps$ is an isomorphism.
Let $\pi_{\Vdag}$ denote a projection
map $Y \rightarrow \Vdag$  such that  $\pi_{\Vdag}(v_{\lambda})
\ne 0$.
\begin{definition} 
We define $S_i(\lambda) \subset V_i$ and thus $Y'(\lambda)$ as follows:
\begin{enumerate}
\item If $V_i \ne \Vdag$, then $S_i(\lambda):=S_i$.
\item If $V_i = \Vdag$, then $S_i(\lambda):  =$ a hyperplane in $S_i$ that does not contain $\pi_{\Vdag}(v_{\lambda})$.
\end{enumerate}
\end{definition}
\begin{definition}
\label{definition:forlater}
$$Y'(\lambda) = \bigoplus S_i(\lambda) \otimes {T_i}^*  \subseteq\ Y' = \bigoplus S_i\otimes {T_i}^* \subseteq \La$$
\end{definition}  The earlier discussion gave us that $Y' \cap \Yni(\lambda)$ is one-dimensional and generated by $v_{\lambda}$ and the construction of $Y'(\lambda)$ above guarantees that $v_{\lambda} \notin Y'(\lambda)$; thus $Y'(\lambda) \cap \Yni(\lambda) = 0$.

\begin{proposition}\label{prop:nearend}
 Assume Hypothesis~\ref{hypoth}. Then $\dim(S_i(\lambda)) \ge \dim(V_i|c = - 1)$.
\end{proposition}

\begin{proposition}\label{prop:end}
 The Strong Leopoldt Conjecture is in contradiction with Hypothesis~\ref{hypoth}.
\end{proposition}
\bigskip

 Theorem~\ref{mainth} follows from Propositions~\ref{prop:nearend} and \ref{prop:end}. 
We prove Propositions~\ref{prop:nearend} and~\ref{prop:end} in
sections~\ref{subsection:proofofnearend} and~\ref{subsection:propend} 
respectively.

 \subsection{Proof of Proposition~\ref{prop:nearend}} 
\label{subsection:proofofnearend}

By Hypothesis~\ref{hypoth}, $\rho$ does not descend as in the main
theorem, and hence if ${\mathcal C}$ is the conjugacy
class of complex conjugation in $G$, then every 
$c \in {\mathcal C}$ is not even, and there exists at least
one $c \in {\mathcal C}$ such that $c \notin N$.

In particular, the results of Lemma~\ref{lemma:twothirds} apply.

\begin{proposition}
Suppose that
 $\Gamma \in \{A_4, S_4, A_5\}$, then
$$\dim(S_i(\lambda)) \ge
\dim(S_i) - 1
= \frac{2}{3} \dim(V_i) - 1 \ge \dim(V_i|c = - 1).$$
\end{proposition}

\begin{proof}
The last inequality follows from
Lemma~\ref{lemma:twothirds}.
\end{proof}

\bigskip

Now let us consider the dihedral case. Suppose that $\Gamma = D_n$. If $V_i \subseteq \Ind^{G}_{H} X$,
then $V_i \ne \Vdag$, and thus
$$\dim(S_i(\lambda))  =
\dim(S_i) 
= \frac{1}{2} \dim(V_i) = \dim(V_i|c = - 1),$$
after applying  Lemma~\ref{lemma:twothirds}.
Thus it suffices to consider the case when
$V_i \subseteq \Ind^{G}_{H} \eps$.

\medskip

\Remarks
\begin{enumerate}
\item If $V_i \subseteq \Ind^{G}_{H} \eps$ and $V_i \ne \Vdag$, then
$\dim(S_i(\lambda))  =
\dim(S_i) 
= \dim(V_i) \ge \dim(V_i|c = - 1)$,
which suffices,
\item If $V_i = \Vdag$, then
$\dim(S_i(\lambda))  =
\dim(S_i) - 1
= \dim(\Vdag) - 1$,
which is $\ge \dim(\Vdag|c = - 1)$ if and only if
$c$ does not act centrally as $-1$ on $\Vdag$.
\end{enumerate}
From these calculations we have proven Proposition~\ref{prop:nearend},
 except in the case where $\Gamma = D_n$ and
$\Vdag$ has the property that $c$ acts centrally
as $-1$ on $\Vdag$.
The next lemma implies that under certain conditions
we may choose $\Vdag$ so that $c$ does not act centrally,
completing the proof of the main theorem in these cases.
For the remainder of the section we show that all other
situations lead to case two of the main theorem, which
completes the proof.

We postpone the proof of the following lemma to the next section.

\begin{lemma} Suppose that the action of $\phi$ on
$\Wg$ is through an element of order two for all $g$, and              
that $G$ admits at least one complex
conjugation $c \notin N$. Then there exists a choice of
involution $\iota$ such that the
projection $\pi(\Ind^{G}_{H} \eps,v_{\lambda})$
of $v_{\lambda}$ to $\Ind^{G}_{H} \eps$ is not a
$-1$ eigenvector for all conjugates $c$.
\label{lemma:remains}
\end{lemma}

Since we are assuming there exists a $c \in G$ such
that $c \notin N$,  if
$\phi$ acts through an element of
order two for all $g \in G/N$ we may choose a
$\Vdag$ on which some conjugate of $c$ does not act as $-1$.
Yet then $\dim(\Vdag|c = -1) \le \dim(\Vdag) - 1$
and we are done. Hence it remains in the final
cases to establish a suitable descent.

\begin{lemma} Suppose that there exists at least one
$g$ such that $\phi$ acts on $\Wg$ through an element
of odd order. Then $\eps$ descends to a quadratic character
of a totally real  subfield $H^{+}$ of $K$ such that the fixed
field $H$ of the kernel is a CM field, and such that
at least one prime above $p$ in $H^{+}$ splits in $H$.
\end{lemma}

\begin{proof}
Let $F$ be the fixed field of $\eps$.
Let $\Fg$ be the Galois closure of $F$. The
group $\Gal(\Fg/K)$ is an elementary abelian $2$-group.
The fixed fields of the conjugate representations
$\epsg$ are the conjugate fields $F^g$.

Let $\Vdag$ be an irreducible constituent of $\Ind^{G}_{H} \eps$ on
which $c$ acts by $-1$.
As in the proof of Lemma~\ref{lemma:frob}, the representation
$\Res_{H} \Vdag$ contains every conjugate of $\eps$.
Thus $\Gal(\Fg/K)$ acts faithfully on $\Vdag$.
The action of $G$ on $\Vdag$ factors through a quotient in which
$c$ is central and non-trivial. Thus the fixed field is a Galois CM
field $E$ with corresponding Galois totally real subfield $E^{+}$.
The faithful action of $\Gal(\Fg/K)$ implies that $E.K = \Fg$.
Let $H^{+}$ be $E^{+} \cap K$.
We break the remainder of the proof into two cases depending
on whether $E \cap K = H^{+}$ or not.
\begin{enumerate}
\item
Suppose that $K \cap E = H^{+}$.
Then we have the following
diagram of fields:
$$
\begin{diagram}
K & \rLine &  K.E^{+} & \rLine & \Fg = K.E \\
\dLine & & \dLine  & & \dLine \\
H^{+} = K \cap E & \rLine & E^{+} & \rLine &  E \end{diagram}.$$
It follows that there exists a canonical isomorphism
$\Gal(\Fg/K) \simeq \Gal(E/H^{+})$,
and in particular the field $F$ descends to a quadratic extension
$H/H^{+}$. Every extension between $E$ and $H^{+}$
is either totally complex or totally real, and is
totally real if and only if it is contained in $E^{+}$.
If $H$ were contained in $E^{+}$, then $F$ would be
contained in $K.E^{+}$, which is Galois over $\Q$ since both
$K$ and $E^{+}$ are. Yet the Galois closure of $F$ is
$\Fg$, and thus $H/H^{+}$ is a CM extension.
Suppose that all the primes above $p$ in $H^{+}$
are inert in $H$. Then  since $p$ is completely
split in $K/H^{+}$  it follows that all primes
above $p$ in $K$ are inert in $F$. The same argument applies
\emph{mutatis mutandis} to
 all the conjugates of $F$. Yet this
contradicts the fact that $\phi$ acts with odd
order on $\Wg$ for some $g$.
\item Suppose that $E \cap K$ has degree two over $H^{+} = E^{+} \cap K$.
Then necessarily it defines a CM extension. We have the following
diagram of fields:
$$
\begin{diagram}
K & \rLine & \Fg = K.E \\
\dLine & & \dLine \\
K \cap E & \rLine & E \\
\dLine & & \dLine \\
H^{+} = K \cap E^{+} & \rLine & E^{+}
\end{diagram}.$$
It follows that there exists a canonical isomorphism
$\Gal(\Fg/K) \simeq \Gal(E^{+}/H^{+})$,
and thus the field $F$ descends to a totally real  quadratic extension
$H^{++}$ of $H^{+}$. Since $H^{++}/H^{+}$ is  totally
real  and $K \cap E/H^{+}$ is  CM,
the compositum $(K \cap E).H^{++}$
contains a third quadratic extension $H/H^{+}$, which must be CM.
Moreover by construction we see that  $K.H = F$. Thus $H$ descends $F$. Arguing
as in case one we deduce that at least one prime
above $H^{+}$ splits in $H$.
\end{enumerate}
\end{proof}

\subsection{Varying the involution}

It remains to prove
Lemma~\ref{lemma:remains} which, for the convenience of the reader, we restate here.

\begin{lemma} Suppose that the action of $\phi$ on
$\Wg$ is through an element of order two for all $g$, and
that $G$ admits at least one complex
conjugation $c \notin N$. Then there exists a choice of
involution $\iota$ such that the
projection $\pi(\Ind^{G}_{H} \eps,v_{\lambda})$
of $v_{\lambda}$ to $\Ind^{G}_{H} \eps$ is not a
$-1$ eigenvector for all conjugates $c$.
\end{lemma}

\begin{proof} 
Recall that the definition of
$\iota$ involved for each $g \in G/N$
a choice of element $\psi \in \Gamma$ of order two
distinct from
(and therefore not commuting with) $\phi$.
Consider any such choice.
There is an isomorphism of vector spaces
$$\Pi(\iota):\bigoplus_{G/H} E \rightarrow \Ind^{G}_{H} \eps$$
defined as follows.  The space $\Yin$ projects onto both spaces via
the maps:
$$\Yin \rightarrow \Yin/\Yni = \bigoplus_{G/H} E, \qquad
\Yin \subseteq Y \rightarrow \Ind^{G}_{H} \eps$$
respectively (the surjection of the second map follows
from Lemma~\ref{lemma:inv} part two).
These maps identify $\Yin$ with the direct product of
these spaces (exactly as in Sublemma~\ref{sublemma:one}
of Lemma~\ref{lemma:inv}).
Now we define $\Pi(\iota)$ by taking $Y' \cap \Yin = \Yiota \cap \Yin
\subseteq \Yin$ to be the graph of $\Pi(\iota)$.
Explicitly the map  $\Pi(\iota)$  decomposes into a direct
product of the maps
 $\eta_{\psi}$ of Lemma~\ref{lemma:inv} for each $g \in G/H$. 
On the other hand, 
by construction, to compute $\Pi(\iota)(\lambda)$ one lifts
it to $\Yin$, intersects with $Y'$ and then projects down
to $\Ind^{G}_{H}$, and thus the image of $\lambda$ is exactly
$w_{\lambda}:=\pi(\Ind^{G}_{H} \eps,v_{\lambda})$.

\medskip

Let $\gamma \in G/H$ be an element such that
the $\gamma$-component of $w_{\lambda} \ne 0$.
Suppose we vary the choice of element $\xi$ of order two
in the construction of $\iota$
for the $g \in G/N$ corresponding to $\gamma$. Then we obtain
a new vector $v'_{\lambda}$ and a new projection
$$w'_{\lambda}:=\pi(\Ind^{G}_{H} \eps,v'_{\lambda})
= \Pi(\iota')(\lambda).$$
By Lemma~\ref{lemma:inv},
$\Pi(\iota)$ differs from $\Pi(\iota')$ only
for those $g$ of the form $\gamma n$ for $n \in N$,
and \emph{does} differ for those $g$ (since $\eta_{\psi}
\ne \eta_{\xi}$).
Thus the only non-zero components of
 $(w_{\lambda} - w'_{\lambda}) \in \Ind^{G}_{H} \eps$
are contained in  $g \in G/H$ of the form 
$g  = \gamma n$ with $n \in N$, and
$(w_{\lambda} - w'_{\lambda})$ \emph{does} have a non-zero
component at $g$.
 If $w_{\lambda}$ and
$w'_{\lambda}$ both generate representations on which
$c$ acts by $-1$, then they are both $-1$ eigenvectors
for $g c g^{-1}$ for any $g$, and in particular
for $\gamma c \gamma^{-1}$.
Yet then $\gamma c \gamma^{-1}(w_{\lambda} - w'_{\lambda})
= (w'_{\lambda} - w_{\lambda})$ has a non-zero component
at $(\gamma c \gamma^{-1}) \gamma = \gamma c$, and hence
$c \in N$. This is a contradiction, and thus at least
one of $w_{\lambda}$ or $w'_{\lambda}$ generates a representation
$\Vdag$ on which $c$ does not act centrally by $-1$.
\end{proof}

 \subsection{Proof of Proposition~\ref{prop:end}}
\label{subsection:propend}

To summarize, in Definition~\ref{definition:forlater}, we have constructed
(under our hypotheses)
a subspace $Y'(\lambda) \subset \Lambda$ with the
following properties:
\begin{enumerate}
\item $Y'(\lambda)$ admits a decomposition
$\displaystyle{Y'(\lambda) = \bigoplus S_i(\lambda) \otimes {T_i}^*  
\subseteq\ Y' = \bigoplus S_i\otimes {T_i}^* \subseteq \La}$,
\item $\dim(S_i(\lambda)) \ge \dim(V_i|c = -1)$.
\end{enumerate}
We prove that this contradicts the Strong Leopoldt Conjecture.
After shrinking the spaces $S_i(\lambda)$, if necessary, we may
assume that $\dim(S_i(\lambda)) = \dim(V_i|c = -1)$.
Because $Y'(\lambda)$ admits a direct sum decomposition as above,
from the discussion preceding
Lemma~\ref{lemma:type}, we infer the existence of a $G$-submodule
$N \subset \Lambda$ such that
$$\Hom_G(\Lambda/N,Y) = Y'(\lambda).$$
Furthermore, from the final claim
of Lemma~\ref{lemma:type}, we may deduce the possible
structures of $N$ as an abstract $G$-module from
the dimension formula
$\dim(\Hom_G(N,V_i)) = \dim(V_i) - \dim(S_i(\lambda))
= \dim(V_i|c = 1)$. Assuming the classical conjecture of
Leopoldt, one has a decomposition 
$$I = I_{\mathrm{Global}}
\simeq \bigoplus_{\IR(G) \ne E}  V^{\dim(V_i|c = 1)}_i$$
of $G$-modules, where the sum runs over all
\emph{non-trivial} irreducible representations of $G$
(see, for example,~\cite[Ch IX, \S 4]{Lang}).
 It follows that we may assume that $I \simeq N$ as
$G$-modules. Hence, by Lemma~\ref{lemma:maingeneric},
$$\dim(\Hom_G(\Lambda/I,Y;Z))
\le \dim(\Hom_G(\Lambda/N,Y;Z)):=
Y'(\lambda) \cap Z$$
for any $D$-rational subspace $Z$ of $Y$.
Applying this lemma to $Z =  \Yni(\lambda)$, we deduce,
using Lemma~\ref{lemma:theend},
that
$$1 \le \dim(\Hom_G(\Lambda/I,Y;\Yni(\lambda)))
\le 
Y'(\lambda) \cap \Yni(\lambda) = 0,$$
which is a contradiction.


\section{Ordinary families with non-parallel weights}
\label{section:ordinaryfamilieswithnonparallelweights}

\subsection{A result in the spirit of 
Ramakrishna}

Let $p$ be prime, and
let $K/\Q$ be an imaginary quadratic field in which $p$ splits.
Let $\F$ be  a finite field of characteristic $p$, and let
$\chi$ denote the cyclotomic character.
Let $\rhobar: \Gal(\Kbar/K) \rightarrow \GL_2(\F)$ denote an
absolutely irreducible continuous Galois representation,
and suppose, moreover, that $\rhobar$ is ordinary at the
primes dividing $p$. The global Euler characteristic formula
implies that the universal nearly ordinary deformation ring of $\rhobar$
with fixed determinant has dimension $\ge 1$. By allowing
extra primes of ramification, one may construct infinitely many one
parameter nearly ordinary deformations of  $\rhobar$.
Our conjecture (\ref{conj:conjecture}) predicts that these families
contain infinitely many automorphic Galois representations if and only
if they are CM or arise from base change.
 Since many
 families are not of this form (for example, 
if $\rhobar$ is neither Galois invariant nor induced
from a quadratic character of a CM field, then no
such family is of that form), one expects
to find many such one-parameter families with few automorphic points.
In this section we prove that for many $\rhobar$,
there  \emph{exist}  nearly ordinary families
deforming $\rhobar$  with only finitely many
specializations $\rho_x$ with equal Hodge-Tate weights.
Since all cuspidal automorphic forms for $\GL(2)_{/K}$ of
cohomological type have parallel weight,  only finitely
many specializations of $\rho$ can therefore possibly be the
(conjectural) Galois representation attached to such
an automorphic form (for a more explicated version
of what we are ``conjecturing'' about these forms,
see section~\ref{subsection:assumptions}).

\medskip

Let $\chi$ denote the cyclotomic character.
If $\Sigma$ denotes a (possibly infinite) set of places
of $K$,
let $G_{\Sigma}$ denote the Galois group of the maximal
extension of $K$ unramified outside $\Sigma$.

\begin{theorem} 
\label{theorem:raviripoff}
Let $\rhobar: \Gal(\Kbar/K) \rightarrow \GL_2(\F)$ be a continuous
(absolutely irreducible) Galois
representation, and
assume the following:
\begin{enumerate}
\item The image of $\rhobar$ contains $\SL_2(\F)$, and
$p \ge 5$.
\item If $v|p$, then $\rhobar$ is nearly ordinary at $v$,
and takes the shape:
$\displaystyle{\rho|_{I_v}
= \left( \begin{matrix} \psi & * \\ 0 & 1 \end{matrix}
\right)}$, 
 where $\psi \ne 1,\chi^{-1}$, and $*$ is tr\`{e}s ramifi\'{e}e (see~\cite{Serre2})  if
$\psi = \chi$.
\item If $v \nmid p$ and $\rhobar$ is ramified at $v$, then
$H^2(G_v,\ad) = 0$.
\end{enumerate}
Then there exists a Galois representation:
$\displaystyle{\rho: \Gal(\Kbar/K) \rightarrow \GL_2(W(\F)[[T]])}$
lifting $\rhobar$ such that:
\begin{enumerate}
\item  The image of $\rho$ contains $\SL_2(W(\F)[[T]])$.
\item $\rho$ is unramified outside some finite set of primes $\Sigma$.
If $v \in \Sigma$ and $v \nmid p$, then $\rho|D_v$ is
potentially semistable; if $v|p$ then $\rho|D_v$ is nearly ordinary.
\item  
Only finitely many specializations of $\rho$ have parallel weight.
\end{enumerate}
\end{theorem}

\begin{proof}
The proof of this theorem relies on techniques
developed by Ramakrishna~\cite{raviannals,raviinventiones}.
Hypotheses $3$ is 
almost certainly superfluous.
For reasons of space, and to avoid
repetition, we make no attempt in this section
to be independent of~\cite{raviannals}.
In particular, we feel free to refer
to~\cite{raviannals} for key arguments,
and try to keep our notation
as consistent with~\cite{raviannals} as possible.

\medskip

 If $M$ is a finite Galois module over $K$, recall
that $M^*$ denotes the Cartier dual of $M$.
Let $M$ be of order a power of $p$, let $V$ be a finite set of primes of $K$ including all primes dividing $p$ and $\infty$ as well as all primes with respect to which $M$ is ramified. Let $S$ denote the minimal such $V$, i.e., the set of primes dividing $p$, $\infty$ and ramified primes for $M$.  
For $i = 1,2$ let
$$\Sha^i_V(M):= \mathrm{ker} \left\{ H^i(G_V,M) \rightarrow
\prod_{v \notin V} H^i(G_v,M)\right\}.$$
 The group $\Sha^1_V(M)$ is dual to $\Sha^2_V(M^*)$.

Recall  the global Euler characteristic formula:
$$\frac{ \#H^0(G_V,M) \# H^2(G_V,M)}{\# H^1(G_V,M)}
= \frac{\prod_{v|\infty} \# H^0(G_{v},M)}{\# M^{[K:\Q]}},$$
where the product runs over all infinite places of $K$.

Let $S$ be as above, and put $r = \dim \Sha^1_S(\ad^*)$.

\begin{lemma} $\dim H^1(G_S,\ad) = r + 3$.
\end{lemma}

\begin{proof} Our assumptions on $\rhobar|I_v$ together with
assumption $3$ imply that $H^2(G_v,\ad) = 0$ for all $v \in S$.
Thus $H^2(G_S,\ad) \simeq \Sha^2_S(\ad)$, and hence
$$\dim H^2(G_S,\ad) = \dim \Sha^1_S(\ad^*) = r.$$
Since $K$ is an imaginary quadratic field, there is
exactly one place at infinity. Since (by the isomorphism above)
$\dim H^2(G_S,\ad) = r$, the Global Euler characteristic Formula for $M:= \ad$ gives  that
$$\dim H^1(G_S,\ad) = r + \dim H^0(G_S,\ad)
+ 2 \cdot \dim(\ad) - \dim(H^0(G_{\infty},\ad))  = r+3,$$
where we use the absolute irreducibility of
$\rhobar$ and the fact that $G_{\infty}$ is trivial.
\end{proof}

If $v|p$, then $H^2(G_v,\ad) = 0$, by assumption $3$.
The local Euler characteristic formula implies that
$$\# H^1(G_v,\ad) = \# H^2(G_v,\ad) \cdot \# H^0(G_v,\ad)
\cdot \# \ad,$$
 and thus $\dim H^1(G_v,\ad) = 3$ (by assumption $2$).
Let
$\N_v \subset H^1(G_v,\ad)$ for $v|p$
denote the $2$-dimensional subspace 
corresponding to
nearly ordinary deformations of $\rhobar$. Note that our
 $\N_v$ is 
different than Ramakrishna's,
where it is defined as the  $1$-dimensional
subspace of $H^1(G_v,\ad)$
corresponding to \emph{ordinary}
deformations. Over $\Q$, the difference between ordinary
and nearly ordinary deformations
is slight, and amounts  to fixing the determinant and replacing
representations over $W(\F)[[T]]$ by representations over $W(\F)$.
However, over $K$, this flexibility will be crucial.
If we insist on an
unramified quotient locally (at each $v|p$) and also
fix the determinant,
the relevant ordinary deformation ring will
have expected relative dimension $-1$ over $W(\F)$,
and so we would not expect to find characteristic
zero lifts in general.
If $v \nmid p$ and $v \in S$, we take $\N_v = H^1(G_v,\ad)$,
as in~\cite{raviannals}  p.144 (Recall that $H^2(G_v,\ad) = 0$,
by assumption).

\medskip

We shall consider  primes $\q \notin S$ such that
up to twist
$\displaystyle{\rhobar|_{G_{\q}}
 = \left( \begin{matrix} \chi & 0 \\ 0 & 1          
\end{matrix} \right)}$,
and $N(\q) \not\equiv \pm 1 \mod p$.
For these primes 
 $\dim H^2(G_v,\ad) = 1$, and 
$\N_v \subset H^1(G_v,\ad)$ is the codimension one subspace
corresponding to deformations of the form:
$$\rho|_{G_{\q}} = \left( \begin{matrix} \chi & * \\ 0 & 1
\end{matrix} \right).$$
We call such primes auxiliary primes for $\rhobar$.

\medskip

Let $\{f_1, \ldots, f_{r+3}\}$ denote a basis
for $H^1(G_S,\ad)$.
We recall from~\cite{raviannals} p.138
 the definition of some extensions of $K$.
Let $\K$ denote the compositum $ K(\ad) \cdot  K(\mu_p)$. Then $K \subset \K \subset 
K({\bar\rho})$.  If $f_i \in H^1(G_S,\ad)$,
let $\L_i/\K$ denote the extension with $\L_i$
the fixed field of the kernel of
the morphism $\phi(f_i): \Gal(\Kbar/\K) \rightarrow
\ad$, where $\phi$ is the composition:
$$\phi: H^1(G_S,\ad) \rightarrow H^1(\Gal(\Kbar/\K),\ad)
= \Hom(\Gal(\Kbar/\K),\ad).$$
We have avoided issues arising from
 fields of definition (of concern to Ramakrishna
 in~\cite{raviannals}, p. 140)
 since,
 by assumption one,
the image of $\rhobar$ contains $\SL_2(\F)$.
Let $\L$ denote the compositum of all the
$\L_i$, and we have a natural isomorphism
$$\Gal(\L/\K) \simeq \prod_{i=1}^{r+3}\Gal(\L_i/\K)$$ (cf.~\cite{raviannals}, p.141).

\medskip

We recall the construction of some elements in $\Gal(\L/K)$
from~\cite{raviannals} p.141--142.
Let $c \in \Gal(\K/K)$ denote
an element possessing a lift ${\tilde c} \in \Gal(K(\rhobar)/K)$ such that $\rhobar({\tilde c})$ has distinct eigenvalues
with ratio $t \not\equiv \pm 1 \mod p$ and whose projection
to $\Gal(K(\mu_p)/K)$ is $t$. Since the image $\rhobar({\tilde c})$ is defined up to 
twist, the ratio of its eigenvalues is independent of the lift ${\tilde c}$; i.e., is dependent only on $c$.
Let
$$\alpha_i \in \Gal(\L/\K)\simeq \prod_{j=1}^{r+3}\Gal(\L_j/\K)$$ denote the (unique up to scalar)
element which, when expressed as an $r+3$-tuple in $\prod_{j=1}^{r+3}\Gal(\L_j/\K)$ has all entries trivial except for the $i$-th at which its entry is given by a nonzero element $\alpha_i \in \Gal(\L_i/\K)$ on which a lifting of ${c}$ to $\Gal(\L_i/K)$ acts trivially by conjugation.
(cf.~\cite{raviannals} Fact $12$).
 Let $T_i$ denote the Cebotarev class
of primes in $K$ corresponding to the 
 element
$\alpha_i \rtimes  c \in \Gal(\L/K)$.
We have implicitly used the assumption that
$p \ge 5$ to deduce that $\Gal(\L/K)$ is
a semidirect product of $\Gal(\K/K)$
by $\Gal(\L/\K)$.
Using the element $c$ constructed above, the fact that
$p \ge 5$, and that (when $p = 5$) the image of $\rhobar$
contains $\SL_2(\F)$, one may establish
the following (cf.~\cite{raviannals}, Fact 16, p.143):

\begin{lemma} There is a set $Q = \{\q_1, \ldots,\q_r\}$
of auxiliary primes $\q_i \notin S$ such that:
\begin{enumerate}
\item $\q_i$ is unramified in $K(\rhobar)$, and
is an auxiliary prime for $\rhobar$.
\item $\Sha^1_{S\cup Q}(\ad^*) = 0$ and
$\Sha^2_{S \cup Q}(\ad) = 0$.
\item $f_i|G_{\q_j} = 0$ for $i \ne j$,
$1 \le  i \le r+3$ and $1 \le j \le r$.
\item $f_j|G_{\q_j} \notin \N_{\q_j}$ for $1 \le j \le r$.
\item The inflation map $H^1(G_S,\ad) \rightarrow H^1(G_{S \cup Q},\ad)$
is an isomorphism.
\end{enumerate}
\label{lemma:ravifact16}
\end{lemma}

\begin{proof} See~\cite{raviannals}, p.143,
and~\cite{raviinventiones}, p.556. Note
that (in our context) the parameter $s$ is assumed to be 
zero. Of use in verifying that the same argument works in our context
(in particular~\cite{raviannals}, p.143, the last paragraph of the
proof) is the following identity:
$$\dim H^0(G_{\infty},\ad) = 3 =
\dim H^1(G_S,\ad) - r.$$
\end{proof}

Let us consider the map:
$$\Phi:H^1(G_S,\ad) \rightarrow
\prod_{v \in S} H^1(G_v,\ad)/\N_v =  \prod_{v|p} H^1(G_v,\ad)/\N_v.$$
Let $d$ denote the dimension of $\ker(\Phi)$.
Since the source has dimension $r+3$ and the target has
dimension $2$, it follows that $d \ge r+1$.
Let $\{f_1, \ldots,f_d\}$ denote a basis of this kernel,
and augment this to a basis $\{f_1, \ldots, f_{r+3}\}$ of
$H^1(G_S,\ad)$. As in~\cite{raviannals}, p.145, applying
Lemma~\ref{lemma:ravifact16} to this ordering of $f_i$, we may
assume that
for
$1 \le i \le r$ the element $f_i$ restricts to
zero in $H^1(G_v,\ad)$ for all $v \in S \cup Q$ except
$\q_i$, for which $f_i|G_{\q_i} \notin \N_{\q_i}$.
We now consider the elements $\{f_{r+1},\ldots,f_d\}$.
Here we note one important difference from
 Ramakrishna~\cite{raviannals}; since $d \ge r+1$
this set is necessarily
non-empty.

\begin{lemma} For $i = r+1,\ldots,d$ the
restriction map:
$$\theta: H^1(G_{S \cup Q \cup T_i},\ad)
\rightarrow
\prod_{v \in S} H^1(G_v,\ad)$$
is surjective.
For $i = r+1, \ldots d-1$ there exist  primes
$\r_i \in T_i$ such that the map:
$$H^1(G_{S \cup Q \cup \{\r_{r+1},
\ldots,\r_{d-1}\}},\ad) \rightarrow
\prod_{v \in S} H^1(G_v,\ad)/\N_v$$
is surjective.
\label{lemma:Prop10}
\end{lemma}

\begin{proof} The first part
of this Lemma is~\cite{raviannals},  Proposition 10, p.146,
and the same proof applies. The second part
follows exactly as in~\cite{raviannals}, Lemma 14, p.147,
except that only $d-r-1$ primes are required. Explicitly, the target has
dimension two, and the image of
$H^1(G_S,\ad)$ has dimension $r+3 - d$,  and $2 - (r+3-d) = d-r-1$.
\end{proof}

\Remark 
The reason why the numerology differs from that in Ramakrishna is that
we are not fixing the determinant. Note, however, that imposing a 
determinant condition would impose \emph{two} separate conditions at each
$v|p$.

\medskip

Let $T$ denote the set of primes
$\{\r_{r+1},\ldots,\r_{d-1}\}$. That is, $T$
is a finite set consisting of a single prime in $T_i$
for $i = r+1, \ldots, d-1$.
Combining the above results
and constructions we arrive at the
following:

\begin{lemma} The map
$$H^1(G_{S \cup Q \cup T},\ad)
\rightarrow \bigoplus_{v \in S \cup Q \cup T}
H^1(G_v,\ad)/\N_v$$
is surjective, and the kernel is one dimensional.
\end{lemma}

\begin{proof} The source has dimension
$r+3 + (d-r-1)= d+2$,
and the target has dimension $2 + (d-1) = d+1$.
The argument proceeds as in~\cite{raviannals},
Lemma~16, p.149, except now one has no
control over the image of the cocycle $f_d$, and  one may only
conclude that the kernel is at most one dimensional.
By a dimension count this implies the lemma.
\end{proof}

\medskip

Thus we arrive at the following situation.
Let $\Sigma = S \cup Q \cup T$. Then
$\Sha^2_{\Sigma}(\ad) = 0$ 
by Tate-Poitou duality (cf.~\cite{raviannals}, p.151)
and the fact (Lemma~\ref{lemma:ravifact16}) that
$\Sha^1_{S \cup Q}(\ad^*) = 0$.
Thus the map:
$$H^2(G_{\Sigma},\ad) \rightarrow \prod_{v \in \Sigma} H^2(G_v,\ad)$$
is an isomorphism. In other words, all obstructions
to deformation problems are local.
On the other hand, we have a surjection
$$H^1(G_{\Sigma},\ad) \rightarrow \prod_{v \in \Sigma}
H^1(G_v,\ad)/\N_v$$
with one dimensional kernel (the source
has dimension $d+2$ and
the target has dimension $d+1$).
Ramakrishna's construction identifies the local
deformation rings (with respect to the subspaces $\N_v$)
as power series rings, and thus smooth quotients of the
unrestricted local deformation rings. Thus by the remark above there
are no obstructions to lifting representations which are
locally in $\N_v$, and thus by the dimension calculation above
one has the following:

\begin{lemma} Let $R_{\Sigma}$ denote the universal
deformation ring of $\rhobar$ subject to the
following constraints:
\begin{enumerate}
\item The determinant of $\rho$ is some given lift
of
$\det(\rhobar)$ to $W(\F)$.
\item $\rho$ is nearly ordinary at $v|p$.
\item $\rho$ is minimally ramified at $v|S$ for
$v \nmid p$.
\item $\rho|G_{\q}$ takes the shape
$\displaystyle{\left( \begin{matrix} \chi & * \\
0 & 1 \end{matrix} \right)}$ for auxiliary
primes $\q \in \Sigma \setminus S$.
\end{enumerate}
Then $R_{\Sigma} \simeq W(\F)[[T]]$, and the image of the
corresponding universal deformation
$$\rhou: \Gal(\Kbar/K) \rightarrow
\GL_2(W(\F)[[T]])$$
contains $\SL_2(W(\F)[[T]])$.
\end{lemma}

\medskip

This representation satisfies all the conditions
required by
Theorem~\ref{theorem:raviripoff} except (possibly) the claim
about specializations with parallel weight.
To resolve this, we shall
construct one more auxiliary prime $\r$.
Let $f_d \in H^1(G_{\Sigma},\ad)$ denote a nontrivial element in the kernel
of the reduction map $$H^1(G_{\Sigma},\ad) \rightarrow \prod_{v \in \Sigma}
H^1(G_v,\ad)/\N_v.$$ Then the restriction  of $f_d$
lands in
$$\prod_{v|\Sigma} \N_v.$$
We may detect the infinitesimal Hodge-Tate weights
(or at least their reduction modulo $p$) by the analog
of the characteristic zero construction in 
section~\ref{section:deformationtheory}.
In particular,
 the infinitesimal Hodge-Tate weights are
obtained by taking the  image of $f_d$ under
the projection:
$$\omega: \prod_{v|\Sigma} \N_v \rightarrow \prod_{v|p} \N_v \rightarrow
 \prod_{v|p} \F.$$
\emph{A priori}, three possible cases can occur:
\begin{enumerate}
\item The image of $f_d$ in $\F \oplus \F$ is 
zero\footnote{This case presumably never occurs,
since it implies that  the existence of a positive
dimensional family of Global representations with
constant Hodge-Tate weights.}.
\item The image of $f_d$ in $\F \oplus \F$ is non-zero but
\emph{diagonal}, i.e., a multiple of $(1,1)$.
\item The image of $f_d$ in $\F \oplus \F$ is non-zero but
not diagonal, i.e., not a multiple of $(1,1)$.
\end{enumerate}
If the family $\rhou$ has parallel weight, then the infinitesimal
Hodge-Tate weights at any specialization will be (a multiple of) $(1,1)$,
and thus we will be in case $2$. If the family 
$\rhou$ has constant weight, then we shall be in case $1$,
and if the family has non-constant but non-parallel weight,
case $3$ occurs. Thus we may assume we are in the first two
cases.
To reduce to the third case,
we introduce a final auxiliary prime $\r_d$.
By construction, $\{f_1,\ldots,f_{d-1},f_{d+1},f_{d+2}\}$
maps isomorphically
onto $\displaystyle{\bigoplus_{v|\Sigma} H^1(G_v,\ad)/\N_v}$.
We choose a splitting
$$\bigoplus_{v|\Sigma} H^1(G_v,\ad) = 
\left( \bigoplus_{v|\Sigma} H^1(G_v,\ad)/\N_v \right)
\oplus \left(\bigoplus_{v|\Sigma} \N_v \right)$$
such that the first factor is generated by
$f_1,\ldots,f_{d-1},f_{d+1},f_{d+2}$.
Recall that
$T_i$ denotes the primes not in $S \cup Q$ whose 
Frobenius element is in the conjugacy class
$\alpha_i \rtimes  c \in \Gal(\L/K)$.

\begin{lemma} The restriction map
$$H^1(G_{\Sigma \cup T_d},\ad)
\rightarrow \left(\bigoplus_{v \in \Sigma} H^1(G_v,\ad)\right)$$
is surjective.
\label{lemma:extra}
\end{lemma}

\begin{proof} The proof is very similar
to Ramakrishna's proof of Proposition~10, p.146. This proof is, in effect, a calculation using Poitou-Tate duality.  Exactly as in~\cite{raviannals}, p.146, we consider the restriction maps
$\tilde{\theta}$ and $\tilde{\vartheta}$ (this notation defined exactly  as in 
\emph{loc. cit.}). It then
 suffices to prove that
the annihilator of $\left(P^1_{T_d}(\ad) + \mathrm{Image}(\tilde{\theta})\right)$
is trivial. For this, it suffices to show
that the set $\{\tilde{\vartheta}(x)\}$
for $x \in H^1(G_{\Sigma \cup T_d},\ad^*)$ such that $x|G_v = 0$ for
all $v \in T_d$ is trivial. Ramakrishna's argument assumes that
$x|G_v = 0$ for all $v \in Q \cup T_d$, although he never uses that
$x|G_v = 0$ for $v \in Q$. 
 Thus $\theta$ surjects onto
$$\frac{P^1_{\Sigma \cup T_d}(\ad)}{P^1_{T_d}(\ad)} =
\bigoplus_{v|\Sigma} H^1(G_v,\ad).$$
We remark that the same argument proves surjectivity of
the reduction map with $\Sigma$ replaced by $\Sigma \cup 
\widetilde{T}$, for any finite subset $\widetilde{T}
\subset T_d$, although we do not use this fact.
\end{proof}

\medskip

Returning to our argument, we conclude from
Lemma~\ref{lemma:extra} that the map
$$H^1(G_{\Sigma \cup T_d},\ad) \rightarrow  \bigoplus_{v|\Sigma}
 H^1(G_v,\ad) \rightarrow
\bigoplus_{v|\Sigma} \N_v  \ {\stackrel{\omega}{\to}}\ 
 \bigoplus_{v|p} \F$$
is surjective, where this projection map is obtained from
the splitting above (constructed from the cocycles
$\{f_1,\ldots,f_{d-1},f_{d+1},f_{d+2}\}$).
Thus there exists a  auxiliary prime $\r_d$ such that the image
of the new cocycle $f_{d+3}$ in $H^1(G_{\Sigma \cup \{\r_{d}\}},\ad)$
projects to a non-diagonal subspace.
 Let $\Sigma' = \Sigma \cup \{\r_d\}$.
Exactly as previously,
the map
$$H^1(G_{\Sigma'},\ad) \rightarrow
\bigoplus_{v \in \Sigma'} H^1(G_v,\ad)/\N_v$$
is surjective, with a one dimensional kernel. The kernel of
the projection to the smaller space 
$\bigoplus_{v \in \Sigma} H^1(G_v,\ad)/\N_v$ is two
dimensional, and contains $f_d$. If $f_{d+3}$ denotes the
new cocycle in this kernel, then by construction the projection
of $f_{d+3}$ to $\F \oplus \F$ is not diagonal, since adjusting
$f_{d+3}$ by elements of $\{f_1,\ldots,f_{d-1},f_{d+1},f_{d+2}\}$
does not change this projection. Since $f_d|G_{\r_d} \not\in \N_{\r_d}$,
it follows that the kernel of the  projection to $\bigoplus_{v \in \Sigma'} H^1(G_v,\ad)/\N_v$
is generated by $f_{d+3} - \alpha \cdot f_d$ for some
$\alpha$, and thus the projection of this element to 
$\F \oplus \F$ is also not diagonal.
This completes the proof.
\end{proof}

\subsection{Examples}

\begin{corollary} Suppose that $\rhobar$ satisfies
the conditions of Theorem~\ref{theorem:raviripoff}. Then there
exists a set of primes $\Sigma$ such that the if
$R_{\Sigma}$ is the universal
nearly ordinary deformation ring of $\rhobar$ unramified
outside $\Sigma$ then $\mathrm{Spec}(R_{\Sigma})$ contains a one dimensional
component with only finitely many representations
of parallel weight.
\end{corollary}

For example, one source of such $\rhobar$ is the
mod $p$ representations associated to 
$\displaystyle{\Delta = \sum_{n=1}^{\infty} \tau(n) q^n}$
whenever $p \nmid \tau(p)$ and $11 \le p \ne 23,691$.

\section{Complements}
\label{section:extras}

\subsection{Galois representations associated to Automorphic Forms}
\label{subsection:assumptions}

Let $K/\Q$ be an imaginary quadratic field, and
let $\pi$ denote an automorphic representation of
cohomological type for $\GL_2(\mathbf{A}_K)$. It is expected, and
known in many cases, that associated to $\pi$ there exists
a compatible family of $\lambda$-adic representations
 $\{\rho_{\lambda}\}$ such that for $v \nmid N(\lambda)$, the
representation $\rho_{\lambda}|_{D_v}$ corresponds (via the
local Langlands correspondence) to $\pi_v$. One may further 
conjecture a compatibility between $\rho_{\lambda}|_{D_v}$
and $\pi_v$ for $v | N(\lambda)$, in the sense of Fontaine's
theory. 
The weakest form of such a conjecture  is to ask that the
representation $\rho|_{D_v}$ be Hodge-Tate of the weights
``predicted'' (by the general Langlands conjectures) from
the weight of $\pi$. If $\pi$ is \emph{cuspidal}, then a result
of Harder
(\cite{Harder}, \S III, p.59--75) implies that $\pi$
has parallel weight, and thus the expected Hodge-Tate weights
of $\rho_{\lambda}$ at $v|N(\lambda)$ are of the form $[0,k-1]$,
for some $k \in \Z$ that does not depend on $v$. For the
cases in which the existence of the Galois representation  has
already been established (by Taylor et.al.~\cite{taylorothers,taylor}),
one obtains this weaker 
compatibility automatically by construction, as we now explain.
The Galois representations associated to $\pi$
are constructed from limits of
representations arising from Siegel modular forms. Since
Sen-Hodge-Tate weights are analytic functions of deformation
space~\cite{Sen}, it suffices to determine the Hodge-Tate weights
attached to Galois representations of Siegel modular forms
in \emph{classical} weight $(a,b)$,
and in particular to show that the Hodge-Tate weights
of such forms are $[0,a-1,b-2,a+b-3]$. This can be
proven by a comparison between
\'etale cohomology and de Rham cohomology on Siegel threefolds
(see, for example,~\cite{MT}).
When we talk in this paper of classical automorphic
(Galois) representations associated to automorphic forms
over $K$, we refer to the (possibly conjectural) Galois
representations satisfying the usual local Langlands
compatibilities at primes $v \nmid N(\lambda)$, and the
compatibility at $v \mid N(\lambda)$ between the
automorphic weight of $\pi$ and the Hodge-Tate weights
of $\rho_{\lambda}|_{D_v}$. 
Thus, by fiat, all such representations have Hodge-Tate
weights $[0,k-1]$ for $v|N(\lambda)$ and for some $k \in \Z$
independent of $v$.

\subsection{Even two-dimensional Galois representations over $\Q$}\label{antidiagwt}

Let $V$ be a two-dimensional vector space over $E$ 
and let  $\rho: \Gal(\Qbar/\Q) \rightarrow \Aut_E(V)$
be an even continuous (irreducible) nearly ordinary Artin representation. 
Let $K/\Q$ be an imaginary quadratic
field in which $p$ splits.  Restrict $\rho$ to $G_K$ and give it a nearly ordinary structure {\it compatible with the action of $\Gal(K/\Q)$}.
Equivalently,
if $v, {\bar v}$ are the places above $p$, with  decomposition groups $D_v$ and $D_{\bar v}$ chosen so that they are conjugate to one another in $G_\Q$ 
(i.e., $D_{\bar v}=\gamma D_v \gamma^{-1}$ where $\gamma$ is some element  in $G_\Q$),  the chosen lines 
$L_v, L_{\bar v} \subset V$ are related to each other by the formula $L_{\bar v}=\gamma L_v$. 

 The representation $\rho|_K$ admits
infinitesimally classical deformations, with infinitesimal Hodge-Tate
weights lying in the anti-diagonal subspace  of infinitesimal weight space (Definition~\ref{definition:HTW}) $$\{(e,-e) \ | \ e\in E\} \subset E\oplus E\ = \   \{\bigoplus_{v|p}\Q\}\otimes_\Q{\mathcal E}.$$
 
  To prove this,  note that
$H^1(K,\mathrm{ad}^0(\rho))$ decomposes as
$H^1(\Q,\mathrm{ad}^0(\rho)) \oplus H^1(\Q,\mathrm{ad}^0(\rho)
\otimes \chi)$, where $\chi$ is the quadratic character
associated to $K$. The first cohomology group records nearly
ordinary deformations of $\rho$ \emph{over $\Q$}, while
the second measures nearly ordinary deformations \emph{over $K$} in the
``anticyclotomic'' direction. 
The fixed
field of the kernel of $\mathrm{ad}^0(\rho)$ is totally
real, for the following reason. Either complex conjugation acts as $+1$ in the representation $\rho$ (in which case the kernel of $\rho$ itself is totally real) or else it acts as the scalar $-1$ in the representation $\rho$ and therefore  acts as the identity in $\mathrm{ad}^0(\rho)$. 
 The Euler characteristic formula then implies that
$$\dim(H^1(\Q,\mathrm{ad}^0(\rho))) \ge 0, \qquad
\dim(H^1(\Q,\mathrm{ad}^0(\rho) \otimes \chi)) \ge 1,$$
with equality in either case following from the Leopoldt conjecture.

 If $\rho|_K$ lies 
in a one-parameter (or many-parameter, for that matter) 
family  of $G_K$-representations 
in which the classical automorphic representations (i.e., those Galois representations 
 corresponding conjecturally, as in section~\ref{subsection:assumptions},
to classical automorphic form for $\GL(2)_{/K}$) 
are Zariski-dense, then the infinitesimal 
Hodge-Tate
weights of this representation would be diagonal, i.e., would lie in the diagonal subspace  $$\{(e,e) \ | \ e\in E\} \subset E\oplus E\ = \   \{\bigoplus_{v|p}\Q\}\otimes_\Q{\mathcal E}.$$  
Thus if we assume the Leopoldt conjecture (specifically,
for the compositum of  $K$ and the
fixed field of $\rho$), we deduce the following:

\begin{theorem} Let $\rho: \Gal(\Qbar/\Q) \rightarrow \GL_2(E)$
be an even continuous (irreducible) nearly ordinary Artin representation. Let $K/\Q$ be an
imaginary quadratic field in which $p$ splits. Assume the
Leopoldt conjecture for the fixed field of $\rho|_K$. Then the representation
$\rho|_K$ does not deform to a one-parameter family of nearly ordinary  Galois representations arising from classical automorphic forms over $K$.
\end{theorem}

Note that the proof~\cite{BDST} of the 
Artin conjecture for  certain \emph{odd} representations
of $\Gal(\Qbar/\Q)$ 
requires the following idea:
one places the representation in question into a one
parameter family of Galois representations in which the
representations associated to classical ordinary modular forms
are Zariski dense.
The above theorem shows that even the first step in
such a 
program for even representations
(extending the base field such that $\rho$ admits deformations to
other classical automorphic forms) does not na\"{\i}vely work.

Suppose that $H^{+}$ is a Galois extension in which $p$ completely splits,
and $H/H^{+}$ is a CM extension in which at least one prime above $p$ in
$H^{+}$ splits in $H$. Then there exist algebraic Hecke characters $\eta$ of
$\mathbf{A}^{\times}_H$ such
that the Galois representations associated to their automorphic
 inductions on $\GL_2(\mathbf{A}_{H^+})$
 are nearly ordinary and interpolate to 
form families of possibly { lower} dimension than the dimension 
of the (ambient) universal deformation space of nearly ordinary representations. The 
Hodge-Tate weights of these families are constant at the places 
corresponding to primes  $v \ | \ p$ of $H^+$  that are inert in the extension
$H/H^{+}$ --- this is a manifestation of Lemma~\ref{lemma:remains}.

An example where the CM locus is non-trivial but strictly smaller
than the nearly ordinary locus can easily be constructed:
take a non-Galois extension
of a real quadratic field $H^{+}$ such that
$p$ splits in $H^{+}$ and exactly one prime above $p$ splits in $H$.

In both the ``CM'' and ``Base Change'' cases there exist
classical families giving rise to the infinitesimally classical
deformations.

\medskip

In the  ``even'' case of Theorem~\ref{theorem:mainintro}, the field
$K^{+} \subset K$ to which $\rho \otimes \chi$ descends
does not need to be totally real.
Let $K^{+}/\Q$ be a degree three field with signature $(1,1)$ in which
$p$ is completely split, and let $K$ be its Galois closure.
Let $L/K^{+}$ be an $A_4$ extension  that is even at the real place, and
such that $\Gal(M/\Q) =
(A_4)^3 \ltimes S_3$, where $M$ is the Galois closure of $L$.
Suppose that the decomposition group at some place $v|p$
is generated by the element $(u,u,u)$, where
$u \in A_4$ has order three. Let $\rho$ be the representation
$\rho: \Gal(\Kbar/K) \rightarrow \GL_2(E)$ associated to $L/K$.
Then $\rho$ admits an infinitesimally classical deformation.

\subsection{Artin representations of $\Gal(\Qbar/\Q)$ of dimension $n$}

The methods of the previous section can also be applied
to Artin representations:
$$\rho: G_{\Q} \rightarrow \GL_n(E)$$
when $n > 2$. For such a representation $V$, let
$\mathrm{ad}^0(\rho): G_{\Q} \rightarrow \GL_{n^2-1}(E)$ 
denote the
corresponding adjoint representation.
Let $G:=\im(\mathrm{ad}^0(\rho)) \subset \GL_{n^2 - 1}(E)$
and  $\tw{G} =   \im(\rho) \subset \GL_n(E)$; the
group $\tw{G}$ is a central extension of $G$.
We assume that $\rho$
is unramified at $p$, so the image of the
decomposition group $D$ at $p$ is cyclic; we also assume
that $D$  acts with
distinct eigenvalues on $V$, and write (having chosen once
and for all an ordering)
$$V = \bigoplus_{i=1}^{n} L_i,$$
where $L_i$ are $D$-invariant lines of $V$.
Consider deformations which are nearly
ordinary at $p$. 
Let $W = \Hom'(V,V)$ denote the hyperplane in
$\Hom(V,V)$ consisting of endomorphisms of $V$ of trace zero.
The analog of the vector space
$W^0 \subset W$ is 
$\Ker(W \rightarrow \bigoplus_{i < j} \Hom(L_i,L_j))$, which, after a suitable
choice of basis, may  be identified with upper triangular matrices
of trace zero. There is a canonical map
$\eps: W^0 \rightarrow \bigoplus \Hom(L_i,L_i) = E^n$, whose image is the
hyperplane 
$H:\sum x_i = 0$. Let $H_{\Q}$ denote a model of
the hyperplane $H$ over $\Q$.
The map $\eps$ induces, via the local class field theory,
 a map on infinitesimal deformations:
$$\omega: H^1_{\Sigma}(\Q,\mathrm{ad}^0(\rho))
\rightarrow
H^1(D,H) = \Hom(D^{\rm ab},H) \rightarrow \Hom(\Q_p^*,H)
= H_{\Q} \otimes_{\Q} {\mathcal E},$$
where, as in section~\ref{subsection:infinite},
${\mathcal E} = \Hom(\Z_p^{*},E)$.
The  definition of infinitesimal Hodge-Tate
weights proceeds as in section~\ref{subsection:infinite}.

\begin{lemma} If $W = \mathrm{ad}^0(V)$ is irreducible,
then assuming the strong Leopoldt conjecture,
$\rho$ admits no infinitesimally
classical deformations.
\label{lemma:sporadic}
\end{lemma}

\begin{proof}[{\bf Sketch}] It suffices
to construct  a suitable space $Y'$ of $W$ which is
``maximally skew'' to the ordinary subspace,
but still decomposes well with respect to
the decomposition of $W$
as a $G = \Gal(K/\Q)$-representation, where
$K$ is the fixed field of $\mathrm{ad}^0(\rho)$. Yet, by
assumption, $W$ is irreducible,
and so this is no condition at all. For example,
 let $\iota$ denote the involution of $W$
 such that
$\iota \Hom(L_i,L_j) = \Hom(L_j,L_i)$. Then
one may take $Y' = W^{\iota = - 1}$.
\end{proof}

Note that if $\rho$ (equivalently, $V$)
 is self-dual and $n \ge 3$,
then $W = \mathrm{ad}^0(V)$ is \emph{never} irreducible,
since $\mathrm{ad}(V) = V \otimes V^* \simeq
V \otimes V$ decomposes into
$\wedge^2 V$ and $\mathrm{Sym}^2(V)$,  both
of which
have dimension $>1$ if $n \ge 3$.
The following is well known~\cite{Feit}.

\begin{lemma} Let $\rho: G_{\Q} \rightarrow \GL_3(E)$
be an absolutely irreducible Artin representation. Then the projective image of
$\rho$ is one of the following groups.
\begin{enumerate}
\item The simple alternating groups $A_5$ or $A_6$.
\item The simple linear group $\mathrm{PSL}_2(\F_7) = \mathrm{PSL}_3(\F_2)$
of order $168$.
\item The ``Hessian group''  $H_{216}$ of order $216$.
\item The normal subgroup $H_{72}$ of $H_{216}$  of order $72$.
\item The normal subgroup $H_{36}$ of $H_{72}$ of order $36$.
\item An imprimitive subgroup; i.e.,
the semidirect product of $(\Z/3\Z)$ and $S_3$ or $A_3$,
or a subgroup of this group.
\end{enumerate}
\label{lemma:sporadic2}
\end{lemma}

\begin{lemma} If $G$ lies in the list above, and
$\rho$
denotes a three dimensional representation of a central
extension $\tw{G}$ of $G$ with underlying
vector space $V$, then $\mathrm{ad}^0(V)$ is irreducible
except for the following cases:
\begin{enumerate}
\item $\rho$ is the symmetric square up to twist
of a two dimensional
representation.
\item $G$ is imprimitive.
\item $G = H_{36}$, in which case $\ad =
M \oplus M^{*}$
where $\dim(M) = 4$.
\end{enumerate}
\label{lemma:sporadic3}
\end{lemma}

From this we may deduce the following:

\begin{theorem} Assume the strong Leopoldt conjecture.
Suppose that $\rho: G_{\Q} \rightarrow \GL_3(E)$ is continuous, nearly
ordinary, and has absolutely irreducible finite image.
Suppose moreover that $\rho$ is unramified at $p$,
that $\rho$ is $p$-distinguished, and that
the projective image of
$\rho$ is a primitive subgroup of $\PGL_3(E)$.
Then if $\rho$ admits infinitesimally classical deformations, then
$\rho$ is, up to twist, the symmetric square of a two dimensional
representation.
\end{theorem}

\begin{proof} This follows from
Lemma~\ref{lemma:sporadic}, and  Lemmas~\ref{lemma:sporadic2}
and~\ref{lemma:sporadic3}, unless $G = H_{36}$.
If $G = H_{36}$, one may proceed on a case by case analysis ---
each case corresponding to conjugacy class of
$\Frob_p$ in $G$ --- and explicitly construct a suitable space $Y'$
as in Lemma~\ref{lemma:sporadic}  that respects
the decomposition $W = M \oplus M^{*}$ of $W$ into the
two $4$-dimensional representations of $G$. \end{proof}

\medskip

In general, as $n$ becomes large, there are many groups which
admit irreducible $n$-dimensional representations which are not
self dual up to twist, and yet their corresponding adjoint representations
(minus the identity) of dimension $n^2 - 1$ are not irreducible.
Thus to prove a general result for $n$-dimensional representations
would require an analog of our arguments for $2$-dimensional
representations over general fields\footnote{Of course,
Lemma~\ref{lemma:sporadic} may be applied to various sporadic
examples when $n > 3$, for example, if $V$ is either of the
$11$-dimensional representations of the Mathieu group
$M_{12}$ (which does occur as a Galois group over $\Q$ --- see
\cite{Matzat}), then $\mathrm{ad}^0(V)$ is the irreducible representation
of $M_{12}$ of dimension $120$. Thus we may apply Lemma~\ref{lemma:sporadic}
for $11$-dimensional representations $\rho$
with fixed field $K/\Q$ of $\rho$ satisfying the following:
$\Gal(K/\Q) \simeq M_{12}$,
$K/\Q$ unramified at $p$,  $\Gal(K/\Q) = M_{12}$, and
$\mathrm{Frob}_p$ has order $11$.}.

\subsection{Deformations of Characters: $\GL(1)_{/K}$}
\label{subsection:dimone}

In this section we explain why the analog of our conjectures
for characters (equivalently, automorphic forms for $\GL(1)_{/K}$)
can be deduced from a simple application of class field theory.
Namely, suppose that
$\rho: \Gal(\Kbar/K) \rightarrow E^{*}$ is the trivial representation,
or indeed, any finite character. Suppose that $K$ is Galois
and that $G:=\Gal(K/\Q)$.
Then the universal deformation ring $R$ of $\rho$ unramified away
from  $p$ (and with no local condition at primes dividing $p$) may be described 
explicitly in terms of class field theory. In particular, the tangent space to $R$ at 
$\rho$ is equal to
$\Hom_E(\Gamma_K,E)$, where $\Gamma_K$ is the maximal  $\Z_p$-extension
of $K$ unramified outside $p$.
Identifying $\Gamma_K \otimes E$ as the quotient of
$U_{\loc} \simeq \Lambda$ by (the image $I$ of) $U_{\glob}$ as in section~\ref{subsection:selmer2}, we may write this as
$$\Hom_E(\Lambda/I,E) = \Hom_G(\Lambda/I,\Ind^{G}_{\langle1 \rangle} E) 
= \Hom_G(\Lambda/I,\Lambda) = [I]\Lambda.$$
All one dimensional representations admit a trivial infinitesimal
deformation arising from twisting via the cyclotomic character. 
The following result is the analog of
Theorem~\ref{theorem:mainintro}
\begin{lemma} Assume that $p$
splits completely in $K$, and suppose that $K$ is Galois
with Galois group $G$.
If the trivial representation  admits an infinitesimal
deformation over $\Q$ which is not cyclotomic, then
this deformation descends to a subfield $H$
which is a CM field.
\label{lemma:weil}
\end{lemma}

\begin{proof}[Sketch] The rational infinitesimal weights correspond
to the natural identification of $\Lambda$ with $\Lambda_{\Q} \otimes E$.
The existence
of a rational infinitesimal deformation corresponds
exactly to an inclusion $\eta \in [I]\Lambda$, where the vector
space $\eta E$ is a rational subspace of $\Lambda$ --- equivalently,
where $\eta \in \Lambda_{\Q} = \Q[G]$ (up to scalar).
Without loss of generality we assume that $\eta \in \Z[G]$.  If
$\eta = \sum_{g \in G} g$ is the norm, then $\eta$ \emph{does} annihilate $I$
(on the left and right) and so lies in the  annihilator group
$[I]\Lambda$ --- such a choice corresponds to the
cyclotomic deformation. Thus we assume that $\eta$ is not a multiple of
$N_{K/\Q}$,
and in particular is not $G$-invariant.
If we assume the strong Leopoldt conjecture, then as
in section~\ref{subsection:propend}, there is an isomorphism
of (right) $G$-modules:
$$[I]\Lambda \simeq E \oplus \bigoplus_{\IR(G) \ne E} V^{\dim(V_i|c = -1)}_i,$$
where the second product runs over all non-trivial
irreducibles of $G$. 
If there was an equality
$\dim(V_i|c = -1) = \dim(V_i)$ for some (any) representation
$V_i$,  the action of $G$ on $V_i$
would factor (faithfully) through $\Gal(H_i/\Q)$ for some CM field $H_i$. 
Since, by assumption, $\eta \notin E$, the
strong Leopoldt conjecture implies (as in the argument of
section~\ref{subsection:propend}) that 
$\dim(V_i|c = -1) = \dim(V_i)$ for all representations $V_i$
in which $\eta$ projects non-trivially. It follows that
$\eta$ is divisible by $N_{K/H}$ (for some CM field $H$ with
is the compositum of the relevant CM fields $H_i$) and that the 
infinitesimal deformation
descends to $H$. Thus the conclusion of the lemma
follows from the strong Leopoldt conjecture.
To prove the lemma without this assumption, we use the following result.

\begin{sublemma} If $\eta \in \Z[G]$ annihilates $I$, then it
annihilates the global unit group $U_{\glob}$.
\label{sublemma:brum}
\end{sublemma}

\begin{proof} If the map $U_{\glob} \rightarrow I$ is
injective (the Leopoldt conjecture) the result is obvious.
Yet, in any circumstance, this map is $G$-equivariant, and thus
the image of $(U_{\glob}) \eta$ is $I \eta$. Since the map
from the group of units $\otimes \Q$ (that is, the image
of the map
$U_{\glob,\Q} \rightarrow I \subset U_{\loc}$) \emph{is} injective, 
we deduce that
 $I \eta = 0$ if and
only if $(U_{\glob}) \eta = 0$. \end{proof}

Since $\displaystyle{U_{\glob,\Q} \otimes \C  \simeq \bigoplus_{\IR(G) \ne E}
 V^{\dim(V_i|c = 1)}_i}$,
we see that if an element $\eta \in \Z[G]$ other than
the norm annihilates $U_{\glob}$, then $\dim(V_i|c = 1) = 0$
for some collection $V_i$. As above, we may deduce from this
that the deformation descends to a CM field $H$.
\end{proof}

This result could presumably be strengthened to include
infinitesimal
deformations over some number field $F$ rather than $\Q$. For this
we would have needed to show that
$U_{\glob}$ was not annihilated by an element
of $F[G]$, which would have required an appeal to
Baker's results on linear forms in logarithms, as in Brumer's
proof of Leopoldt's conjecture.

\medskip

In the case of one dimensional representations there is no
difficultly in establishing, by class field theory, an $R = \T$ theorem.
The analog of Lemma~\ref{lemma:weil} for automorphic representations
(namely, that automorphic representations are not dense in $\Spec(\T)$
for general $K$) follows from the fact that, up to finite characters,
the only non-trivial algebraic Hecke characters arise from CM fields
$H$. Indeed, an analysis of the infinity type shows that this question
exactly reduces to determining the annihilator of $U_{\glob,\Q}$ in
$\Z[G]$.

\subsection{Automorphic Forms for $\GL(2)_{/K}$}
\label{section:auto}

The goal of this subsection will be to sketch a proof
of Theorem~\ref{theorem:hidafamilies}.

\bigskip

The group $\GL(2)_{/K}$ is associated
to a PEL Shimura variety if and only if $K$ is totally real.
We assume that $K/\Q$ is an imaginary quadratic field in which $p$
splits completely. Since we shall ultimately take $K = \Q(\sqrt{-2})$,
we also assume that the class group of $K$ is trivial.
 Let $\OL$ be the localization of the ring of integers
$\OL_K$
of $K$ at some prime above $p$. Let $\Gamma_K$ be the torsion
free part of $\OL^{\times}$, and let
$\Lambda = \OL[[\Gamma_K]]$
be the Iwasawa algebra, which is abstractly isomorphic to
$\OL[[T_1,T_2]]$. Hida (page 29  of~\cite{Hidaimag}) constructs a finitely generated $\Lambda $-module $\Hid:= H_{\rm ord}^1(Y(\Phi),{\underline{\mathcal C}})^*$.  The right hand side of this equation is in the notation of \emph{loc. cit.};  $Y(\Phi)$ is, in effect, 
the arithmetic quotient of the 
symmetric space for $\GL_2$ over $K$ of an appropriate level, 
and the standard Hecke operators act on this cohomology group. 
Following Hida, let  $\mathbf{T}$ be the subring of endomorphisms
of $\mathrm{End}(\Hid)$ generated by Hecke operators.
Then $\mathbf{T}$ is finitely generated over $\Lambda$. Hida shows in Theorem~5.2 of~\cite{Hidaimag} that the support of  $\Hid$ is some equidimensional space
of codimension one in $\Spec(\Lambda)$.

Recall that by a ``classical point" of ${\mathrm{Spec}}({\bf{T}})$  we mean a ${\bf C}_p$-valued point corresponding to a  homomorphism $\eta: {\bf{T}} \to {\bf C}_p$  for which there exists a classical automorphic eigenform (cf. subsection~\ref{subsection:assumptions}) whose Hecke eigenvalues are given by the ${\bf C}_p$-character $\eta$.

Let $K = \Q(\sqrt{-2})$, and let
$\n = 3 - 2 \sqrt{-2}$. For the rest of this subsection we let  $\mathbf{T}$ denote the nearly
ordinary $3$-adic Hida algebra of tame level $\n$.

 The affine scheme ${\mathrm{Spec}}({\bf{T}}\otimes \Q)$ is nonempty and contains at least one classical point, i.e., the point corresponding to the classical automorphic form $f$ of weight $(2,2)$ of level $\Gamma_0(7 + \sqrt{-2})$ that is discussed in the following lemma, a proof of which will be given in the next subsection (Lemma~\ref{lemma:steinpollack}).

\begin{lemma}\label{twotwo} Let $K = \Q(\sqrt{-2})$,
let $\n = (3 - 2\sqrt{-2})$, and let $\p = (1 + \sqrt{-2})$.
There exists a unique cuspidal eigenform $f$ of
weight $(2,2)$ and level $\Gamma_0(\n \p)= \Gamma_0(7 + \sqrt{-2})$ which is ordinary
at $\p$ and $\overline{\p}$. There does not exist any
cuspidal eigenform (ordinary or otherwise) of level $\Gamma_0(\n)$
and weight $(4,4)$.
\end{lemma}

As 
mentioned above, Hida has shown that  the nonempty affine scheme ${\mathrm{Spec}}({\bf{T}})$
is  
of pure relative dimension one over $\mathrm{Spec}(\Z_p)$.

\begin{theorem} 
 The scheme $\mathrm{Spec}(\mathbf{T})$ has only finitely many classical
points.
\label{theorem:hidafamilies2}
\end{theorem}

\Remarks
\begin{enumerate}
\item
This answers in the negative a question raised in Richard Taylor's 
thesis (\cite{taylorthesis}, Remark, p.124).
\item  We expect, 
but have not proven, that the automorphic form $f$ 
referred to in Lemma~\ref{twotwo} corresponds to the elliptic curve 
$$E: y^2 + \sqrt{-2} x y + y = x^3 + (\sqrt{-2} - 1) x^2 - \sqrt{-2} x$$
over $\Q({\sqrt{-2}})$, which can be found in Cremona's 
tables~\cite[Table~3.3.3]{Cremona}.
Presumably one could prove the association between $f$ and $E$
using the methods of~\cite{taylor}.
\item We also expect that there is a two-dimensional  Galois representation over $\mathbf{T}$ whose Frobenius-traces are equal to corresponding Hecke operators  $\mathbf{T}$ (this correspondence being meant in the usual sense) but we have not constructed such a representation.  If there were such a Galois representation,
 one effect of the above theorem is to show that {\it none}  of the Galois representations over $\C_p$ obtained by specializing it  to $\C_p$ via homomorphisms $\mathbf{T} \to \C_p$ is ``limit-automorphic;" i.e., is the limit of a sequence of distinct classical modular Galois representations.  This is in contrast to what is expected (and is often known) to be true over $\Q$,
 where {\it every} ordinary two-dimensional $p$-adic Galois representation of integral weight that is residually modular should be ``limit-modular'',
 in the sense that it should be obtainable as the limit of a sequence of modular ordinary representations whose weights tend to infinity.
 \item 
In passing, we should note that it does not follow from the construction of
Taylor et. al.~\cite{taylorothers,taylor} that the Galois representation associated to an ordinary cuspidal modular form
over $K$ is ordinary in the Galois sense. However, in
many cases this is known by a theorem of Urban~\cite{Urban}.

\item
A result analogous to Theorem~\ref{theorem:hidafamilies2} is established in~\cite{ash} for
ordinary families on $\GL(3)_{/\Q}$. In this setting, however,
one does not have Hida's purity result on the dimension of $\Spec(\T)$,
and can only deduce the existence of a component
of $\mathrm{Spec}(\mathbf{T})$ with finitely many classical points --- \emph{a priori}
this component could have dimension zero.
(See, however,  work in progress
of Ash and Stevens in which this result is established,
essentially following Hida's methods.)
One of the main motivations of this paper
was to understand the examples of~\cite{ash} from the perspective
of Galois representations.
\end{enumerate}

\bigskip

To prove Theorem~\ref{theorem:hidafamilies2}, note that if
$\mathrm{Spec}(\mathbf{T})$ contained infinitely many
classical points, then since $\mathbf{T}$ is finitely generated
over $\La$ it would follow that the support of $\mathbf{T}$ would include
\emph{every} classical parallel weight of $\mathrm{Spec}(\Lambda)$.
In particular it would contain a classical point of weight $(4,4)$.
By 
Hida's control theorem~\cite[Thm 3.2]{Hidaimag},
the specialization of $\mathbf{T}$ at a point of
 weight $(4,4)$ would correspond to a classical automorphic 
form of level $\Gamma_0(3 - 2\sqrt{-2})$,
 which contradicts the statement of Lemma~\ref{twotwo}.

\subsection{Modular Symbols}

In order to compute explicit examples of
Hida families (or the lack thereof) one needs
a method of computing spaces of modular forms
for imaginary quadratic fields. This is
equivalent to computing group cohomology of
congruence subgroups of $\GL_2(\OL_K)$ with
coefficients in some local system. In the
case of weight $(2,2)$, the local system is trivial,
and the problem reduces to a homology computation
on the  locally symmetric space
$\G(K) \backslash \G(\mathbf{A}_K) \slash  U = \H/\Gamma$,
where $\G = \mathrm{Res}(\GL(2)_{/K})$, $U$ is the appropriate compact
open of $\GL_2(\mathbf{A}_K)$ containing the connected
component of a maximal compact subgroup of $\GL_2(\mathbf{R})$,
$\H$ is the hyperbolic upper half space, and
$\Gamma$ a finite index congruence subgroup of $\GL_2(\OL_K)$
(recall we are assuming that the class number of $K$ is $1$).
One technique for doing this is to use modular
symbols, and this was carried out by J.~Cremona~\cite{Cremona}
for certain fields $K$ of small discriminant and class number one.
The only higher weight computations previously carried
out are some
direct computations of group cohomology in level one
and weight $(2k,2)$ by C.~Priplata~\cite{Thes}; however, as
the only automorphic forms of this level (with $2k > 2$) are
Eisenstein, the result was torsion.
To compute in general weights $(k,k)$ one must work with
modular symbols of higher weight. The generalization of
modular symbols to higher weight is completely standard and is analogous
to the case of $\GL(2)_{/\Q}$.
For our specific computations, we have decided to work with
the field $K = \Q(\sqrt{-2})$. We choose this field
because it is the simplest field in which $3$ splits completely.

\medskip

Let $V$ be a two dimensional
vector space over $K$. There is  a natural action
of $\GL_2(\OL)$ on $V$, and choosing a basis
we may identify $V$ with degree one homogeneous
polynomials in $K[x,y]$. Let
$S_k = \Sym^k(V)$ and let $\overline{S}_k$ denote
$S_k$ where the action of $\GL_2(\OL)$ is twisted
by the automorphism of $\Gal(K/\Q)$. Modular forms
over $K$ of weight $(k,k)$ can be considered as sections of
 $\Hcb(\H/\Gamma,\Laa_{k,k})$, where $\Laa_{k,k}$ is
the local system associated to the representation
$S_{k-2} \otimes \overline{S}_{k-2}$.

\medskip

The following theorem is essentially due to
Cremona\footnote{There is a discrepancy between our presentation and the
presentation given in Cremona~\cite{Cremona}, because Cremona works
with the groups $\widetilde{\Gamma}_0(\n) = \SL_2(\OL) \cap \Gamma_0(\n)$
rather than $\Gamma_0(\n)$, which is more natural from an automorphic
point of view. Note that $\Gamma_0(\n)/\widetilde{\Gamma}_0(\n)$
has order two, and is generated by $J$. Thus  the characteristic
zero cohomology of
the orbifold $\H/\Gamma_0(\n)$ can be recovered from the
cohomology of $\H/\widetilde{\Gamma}_0(\n)$ by taking $J$-invariants, and
the space of forms we consider
is precisely Cremona's $+$ space.}~\cite{Cremona}, 
although one should note that previous limited computations of these spaces
were carried out in~\cite{Germans}.

\begin{theorem}
\label{theorem:cremona}
Let $K = \Q(\sqrt{-2})$, and let
$\n \subseteq \OL_K$ be an ideal.  Let
$M$ be the free $K$-module on the formal generators
$$\left\{[c,P] \ | \  c \in \mathbf{P}^1(\OL/\n),
P \in S_{k,k} \right\}.$$
There is an action of $\GL_2(\OL)$ on $M$ given by
$[c,P].h = [Ph,ch]$. Explicitly, if
$\displaystyle{\gamma =
\left(\begin{matrix} a & b \\ c & d \end{matrix}\right)}$, then
$$P(x,y,\xbar,\ybar) \gamma
=  P(dx-cy,-bx + ay,\sigma(d) \xbar - \sigma(c) \ybar,
- \sigma(b)\xbar + \sigma(a)  \ybar),$$
where $\Gal(K/\Q) = \langle \sigma \rangle$.
The action of $\GL_2(\OL)$ on
$\mathbf{P}^1(\OL/\n)$ arises from the identification of
this space
with coset representatives of $\Gamma_0(\n)$ in
$\GL_2(\OL)$.
Let $S$, $T$,  $J$ and $X$ denote the
following elements of $\GL_2(\OL)$:
$$S = \left( \begin{matrix} 0 & -1 \\ 1 & 0 \end{matrix} \right), \quad
T = \left( \begin{matrix} 1 & 1 \\ 0 & 1 \end{matrix} \right), \quad
J = \left( \begin{matrix} -1 & 0 \\ 0 & 1 \end{matrix} \right), \quad
X = \left( \begin{matrix} \sqrt{-2} & 1 \\ 1 & 0 \end{matrix} \right),$$
and let $I$ be the identity.
Let $M_{rel}$ be the subspace of $M$
generated by the subspaces:
$$(1+S) M, \quad  (1 + (ST) + (ST)^2) M, \quad (I - J)M,
\quad
(1+X+X^2+X^3)M$$
for $m \in M$.
Then $M/M_{rel} \simeq \Hcc(\H/\Gamma_0(\n),\Laa_{k,k})$.
\end{theorem}

\medskip

The authors would like to thank William Stein
and David Pollack, who computed the following
data using the result above on December 15, 2004 and
November 21, 2005 respectively; William Stein
writing a program to compute the dimensions of these
spaces, and David Pollack writing an
additional routine to compute Hecke operators.
The authors
take full responsibility for the correctness of these results.
As some independent confirmation of these computations, we note that
part 1 is already known by previous computations of
Cremona~\cite{Cremona}, and in  part
3 the dimensions we compute agree with a known lower bound
that is independent of the computation. Finally, for part two,
we list the first few eigenvalues and note that they lie in
$\Q$ (the space of cusp forms is one dimensional over $\Q$) and satisfy
the (conjectural) Ramanujan bound $|a_{\p}| \le 2 N(\p)^{3/2}$.

\begin{lemma} Let $K = \Q(\sqrt{-2})$, let $\OL = \OL_K$,
and let $\n \subset \OL$ be an ideal such that
$N_{K/\Q}(\n)$ is squarefree and co-prime to $2$.
Then:
\begin{enumerate}
\item The space of cuspidal modular forms over $K$
of weight $(2,2)$ and level $\Gamma_0(\n)$ is trivial
for all such $\n$ with $N_{K/\Q}(\n) \le 337$ except
for $\n$ in the following list:
$$\left\{7+\sqrt{-2},11 + 7 \sqrt{-2},7 + 10 \sqrt{-2},13 + 7 \sqrt{-2},
5 + 11 \sqrt{-2},9 + 11 \sqrt{-2}, 7 + 12 \sqrt{-2}\right\}$$
of norms $51, 219, 249, 267, 323$ and $337$ respectively.
\item The space of cuspidal modular forms over $K$
of weight $(4,4)$ and level $\Gamma_0(\n)$ is trivial
for all such $\n$ with $N_{K/\Q}(\n) \le 337$ except
for $\n = 5 + 7 \sqrt{-2}$ of norm $123$. There is a unique cuspform $g$ up to normalization of weight $(4,4)$ and level  $\Gamma_0(5 + 7 \sqrt{-2})$.
The Hecke eigenvalues $a_{\p}$ for
$\p \nmid \n$ and $N(\p) \le 25$  of $g$
are given by the following table:
\begin{center}
\begin{tabular}{|c|c|c|c|c|c|c|c|}
\hline
$\p$     &  $\sqrt{-2}$ & $1 - \sqrt{-2}$ &
$3 + \sqrt{-2}$ & $3 - \sqrt{-2}$  & $3 + 2 \sqrt{-2}$ &
$3 - 2 \sqrt{-2}$ & $5$ \\
\hline
$a_{\p}$ &  $0$ & $4$ &  $-24$ & $36$ & $-54$ &  $-102$ & $-118$  \\
\hline
$\lceil 2  N(\p)^{3/2} \rceil$ & $6$ &  $11$ &
$73$ & $73$ & $141$ & $141$ & $250$  \\
\hline
\end{tabular}
\end{center}
\item The space of cuspidal modular forms over $K$
of weight $(2k,2k)$ and level $\Gamma_0(1)$ is equal
to the space of forms of level one arising from base
change from $\Q$ for all $2k \le 96$.
\end{enumerate}
\label{lemma:steinpollack}
\end{lemma}

Referring to the above table, note that the congruences
$a_{\p} \equiv 1 + N(\p)^3 \mod 4$ for $\p \nmid 2$ 
and $a_{\p} \equiv 1 + N(\p)^3 \mod 3$ appear to hold.
(Since $K$ is ramified at $2$, being Eisenstein at $2$ does
not imply ordinariness). It might be interesting to determine the
$\GL_2(\F_5)$-representation associated to $g$, although
it could \emph{a priori} be quite large.

\appendix

\section{Generic Modules}\label{app}
\label{section:appendix}

\subsection{Left and Right $G$-Modules}

Let $G$ be a finite group, and $D \subset G$ a subgroup. For ${\mathcal K}$ a field, denote the group ring $G$ with coefficients in ${\mathcal K}$ by $\La_{\mathcal K}:= {\mathcal K}[G]$.  We will view $\La_{\mathcal K}$ as a ${\mathcal K}$-algebra, and therefore also as bi-module over itself.
   Let $F$
denote a field of characteristic zero
over which all irreducible
representations of $G$ over an algebraically closed field are defined,
and let $E$ be some extension of $F$. In our application, $F$ will
be a field algebraic over $\Q$ and $E$ a localization of $F$ at some prime above $p$.
Let $\La := \La _E =E[G]$ so that we have the natural isomorphism of $E$-algebras $\La \simeq \La_F \otimes_F E$.
A  {\bf left ideal} $M$ of $\La$ is a
 left $\La$-module equipped with an  inclusion $M \hookrightarrow \La$
of left $\La$-modules, and a {\bf right ideal}  $N$ of $\La$ is defined similarly.

\begin{definition} For any subset $M$ of $\La$, let $\La[M]$
denote the set of elements of $\La$ that are annihilated by $M$ on the right; i.e., $$\La[M]:=\{\lambda \in \La\ | \ m\lambda = 0\ {\rm for\ all\ } m \in M\}$$ and let $[M]\La$
denote the set of elements of $\La$ that are annihilated by $M$ on the left; i.e., $$[M]\La:=\{\lambda \in \La\ | \ \lambda m = 0\ {\rm for\ all\ } m \in M\}.$$

\end{definition}

\begin{lemma} 
\

\begin{itemize} \item If $M$ is a right ideal of $\La$ then  $\La[M]$ is a left ideal.
 \item  If $M$ is a left ideal then  $[M]\La$ is a right ideal.
 \end{itemize}
\end{lemma}

For any left $\La$-module $W$,
the set $\Hom_{G}(\La,W)$ inherits the structure of
a left $\La$-module  where scalar multiplication  $(\lambda, \phi) \mapsto \lambda\cdot \phi$ for $\lambda \in \La$ and $ \phi \in \Hom_{G}(\La,W)$  is given by the  rule
$(\lambda\cdot \phi)(x) = \phi(x\cdot \lambda)$. The left $\La$-module $\Hom_{G}(\La,W)$ is then identified
with $W$ via the map $i$ sending $\phi$ to $i(\phi): = \phi([1])$.
In particular, taking $W = \La$ the isomorphism  $i$ establishes the canonical identification  $\phi \mapsto \phi([1])$  of the left $\La$-module $\Hom_{G}(\La,\La)$ 
 with $\La$.  Here, $\Hom_{G}(\La,\La)$ refers to the set of homomophisms $\phi:\La \to \La$ that preserve the left-$\La$-module structure of $\La$; explicitly $\phi(\lambda\cdot x) = \lambda\cdot \phi(x)$ for all $\lambda, x \in \La$.

\begin{lemma} Let $M$ be a left ideal of $\La$.  The isomorphism $i$ induces a commutative diagram

\bigskip

$$\begin{diagram}
\Hom_{G}(\La/M,\La) & \rTo & [M]\La\\
\dTo &  & \dTo \\
\Hom_{G}(\La,\La)& \rTo &  \La \\
\end{diagram}$$

\bigskip

where the vertical homomorphisms are the natural inclusions.
\label{lemma:dual}
\end{lemma}

\begin{proof} The homomorphism $\phi:\La \to \La $ lies in $\Hom_{G}(\La/M,\La)$ if and only
if $\phi(m) = 0$ or equivalently if
$(m\cdot \phi)([1]) = 0$, or $m\cdot i(\phi) = 0$,  for all $m \in M$.
\end{proof}

\begin{lemma}  The mapping  $$M \ \ \longmapsto \ \ \  N:= [M]\La$$ from left ideals $M$ to right ideals $N$, and the mapping
$$N \ \ \longmapsto \ \ \ M:=\La[N]$$ from right ideals $N$ to left ideals $M$,
are two-sided inverses of each other, and are one:one correspondences between the set of left ideals of $\La$ and the set of right ideals of $\La$.
\label{lemma:rightmodules}
\end{lemma}

\begin{proof} Let $M$ be a left ideal and $ N:= [M]\La$.  It follows immediately from the definitions that $M \subseteq \La[N]$, and thus to show that
$M = \La[N]$ we need only show $\dim(M) \ge \dim(\La[N])$.   From the representation theory of finite groups one
sees that
$$\dim(M) + \dim(\Hom_{G}(\La/M,\La)) = \dim(\La),$$ and using Lemma~\ref{lemma:dual} we have
$$\dim(M) + \dim([M]\La) = \dim(\La).$$
Similarly, using (only) that $N$ is a right ideal, we have
$$\dim(N) + \dim(\La[N]) = \dim(\La).$$

Since $N= [M]\La$ the lemma follows. 
\end{proof}

 We have the natural $E$-linear functional $\iota: \La \to E$ that associates to any element $\sum_{g\in G}a_g[g] \in \La$  the coefficient of the identity element, $a_1 \in E$; i.e., $\iota(\sum_{g\in G}a_g[g]) = a_1$.   For any left $\La$-module $W$, composition with this functional $\iota$ induces a homomorphism $$\Hom_{\La}(W, \La) \subset \Hom_E(W, \La){\stackrel{\iota}{\longrightarrow}}  \Hom_E(W, E) = W^{\vee}.$$  The flanking $E$-vector spaces, $\Hom_{\La}(W, \La) \to W^{\vee},$ both have natural right $\La$-module structures defined as follows. On $\Hom_{\La}(W, \La)$  the right scalar multiplication $(\phi, \lambda) \mapsto  \phi \cdot \lambda$  is given by the rule $(\phi \cdot \lambda)(x) = \phi(\lambda \cdot x)$, or equivalently, by letting $g \in G$ act on $\phi:W\to\La$ by composition with $g^{-1}: \La \to \La$ and extending this action linearly to obtain a right  $\La$-module structure.  We define the right $\La$-structure on $W^{\vee}$ by the corresponding rule: for $\phi \in \Hom_E(W, E)$  and $g \in G$ define $\phi\cdot g$ by  $(\phi \cdot g)(w) = \phi(g \cdot w)$ for $w \in W$.

 Consider, as well, the homomorphism 
 $ W^{\vee} = \Hom_E(W, E) {\stackrel{j}{\longrightarrow}}\Hom_E(W, \La)$ that associates to $\phi \in \Hom_E(W, E)$ the homomorphism $j\phi \in \Hom_\La(W, \La)$ which, for $w \in W$ satisfies the formula $$j\phi(w) = \sum_{g\in G}\phi(g\cdot w)[g]^{-1} \in \La.$$ 
 \begin{lemma}\label{iota} $\Hom_{\La}(W, \La) {\stackrel{\iota}{\longrightarrow}} W^{\vee}$ is an isomorphism of right $\La$-modules and the homomorphism $j$ described above is its inverse.
\end{lemma}

\begin{proof}
  This is a direct check.
\end{proof}

In the special case where we take $W$ to be $\La$ itself, viewed as $\La$-bimodule via left and right multiplication, the above identifications offer us isomorphisms of bi-modules, $${\rm Maps}(G,E) \stackrel{\simeq}{\longrightarrow}\Hom_E(\La,E) \stackrel{\simeq}{\longrightarrow} \La  \stackrel{\simeq}{\longrightarrow} \Hom_{G}(\La,\La),$$ the composition of the first two of these isomorphisms being given by the rule $\phi \mapsto \sum_{g\in G}\phi(g)\cdot[g^{-1}]$ for $\phi \in {\rm Maps}(G,E)$, and the last by the rule $\lambda \mapsto \{x\mapsto \lambda\cdot x\}$.  This latter isomorphism is indeed an isomorphism of {\it bi}-modules if we impose a left $\La$-module structure on $\Hom_{G}(\La,\La)$ in the usual way---i.e., $(\lambda\cdot \phi)(x) = \lambda\cdot \phi(x)$---and if we give $\Hom_{G}(\La,\La)$ a right $\La$-module structure by requiring $(\phi\cdot g)(x)= \phi(x\cdot g^{-1})$ for all $g \in G$ and extending linearly.

A consequence of Lemma~\ref{iota} is that for any left ideal $M \subset \La$ we have a natural  identification of right $\La$-modules:
$$(\La/M)^{\vee}\ \ \  {\stackrel{\iota}{\simeq}}\ \ \  \Hom_{\La}(\La/M, \La)\ \ \  =\ \ \  [M]\La.$$

If $V$ is a left $D$-module for the subgroup $D \subset G$, then
$V^{\vee}:=\Hom_{E}(V,E)$ inherits the structure of
a right $D$-module (via the usual formula $\phi\cdot g(v) = \phi(g\cdot v)$)
and if $V$ is a right $D$-module then
$V^{\vee}$ is naturally a left $D$-module.

\subsection{Generic Ideals}

The module $\La$ has a canonical $F$-rational structure
as a $\La$-bi-module coming from the identification
$\La = \La_{F} \otimes_F E$. 
Given a left or right ideal  $M$ of $\La$, one can ask the
extent to which $M$ sits ``perpendicularly'' inside
$\La$ with respect to this rational structure.
If $M$ is maximally skew, then
we shall define $M$ to be generic. More generally, we shall
define such a notion relative to the  subgroup $D \subset G$.

\medskip

\begin{definition}
An arbitrary $E$-vector subspace $V \subseteq \La$
is called a \emph{right rational $D$-subspace} if there exists
a right $F[D]$-submodule
$V_{F} \subseteq \La_{F}$ such that $V = V_{F} \otimes E$. One defines similarly the notion of \emph{left rational $D$-subspace}.
\end{definition}

\begin{lemma} If $V$ is a right rational $D$-subspace,
then $(\La/V)^{\vee} = \Hom_{E}(\La/V,E)$ is a left rational
$D$-subspace.
\end{lemma}

\begin{proof} It suffices to note that if $(\La/V)^{\vee}_{F}:=
\Hom_{F}(\La_{F}/V_{F},F)$ then
$(\La/V)^{\vee}_{F}$ is a left $D$-module with a natural
inclusion into $\La_{F}$, and
 $(\La/V)^{\vee} = (\La/V)^{\vee}_{F} \otimes E$.
\end{proof}

\begin{definition} A left ideal $M \subseteq \La$ is
\emph{right generic} with respect to $D$ if
$$\dim(M \cap V) \le \dim(N \cap V)$$
for all right rational $D$-subspaces $V$ and left ideals
$N \subseteq \La$ such that $M \simeq N$ as $G$-modules.
If $M \subseteq \La$ is a right ideal, then $M$ is left generic with
respect to $D$ if the same formula holds for all left rational
$D$-subspaces $V$ and right ideals $N \simeq M$.
\label{definition:definitionofgeneric}
\end{definition}
\begin{example} Let $G$ be an abelian group, and $D$
any subgroup. Then
every irreducible representation of $G$ occurs with
multiplicity one inside $\La$. In particular,
 if $M$ and $N$
are any two $\La$-submodules with $M \simeq N$ then
$M = N$, and so all such submodules are generic.
\end{example}

\begin{lemma} The left ideal $M$ is right generic if and only
if the right ideal $[M]\La$ is left generic.
\label{lemma:dual2}
\end{lemma}

\begin{proof} 
For any $M \subseteq \Lambda$ there is an exact sequence and
corresponding dimension formula:
$$0 \rightarrow M \cap V \rightarrow M \oplus V
\rightarrow \La \rightarrow \left((\La/M)^{\vee} \cap (\La/V)^{\vee}\right)^{\vee}
\rightarrow 0,$$
$$\dim((\La/M)^{\vee} \cap (\La/V)^{\vee}) = \dim(M \cap V)
-\dim(M) -\dim(V) +\dim(\La).$$
From this formula (and its dual)  it follows that $M$ is generic if and only
if $(\Lambda/M)^{\vee}$ is generic, since the dimensions of the
intersections of $(\La/M)^{\vee}$ with left rational $D$-modules can
be explicitly compared to the intersection of $M$ with right rational
$D$-modules.
Yet $(\La/M)^{\vee} = [M]\La$,
so we are done.
\end{proof}

\subsection{First properties of generic modules}
\label{section:firstproperties}

Genericity is not in general preserved under automorphisms
of $\La$. On the other hand, we have the following:
\begin{lemma} Let $S: \La \rightarrow \La$ be an isomorphism of
left $\La$-modules sending $E[D]$ to itself.
Then $S(M)$ is right  generic with respect to $D$ if and only if $M$
is right generic with respect to $D$.
\label{lemma:indep}
\end{lemma}
\begin{proof}  Let $V$ be a right rational $D$-subspace.
Then $V = V.E[D]$. Thus the image of $V$ under $S$ is
$V.E[D] = V$. 
Hence $S$ preserves
right rational $D$-subspaces, and the genericity of $M$ follows from
the following obvious formula:
$$\dim(S(M) \cap V) =  \dim(S(M \cap V)) = \dim(M \cap V).$$
\end{proof}

\medskip

We now extend our notion of generic to include certain
submodules
$M \subseteq P$, where $P$ is a free $\La$-module of rank one
which does not necessarily come with a canonical generator.
Let $A$ be an ``abstract" free left $E[D]$-module of rank one,
without a canonical generator. Let $P = \La \otimes_D A$.
Clearly $P$ is a free left $\La$-module of rank one. Moreover,
there is a natural class of isomorphisms from $\La$ to $P$
given by maps of the form
$$T: \La \rightarrow P, \qquad [1] \mapsto [1] \otimes a$$
for some generator $a$ of $A$.
\begin{definition}\label{rightgenericnoncanonical}
A left module $M \subseteq P$ is right generic with respect to $D$
if for some generator $a\in A$ of the $E[D]$-module $A$, the left ideal
$T^{-1}(M) \subseteq \La$ is right generic with respect to
$D$.
\end{definition}
\begin{lemma} If $M \subseteq P$ is right generic with respect to $D$,
then $T^{-1}(M) \subseteq \La$ is right generic with respect
to $D$ for \emph{all} $a$ generating $A$.
\end{lemma}
\begin{proof} Given two such choices of generator
$a$, $a'$, it suffices to note 
that the composite: $S: T(a')^{-1} \circ T(a)$ preserves
$E[D]$ and thus preserves generic ideals, by  Lemma~\ref{lemma:indep}.
\end{proof}

\subsection{Relative Homomorphism Groups}
\label{section:ext}

Let $Y$ be a left $\La$-module, and let $Z \subseteq Y$ be
a vector subspace.

\begin{definition} Let $\Hom_G(\Lambda,Y;Z)$ denote the $G$-equivariant
homomorphisms $\phi$ from $\Lambda$ to $Y$ such that
$\phi([1]) \in Z$. 
For a left ideal $M \subseteq \Lambda$, let
$\Hom_G(\Lambda/M,Y;Z)$ denote the $G$-equivariant
homomorphisms $\phi$ such that $\phi([1]\kern-0.3em{\mod M})  \in Z$.
\end{definition}

\Remark  The identification of $\Hom_G(\La,Y)$ with $Y$
identifies $\Hom_G(\La,Y;Z)$ with $Z$.

\medskip

Suppose that $Y$ and $Z$ have compatible rational structures,
namely, there exists a left $\La_F$-module $Y_F$, an $F$-vector subspace
$Z_F \subseteq Y_F$ and isomorphisms $Y = Y_F \otimes_F E$, $Z = Z_F \otimes_F E$
compatible with the inclusion $Z \subseteq Y$. Then if $M$ is a generic left
ideal one may expect the homomorphism group 
$\Hom_G(\Lambda/M,Y;Z)$ to be `as small as possible'. This expectation
is borne out by the following result, which is the main form in which
we apply our generic hypothesis.

\begin{lemma} Let the left ideal $M \subseteq \La$ be right generic 
with respect to $D$. Suppose that $Y$ is a left $\La$-module and
$Z \subseteq Y$ a $D$-module subspace, and that $Y$ and $Z$ admit
compatible rational structures. Suppose furthermore that $Y$
admits a $\La$-module injection $Y \rightarrow \La$. Then for all
left ideals $N \subseteq \La$ with $M \simeq N$ as $G$-modules,
$$\dim(\Hom_G(\Lambda/M,Y;Z)) \le \dim(\Hom_G(\Lambda/N,Y;Z)).$$
\label{lemma:maingeneric}
\end{lemma}

\begin{proof} Choose an injection $Y \hookrightarrow \La$ compatible
with rational structures. Such a map makes $Z$ a rational
left $D$-submodule of $\La$. On the other hand one has the following
identification:
$$\Hom_G(\Lambda/M,Y;Z)) = \Hom_G(\Lambda/M,Y) \cap \Hom_G(\Lambda,Y;Z) 
= \Hom_G(\Lambda/M,\Lambda) \cap
\Hom_G(\Lambda,Y;Z).$$
By Lemma~\ref{lemma:dual} we may write this as
$[M]\Lambda  \cap Z$.
By Lemma~\ref{lemma:dual2}, $[M]\Lambda$ is left generic
with respect to $D$. Thus, as $Z$ is a left rational $D$-module,
$$
\dim(\Hom_G(\Lambda/M,Y;Z)) =
\dim([M]\Lambda \cap Z) \le \dim([N] \Lambda \cap Z)
=  \dim(\Hom_G(\Lambda/N,Y;Z)),
$$
for any left $E[G]$-ideal  $N$ such that there exists a left $E[G]$-module isomorphism $N \simeq M$.
\end{proof}

In order to apply this lemma, we shall describe exactly what
vector subspaces of $Y$ are of the form $\Hom_G(\Lambda/N,Y)$ for some  left $E[G]$-ideal $N\subset \La$. If $\IR(G)$ denotes the set of irreducible representations of $G$, for each  $i \in \IR(G)$,
fix a choice $V_{i,F}$  of a left $F[G]$-module that, as $G$-representation, corresponds to the irreducible representation $i$.  Let $V_{i,F}^*:= \Hom_F(V_{i,F},F)$.   Put $V_i:= V_{i,F}\otimes_FE$ and $V_i^*: = V_{i,F}^*\otimes_FE \simeq \Hom_E( V_{i}^*,E)$.

 There is  a canonical decomposition
$$\Lambda = \bigoplus_{i \in \IR(G)}  V_i \otimes {V_i}^{*}.$$
All (left) $E[G]$-modules $Y$ can be written as
$\displaystyle{Y = \bigoplus_{i \in {\IR}(G)} V_i \otimes {T_i}^*}$,
for some vector space ${T_i}^*$ with
$\dim({T_i}^*) = \dim(\Hom(Y,V_i))$.
Requiring that $Y$ admits an injective $E[G]$-module homomorphism into $\La$ is equivalent to
insisting that
$\dim({T_i}^*) \le \dim({V_i}^*)$. Any left ideal $N$  contained in $E[G]$ is then expressible as $$N = \bigoplus_{i \in {\IR}(G)} V_i \otimes {T_i}^* \subset
 \bigoplus_{i \in {\IR}(G)} V_i \otimes {V_i}^*$$ where the $T_i^*$ are $E$-vector subspaces  $T_i^* \subset {V_i}^*.$
Similarly, any right $E[G]$-module  is of the form
$\displaystyle{ \bigoplus_{i \in {\IR}(G)} S_i \otimes {V_i}^{*}}$.

\medskip

If the left $E[G]$-module $Y$ has a rational structure, an $E[G]$-module homomorphism  $Y \rightarrow \La$ is
compatible with this structure if and only if the corresponding $E$-linear homomorphism
${T_i}^* \rightarrow {V_i}^*$ is compatible with the rational structures on these $E$-vector spaces.

Also, we may express the right $E[G]$-submodule $[N]\Lambda\subset \La$ 
as  $$[N]\Lambda =  \bigoplus S_i \otimes {V_i}^{*}.$$ Note that the isomorphism class of $N$ as $E[G]$-module is completely determined by
(and determines)  the numbers $\dim(S_i)$ (for $i \in \IR(G)$).
Now consider the natural chain of inclusions of $E$-vector spaces  $$\Hom_G(\Lambda/N,Y)\subset \Hom_G(\Lambda,Y)=Y \subset \La.$$  Viewing, then, both  $E$-vector spaces $[N]\Lambda$ and $\Hom_G(\Lambda/N,Y)$ as subspaces of $\La$, we have the formula 
 $$\Hom_G(\Lambda/N,Y)= Y \cap [N]\La  \subset  \La $$  giving us 
$$\Hom_G(\Lambda/N,Y)=  \left(\bigoplus V_i \otimes {T_i}^*\right)\cap \left(\bigoplus S_i \otimes {V_i}^*\right)= \bigoplus S_i \otimes {T_i}^* \subseteq \La.$$

\medskip

The content of Lemma~\ref{lemma:maingeneric} is that if $N$ is generic the subspaces
$S_i \subseteq V_i$ are ``maximally skew'' to any rational
$D$-submodule $Z$ of $\La$, in the sense that
$$\dim\left(Z \cap \bigoplus S_i \otimes {T_i}^* \right) 
\le \dim\left(Z \cap \bigoplus S'_i \otimes {T_i}^* \right),$$
for any $S'_i \subseteq V_i$ with $\dim(S_i) = \dim(S'_i)$.

\medskip

If $Y' \subseteq Y$ is any $E$-vector subspace, define its $
\La$-annihilator in the evident 
way: 
$$N(Y'):= \{\lambda \in \La \ | \ \lambda\cdot y' = 0\ {\rm for\ all \ } y' \in Y'\}.$$  
So, $N(Y') \subset \La$ is a left ideal. Putting $N := N(Y')$ we have the 
inclusions $$Y' \subset Y\cap[N]\La = \Hom_G(\La/N, Y)$$ and the following 
lemma  provides a characterization of which $E$-vector subspaces $Y'$ have the 
property that this inclusion is an isomorphism, i.e., it offers us a 
characterization of the $E$-vector subspaces $Y'$ that
 are of the form $\Hom_{\La}(\Lambda/N,Y)  = Y\cap[N]\La  \subset Y$, 
for some left ideal $N \subset \La$.
  
  Let $Y_i$ denote the $V_i$-isotypic component of $Y$.
 If $Y' \subseteq Y$ is an $E$-vector subspace and $i\in {\IR}(G)$,
 we denote by $Y_i'$ the intersection of $Y'$ with the $i$-isotypic 
component of the left $E[G]$-module $\La$, or equivalently, $Y'_i=Y'\cap Y_i$. 
Any irreducible sub-$E[G]$-module 
in $Y = \bigoplus_{j \in {\IR}(G)} V_j \otimes {T_j}^*$ that is 
isomorphic to $V_i$ is of the 
form  $V_i\otimes {\mathcal L}_i$ for ${\mathcal L}_i \subset {T_j}^*$
for some one-dimensional $E$-subvector space of ${T_j}^*$.   

\begin{lemma}
 Let $Y' \subseteq Y$ be an $E$-vector subspace.  The following bulleted statements are equivalent.
 \begin{itemize}
 \item  For the left ideal $N=N(Y')\subset \La$ the inclusion  
$Y'\subseteq\Hom_G(\Lambda/N,Y)$ described above is an isomorphism.
\item  There are $E$-vector subspaces $S_i \subset V_i$ such that the diagram

$$\begin{diagram}
Y'& \ \ \rEquals \ \ & \oplus_i S_i \otimes {T_i}^* \\
\dTo &  & \dTo \\
Y & \ \ \rEquals \ \ & \oplus_i V_i \otimes {T_i}^* \\
\end{diagram}$$
is commutative, where the vertical morphisms are the natural inclusions.
\item

\begin{enumerate} 

\item $Y'$ is generated by the subspaces $Y' \cap V$ as $V$ ranges over
the irreducible $E[G]$-submodules $V \subseteq Y$.

\item If $V,V' \subseteq Y$ are two irreducible $E[G]$-submodules of $Y$,
and $\mathrm{Isom}(V,V') \ne \emptyset$, then any isomorphism
$V \rightarrow V'$ induces an isomorphism $Y' \cap V \rightarrow Y' \cap V'$.

\end{enumerate}
\end{itemize}
Furthermore, if $Y'$ satisfies one, and hence all, of these bullets, then
$$\dim(S_i) = \dim(\Hom_G(\Lambda/N,V_i)) = \dim(V_i) - \dim(\Hom_G(N,V_i))$$
for any $V_i$ with $\Hom_G(V_i,Y) \ne 0$.
\label{lemma:type}
\end{lemma}

\begin{proof}  First note that for any choice of $E$-vector subspaces $S_i \subset V_i$  (all $i$) there is a unique left ideal $N\subset \La$  such that $$[N]\La = \bigoplus_iS_i\otimes V_i^* \subset \bigoplus_iV_i\otimes V_i^* = \La$$ which (by the discussion just prior to the statement of our lemma) establishes the equivalence of the first two bullets.   We now propose to show the equivalence of the last two bullets.
  
  \begin{enumerate}
\item Suppose
that 
$\displaystyle{Y' = \bigoplus_i S_i \otimes {T_i}^*  \subseteq Y}$.

By Schur's lemma, any left $E[G]$-automorphism of $Y_i=V_i \otimes {T_i}^*$ is induced from a bilinear
isomorphism
 $V_i \times {T_i}^* \rightarrow V_i \times {T_i}^*$ that
is \emph{the identity} on the first factor.

\medskip

 Then, as we have noted, any irreducible $E[G]$-submodule $V \subseteq Y$
is of the form $V_i \otimes {\mathcal L}_i$ for some one
dimensional subspace ${\mathcal L}_i \subset {T_i}^*$. 
Since
$$ Y'\cap(V_i\otimes {\mathcal L}_i) = S_i \otimes {\mathcal L}_i,$$ 
such spaces clearly generate
$Y'$.
Thus $Y'$ satisfies the conditions of part $(1)$.

\medskip

For part $(2)$,  we may write $V =  V_i\otimes {\mathcal L}_i$
and $V' = V_i\otimes {\mathcal L}_i'$ for one dimensional
subspaces ${\mathcal L}_i, {\mathcal L}_i' \subset  {T_i}^*$.
Any  left $E[G]$-module isomorphism $V_i\otimes {\mathcal L}_i \rightarrow
V_i\otimes {\mathcal L}_i'$ arises from
a bilinear map $V_i \times {T_i}^* \rightarrow V_i \times {T_i}^*$
which is the identity on $V_i$ and sends $ {\mathcal L}_i$ to $ {\mathcal L}_i'$.
 All such maps
induce isomorphisms 
$$S_i\otimes {\mathcal L}_i= Y'\cap(V_i\otimes {\mathcal L}_i)
\  {\stackrel{\simeq}{\longrightarrow}} \ 
Y'\cap(V_i\otimes {\mathcal L}_i')=S_i\otimes {\mathcal L}_i'.$$

\medskip
\item
Suppose that $Y'$ is a module satisfying conditions $(1)$ and $(2)$.
Choose any one dimensional subspace ${\mathcal L}_i \subset {T_i}^*$, let
$V = V_i \otimes {\mathcal L}_i \subset Y$, and let
$$Y' \cap V = 
Y' \cap (V_i \otimes {\mathcal L}_i) =: S_i \otimes {\mathcal L}_i.$$
  If ${\mathcal L}_i' \subset {T_i}^*$ is any other one dimensional subspace, then
there is a left $E[G]$-isomorphism $V_i \otimes {T_i}^* \rightarrow
V_i \otimes {T_i}^*$ acting as the identity on the first factor and sending
$V = V_i \otimes{\mathcal L}_i$ to $V':= V_i \otimes{\mathcal L}_i'$. Hence by condition $(2)$ we deduce that
$Y' \cap V' = 
Y' \cap (V_i \otimes {\mathcal L}_i') \simeq S_i \otimes {\mathcal L}_i'$, and letting ${\mathcal L}_i$ range over all one dimensional
subspaces of ${T_i}^*$ we conclude that $S_i \otimes {T_i}^* \subset Y_i'$,
where $Y_i'$ is the $V_i$-isotypic component of $Y'$.
To deduce that this is an equality, recall by condition $(1)$ that $Y$
is generated by subspaces of the form $Y' \cap V$, and hence
$Y'_i$ is generated by subspaces of the form $Y' \cap V$ 
with $V \simeq V_i$. The above argument shows that all such
intersections
 are all of the form $S_i \otimes {\mathcal L}_i$ 
for some ${\mathcal L}_i \subset {T_i}^*$, and hence
$S_i \otimes {T_i}^* = Y_i'$, and we are done.
\end{enumerate}
\medskip

For the statements regarding dimensions of isotypic components,
suppose that $\Hom_G(V_i,Y) \ne 0$. Then ${T_i}^* \ne 0$,
and thus $S_i$ is determined uniquely from
$S_i \otimes {T_i}^*$.
Since $S_i$ is identified with the $V_i$-isotypic component of
$[N]\Lambda$, the dimension formulas follow.
\end{proof}

\end{document}